\documentclass[reqno, 12pt]{amsart}
\usepackage{float}
\usepackage{amsmath}
\usepackage[pdftex]{graphicx}
\usepackage{latexsym}
\usepackage{amsfonts}
\usepackage{amssymb}
\usepackage{bbm}
\usepackage{color} 
\usepackage[shortlabels]{enumitem}
\usepackage{hyperref}

\usepackage{lineno}

\theoremstyle{plain}
\newtheorem{theorem}{Theorem}
\newtheorem{corollary}[theorem]{Corollary}
\newtheorem{lemma}[theorem]{Lemma}
\newtheorem{proposition}[theorem]{Proposition}
\theoremstyle{definition}

\newtheorem{remark}[theorem]{Remark}
\newtheorem*{remark*}{Remark}

\newcommand{\R}{\mathbb R}

\newcommand{\pr}{\mathbb P}
\newcommand{\e}{\mathbb E}
\newcommand{\ind}{\mathbbm 1}

\newcommand{\N}{\mathbb N}

\newcommand{\C}{\mathbb C}
\newcommand{\E}{\mathbb E}

\newcounter{fig}
\newcommand{\f}{\refstepcounter{fig} Fig.\ \arabic{fig}. }

  \evensidemargin
\oddsidemargin
\DeclareMathOperator*{\essinf}{ess \, inf}
\DeclareMathOperator*{\intr}{int}
\DeclareMathOperator*{\cl}{cl}
\DeclareMathOperator*{\V}{Var}

\newcommand*\patchAmsMathEnvironmentForLineno[1]{%
  \expandafter\let\csname old#1\expandafter\endcsname\csname #1\endcsname
  \expandafter\let\csname oldend#1\expandafter\endcsname\csname end#1\endcsname
  \renewenvironment{#1}%
     {\linenomath\csname old#1\endcsname}%
     {\csname oldend#1\endcsname\endlinenomath}}%
\newcommand*\patchBothAmsMathEnvironmentsForLineno[1]{%
  \patchAmsMathEnvironmentForLineno{#1}%
  \patchAmsMathEnvironmentForLineno{#1*}}%
\AtBeginDocument{%
\patchBothAmsMathEnvironmentsForLineno{equation}%
\patchBothAmsMathEnvironmentsForLineno{align}%
\patchBothAmsMathEnvironmentsForLineno{flalign}%
\patchBothAmsMathEnvironmentsForLineno{alignat}%
\patchBothAmsMathEnvironmentsForLineno{gather}%
\patchBothAmsMathEnvironmentsForLineno{multline}%
}

\begin{document}
\title[Persistence of AR($1$) sequences with Rademacher innovations]
{Persistence of AR($1$) sequences with Rademacher innovations and linear mod $1$ transforms}
\thanks{V.\ Vysotsky was supported in part by Dr Perry James (Jim) Browne Research Centre. V.\ Wachtel was supported by the Deutsche Forschungsgemeinschaft (DFG, German Research Foundation) - Project-ID 317210226 - SFB 1283.}

\author[Vysotsky]{Vladislav Vysotsky}
\address{University of Sussex,  Pevensey 2 Building, Brighton BN1 9QH, UK}
\email{v.vysotskiy@sussex.ac.uk}

\author[Wachtel]{Vitali Wachtel}
\address{Fakult\"at f\"ur Mathematik, Universit\"at Bielefeld, Postfach 10 01 31, 33501 Bielefeld, Germany}
\email{wachtel@math.uni-bielefeld.de}

\begin{abstract}
We study the probability that an AR(1) Markov chain 
$X_{n+1}=aX_n+\xi_{n+1}$, where $a\in(0,1)$ is a constant, stays non-negative for a long time. We find the exact asymptotics of this probability and the weak limit of $X_n$ conditioned to stay non-negative, assuming that the i.i.d.\ innovations $\xi_n$ take only two values $\pm1$ and $a \le \frac23$. This limiting distribution is quasi-stationary. It has no atoms and is singular with respect to the Lebesgue measure when $\frac12< a \le \frac23$, except for the case $a=\frac23$ and $\pr(\xi_n=1)=\frac12$, where this distribution is uniform on the interval $[0,3]$. This is similar to the properties of Bernoulli convolutions. For $0 < a \le \frac12$, the situation is much simpler, and the limiting distribution is a $\delta$-measure.

To prove these results, we uncover a close connection between $X_n$ killed at exiting $[0, \infty)$ and  the classical dynamical system defined by the piecewise linear mapping $x \mapsto \frac1a x + \frac12 \pmod 1$. Namely, the trajectory of this system started at $X_n$ deterministically recovers the values of the killed chain in reversed time. We use this fact to construct a suitable Banach space, where the transition operator of the killed chain has the compactness properties that allow us to apply a conventional argument of the Perron--Frobenius type.

\end{abstract}


\maketitle
{\scriptsize
}

\tableofcontents

\section{Introduction and main results}
\subsection{Motivation}
Let $a\in(0,1)$ be a constant and let $\{\xi_n\}_{n\ge1}$ be independent identically distributed random variables. Consider a recursive sequence
\begin{equation}
\label{AR-def}
X_{n+1}=a X_n+\xi_{n+1},\ n\ge0,
\end{equation}
where the starting point $X_0$ is independent of $\{\xi_n\}_{n\ge1}$. This
Markov chain is usually called an {\it autoregressive sequence of order $1$}, or AR($1$) in short. We will denote the distribution of $\{X_n\}_{n\ge0}$ by $\pr_{\nu}$ when $\nu$ is the distribution of $X_0$, or simply by $\pr_x$ when $X_0=x$ is a fixed real number.

We are interested in the asymptotic behaviour of the so-called {\it persistence probability} that the chain remains non-negative for a long time. That is, we want to determine the tail asymptotics for the stopping time
\begin{equation}
\label{tau-def}
\tau:=\inf\{n\ge0: X_n< 0\}.
\end{equation}

The rough logarithmic asymptotics of $\pr_x(\tau>n)$ is known under quite weak restrictions on the distribution of the innovations $\{\xi_n\}$: by Theorem~1 of \cite{HKW19}, if $\mathbb{E}\log(1+|\xi_1|)<\infty$, $\mathbb{E}(\xi_1^+)^\delta<\infty$ for some $\delta>0$, and 
$\mathbb{P}(\xi_1>0)\mathbb{P}(\xi_1<0)>0$, then there exists
a $\lambda_a\in(0,1)$ such that
\begin{align}
\label{lambda-def}
\mathbb{P}_x\bigl(\tau>n\bigr) = \lambda_a^{n+o(n)}
\end{align}
as $n \to \infty$ for any starting point $x>0$. The proof of this result is based on a rather simple subadditivity argument, which allows one to prove existence of $\lambda_a$ but gives no information about any further properties of this exponent.

It is much harder to determine the exact tail asymptotics for $\tau$. To the best of our knowledge, the most general result was 
obtained in \cite{HKW19}: if $0<\mathbb{E}(\xi_1^+)^t<\infty$ for all $t>0$, 
$0<\mathbb{E}(\xi_1^-)^\delta<\infty$ for some $\delta>0$, and the distribution of innovations is {\it absolutely continuous} with a density that is either strictly positive almost everywhere on $\mathbb R$ or has bounded support, then there exists a strictly 
positive function $V(x)$ such that, as $n\to\infty$,
\[
\pr_x(\tau>n)\sim V(x) \lambda_a^n.
\]

The assumption that {\it all} moments of $\xi_1^+$ are finite is optimal, see Proposition 19 in \cite{HKW19}. On the contrary, the assumption on absolute continuity of the distribution of innovations was prompted by the method of the proof, which relied on compactness properties of the transition operator $P$ of the Markov chain $\{X_n\}$  killed at exiting $[0,\infty)$. To clarify, this operator that acts on measurable functions on $[0, \infty)$ by $Pf(x)=\E_x f(X_1) \ind \{ \tau>1\}$, and the {\it killed chain} is the sequence $\{X_n\}_{0 \le n < \tau}$.

It is hard to imagine that the local structure of the distribution of innovations can be crucial for the tail behaviour of $\tau$.
But it is absolutely unclear how to adapt the compactness approach of~\cite{HKW19} to innovations with a generic discrete distribution. In the present paper we do this for the most simple discrete distribution of innovations, given by
\begin{equation}
\label{Rad-assump}
\pr(\xi_1=1)=p, \qquad \pr(\xi_1=-1)=q, \qquad q = 1 - p.
\end{equation}
It is known as the {\it Rademacher distribution} when $p=1/2$.

The key to the compactness approach is in finding a right functional space for the action of the transition operator of the killed chain.
Our choice of the space is prompted by a certain deterministic dynamical system 
defined by the piecewise linear mapping  $x \mapsto \frac1a x + \frac12 \pmod 1$. The details and explanation of the logic leading to this solution are given in Section~\ref{sec:Operator}. 

Initially we attempted a different approach, based on the observation that for every $a \in (0, \frac23)$, $\tau$ is the hitting time of zero for a non-negative Markov  chain $\{Y_n\}$ obtained from $\{X_n\}$ by a certain aggregation of states. This {\it aggregated} (or {\it lumped}) chain has a finite number of states for almost  every $a$. For such $a$, $\lambda_a$ is the Perron--Frobenius eigenvalue of the substochastic matrix corresponding to the transition operator of $\{Y_n\}$ killed at reaching $0$.  This approach is worth the attention even though it does not cover the remaining values of $a$ on the set of measure zero. We give the details in Section~\ref{Sec: lumped}.

Unfortunately, neither of the two approaches works for arbitrary $a> \frac23$ aside some exceptional values (see Remark~\ref{rem: >2/3}). The case $a \in (0, \frac12]$ can be solved by a simple direct argument, which gives that $\lambda_a=p$ for such~$a$. Therefore, our main interest in $a \in (\frac12, \frac23]$.

Let us comment on the background and the related literature besides \cite{HKW19}. The standard Perron--Frobenius method allows one to find the asymptotics of the probability that a finite state Markov chain avoids a subset of its states for a long time. Extending this to an infinite state space requires compactness properties of the transition operator of the corresponding killed chain. For persistence of general Markov chains, this is explored in~\cite{AMZ21}, which gives further references  and considers many examples including autoregressive chains with absolutely continuous innovations. A different approach is used in~\cite{ABRS21}, which gives explicit combinatorial formulas for persistence probabilities of the AR($1$) chain with symmetric uniformly distributed innovations. The case where the innovations have logarithmic tail probabilities is considered in~\cite{DHKW22}. For a general background on persistence problems, which have been extensively studied for many types of stochastic processes, we refer to the surveys \cite{AS15, BMS13}, where the second one gives a theoretical physics perspective.

\subsection{Main results} 
It is easy to see that if the starting point $x$ is not greater than $1/(1-a)$, then all values of the chain $\{X_n\}$ do not exceed $1/(1-a)$. If the starting point $x$ is greater than $1/(1-a)$, then the trajectory of the chain is monotonically decreasing before the downcrossing of the level $1/(1-a)$.
For this reason, it is natural to restrict our consideration to the starting points in the interval $[0, \frac{1}{1-a}]$, which we will {\it always} regard as the state space of the chain $\{X_n\}$ killed at the exit time $\tau$. We stress that our results can be easily generalized to arbitrary starting points $x \ge 0$; see, for example, Corollary~\ref{cor:large-x} below.

Assume that $a \in [\frac12, \frac23)$ and consider the mapping 
\begin{align} \label{eq: T def}
T_a(x)=
\begin{cases}
\frac{1}{a}(x+1), & 0 \le x \le \frac{2a-1}{1-a},\\
\frac{1}{a}(x-1), &1 \le x \le \frac{1}{1-a}.
\end{cases}
\end{align}
We underline that $T_a$ is defined on the set
$[0,\frac{1}{1-a}]\setminus I_a$, where 
$$
I_a=\Big(\frac{2a-1}{1-a},1\Big).
$$
This mapping plays a key role for our paper. In particular, it features in the following remarkable property that the killed
AR($1$) chain is {\it deterministic in the reversed time}: for any $n \ge 1$, on the event $\{\tau>n\}$ we have
\begin{equation} \label{eq: reversed time}
X_{n-k}=T_a^k(X_n), \qquad 0 \le k \le n.
\end{equation}
A similar property is known to hold for the stationary AR($1$) chains in the case when $1/a$ is an integer and the innovations are discrete uniform on $\{0, 1, \ldots, 1/a-1\}$, see the discussion in Bartlett~\cite{Bartlett90}. We will prove these properties in Section~\ref{sec: reversed time}. 

We can also consider the case $a=\frac23$, where we define $T_{2/3}$ as above but specify that $T_{2/3}(1)=0$ because \eqref{eq: T def} gives two values at $x=1$. Note that \eqref{eq: reversed time} ceases to hold for $a=\frac23$: if
$X_n=1$ and $\tau>n$, then we have two possible values for $X_{n-1}$, namely $0$ and $3$.

Since $T_a$ is defined on $[0,\frac{1}{1-a}]\setminus I_a$, the iterates $T_a^k(x)$ are defined only up to the first hitting time of $I_a$, given by 
\begin{equation}
\label{eq:kappa-def}
\varkappa_a(x)=\inf\left\{k\ge0:\, T_a^{k}(x)\in I_a\right\}
\in[0,\infty].
\end{equation}
In particular,  $x$ is a possible value of $X_n$ on the event $\{\tau>n\}$ if and only if $T_a^k(x)\notin I_a$ for all $0 \le k<n$. 
Furthermore, put
\begin{equation}
\label{delta-def}
\delta_k(x)= \ind \{T_a^k(x) <1 \}, \quad 0 \le k < \varkappa_a(x)+1,
\end{equation}
and define the occupation times of $[0,1)$ by
\begin{equation}
\label{L-def}
L_0(x)= 0 \quad \text{and} \quad L_k(x) = \sum_{i=0}^{k-1} \delta_i(x), \quad 1 \le k < \varkappa_a(x)+1.
\end{equation}
We omitted the index $a$ to simplify the notation.

To interpret the quantities $\delta_k(x)$ and $L_k(x)$, we note that the mapping $T_a$ is closely related to the mapping $\widehat T_a$ given by $\widehat T_a(x)=\frac1a x +\frac12 \pmod 1$ for $ 0 \le x \le 1$. Namely, since
\[
\widehat T_a \Big(\frac{ax}{2} \Big) =
\begin{cases}
\frac12 (x +1), & 0 \le x < 1, \\ 
\frac12(x-1), & 1 \le x < 3, \\
\frac12(x-3), & 3 \le x \le \frac2a, \\
\end{cases}
\]
we have $T_a(x) = \frac{2}{a}\widehat T_a(\frac{ax}{2})$ for all $x$ in the domain of $T_a$ excluding $x=3$ when $a=\frac23$. Consequently, for every $k \ge 1$ and $x \neq 3$ such that $\varkappa_a(x) \ge k$, 
\begin{equation} \label{eq: T mod 1}
T_a^k(x) = \frac{2}{a}\widehat T_a^k \Big(\frac{ax}{2} \Big).
\end{equation}
Note in passing that the sequences $ \{\frac2a \widehat T_a^k(\frac{ax}{2})\}_{k \ge 0}$ with $x \in [0, \frac{1}{1-a}]$ can enter the set $[3, \frac2a]$ only through the ``hole'' $I_a$. The related dynamical systems defined by the $\beta$-{\it transformations} $x \mapsto \beta x \pmod 1$ with a hole  were studied in~\cite{Clark16, GS15}.

On the other hand, $\widehat T_a$ is one of the  {\it linear mod~$1$} mappings $\widehat T_{\beta, \alpha}(x)=\beta x + \alpha \pmod 1$, where $\beta >1$ and $\alpha \in [0, 1)$. By  Parry~\cite[p.~101]{Parry64}, every $y \in [0,1)$ can be written  as
\begin{equation} \label{eq: expansion Parry}
y = \sum_{k=0}^\infty \big(d_k(y) - \alpha \big) \frac{1}{\beta^{k+1}},
\end{equation}
where $d_k(y)$ are the ``digits'' of $y$ given by 
$d_k(y) = [ \beta \widehat T_{\beta, \alpha}^k(y) + \alpha]$, with $[ \cdot ]$ denoting the integer part. This is a particular representation of $y$ in the base $\beta$, which we call the $(\beta, \alpha)$-{\it expansion} of $y$. Such expansions were first studied by Parry~\cite{Parry64}. In the particularly important case $\alpha=0$ (and non-integer $\beta$), these are the so-called $\beta$-{\it expansions} introduced by R\'enyi~\cite{Renyi57}. 

It follows from~\eqref{eq: T mod 1} that $1-\delta_k(x)$, for $0 \le k < \varkappa_a(x)+1 $, are the first digits 
in the $(\frac{1}{a}, \frac12) $-expansion of $ax/2$ (unless $x = 3$), and thus $L_k(x)$ is the number of $0$'s in the first $k$ digits of this expansion. Moreover, we will also show that $\delta_0(x), \delta_1(x), \ldots$ for  $a=2/3$ are the digits of the $\frac32$-expansion of $1-x/3$ (up to a minor modification); see~\eqref{eq: digits 2/3} below, where we write $T_{2/3}$ in terms of the $\frac32$-transformation $\widehat T_{3/2, \, 0}$. There are many works on digit frequencies in $\beta$-expansions, including~\cite{BCH16, Schmeling97}; unfortunately, they hardly consider concrete values of $(\beta,x)$.

The orbits of $0$ and $1$ have a distinguished role for the linear mod $1$ transformations; for example, they appear in formula~\eqref{eq: Parry} for the invariant density of  $\widehat T_{\beta, \alpha}$. Similarly, the orbit of zero under $T_a$, that is $\{T_a^k(0): 0 \le k < \varkappa_a(0) +1 \}$, is very important for our analysis. For this reason,  we denote
\begin{equation} \label{eq: zero}
\varkappa_a=\varkappa_a(0), \quad \delta_k=\delta_k(0), \quad L_k=L_k(0).
\end{equation}
It will be crucial whether the orbit of zero is finite or not, so we put
\begin{equation} \label{eq:S def}
S= \{a \in [1/2, 2/3]: \varkappa_a=\infty \}.
\end{equation}
Because $\varkappa_a$ can be infinite due to either chaotic or cyclic behaviour of the orbit, define
\[
\varkappa_a' = \# \{T_a^k(0) : 0 \le k < \varkappa_a+1\} - 1
\]
to distinguish between these cases. If $\varkappa_a'<\infty$ but $\varkappa_a=\infty$, we say that the orbit of zero is {\it eventually periodic} otherwise we call it {\it aperiodic}. Then $\varkappa_a=\varkappa'_a$ if and only if the orbit is aperiodic. We specify that the orbit is {\it purely periodic} when $T_a^{\varkappa_a'}(0)=1$. 

It is easy to see that if $\varkappa_a <\infty$, then the sequence $\{\widehat T_a^k(0)\}_{k \ge \varkappa_a}$ strictly increases until it hits $[3a/2, 1)$ at some moment $k'$, hence $d_{k'}(0)=2$. Then
\[
a \in S \text{ if and only if } \text{there are no $2$'s in the $\Big(\frac{1}{a}, \frac12 \Big)$-expansion of }0.
\]
Similarly, for any fixed $a \in (\frac12, \frac23]$, the set
\begin{equation*} 
Q_a=\{x \in [0, 1/(1-a)]: \varkappa_a(x) = \infty\} 
\end{equation*}
can be described as follows: 
\[
x \in Q_a \text{ if and only if } \text{there are no $2$'s in the $\Big(\frac{1}{a}, \frac12 \Big)$-expansion of }ax/2,
\]
once we re-define the digits of the $(\frac32, \frac12)$-expansion of $1$ as $0111 \ldots$. Thus, $Q_a$ is fully analogous to the Cantor ternary set. Lastly, we note that in the case $a \in (\frac23,1)$, which is excluded from our consideration, the $(\frac{1}{a}, \frac12)$-expansion of any point in $[0,1]$ has no $2$'s.



We can now state our main result.

\begin{theorem} \label{thm: main}
Let $\{X_n\}$ be a Markov chain defined by  equation
\eqref{AR-def} with some $a \in (\frac12, \frac23]$. Assume that the innovations $\{\xi_n\}$ satisfy \eqref{Rad-assump} with some
$p \in (0,1)$.
Then there exists a constant $c \in (0,1)$ such that, uniformly in
$x \in [0,\frac{1}{1-a}]$, we have
\begin{equation} \label{eq: main equiv}
\pr_x(\tau>n)\sim cV(x) \lambda_a^n
\end{equation}
as $n \to \infty$, where $\lambda_a=\lambda_a(p) > p$ is the unique positive solution to
\begin{equation} \label{eq: lambda}
\sum_{k=0}^{\varkappa_a} \delta_k \left( \frac{p}{\lambda} \right)^{k+1} \Big ( \frac{q}{p} \Big)^{L_k} =1
\end{equation}
and
\begin{equation} \label{eq: V}
V(x)= \sum_{k=0}^{\varkappa_a}  \Big( \frac{p}{\lambda_a} \Big)^k 
\Big(\frac{q}{p} \Big)^{L_k}\ind \{ T_a^k(0) \le x \},
\end{equation}
with $\delta_k$ and $L_k$  defined in \eqref{delta-def}, \eqref{L-def}, and \eqref{eq: zero}. The constant $c$ is given in \eqref{eq: const} below.

The function $a \mapsto \lambda_a$ satisfies $\lambda_{1/2}=p$ and is continuous and non-decreasing on $[\frac12, \frac23]$. More specifically, it is constant on every interval contained in $[\frac12, \frac23] \setminus S$ and is constant on no open interval intersecting $S$, which is a closed set of Lebesgue measure zero defined in~\eqref{eq:S def}. In other words, the Lebesgue--Stieltjes measure $d \lambda_a$ on $[\frac12, \frac23]$ has no atoms, is singular, and its topological support is $S$.

Moreover, 
the conditional distributions converge weakly, uniformly in
$x,y \in [0, \frac{1}{1-a}]$:
\begin{equation} \label{eq: weak lim}
\lim_{n \to \infty} \pr_x(X_n \le y \mid \tau>n) = 1- \sum_{k=0}^{\varkappa_a(y)} \delta_k(y)  \Big( \frac{p}{\lambda_a} \Big)^{k+1} \Big( \frac{q}{p} \Big)^{L_k(y)},
\end{equation}
where the right-hand side is the distribution function of a probability measure $\nu_a$ on $[0, \frac{1}{1-a}]$. This measure is quasi-stationary, i.e.
\begin{equation} \label{eq: quasi-st}
\pr_{\nu_a}(X_1\in A|\tau>1)=\nu_a(A),
\quad A\in\mathcal{B}([0,1/(1-a)]).
\end{equation}
This measure has no atoms and is singular with respect to the Lebesgue measure, except in the case $a=\frac23$ and $p=\frac12$, where $\nu_{2/3}$ is the uniform distribution on $[0,3]$. The topological support of $\nu_a$ is the set of non-isolated points of $Q_a$, that is, $Q_a$ itself if there is no integer $k \ge 1$ such that $T^k_a(0) = \frac{1}{1-a}$, and otherwise $Q_a \setminus \cup_{k=0}^\infty T_a^{-k}(0)$.

\end{theorem}

\begin{corollary}\label{cor:large-x}
Under the assumptions of Theorem~\ref{thm: main}, for every $x \ge 0$, we have
\[
\pr_x(\tau>n)\sim
c\lambda_a^n \e_x[\lambda_a^{-\sigma} V(X_\sigma)]
\]
where 
\[
\sigma:=\inf \left\{n \ge 0: X_n \le 1/(1-a) \right\}.
\]
Moreover, the weak convergence \eqref{eq: weak lim} holds true for all $x \ge 0$.
\end{corollary}


\begin{center}
	\parbox{0.49\textwidth}{
    \begin{center}
       \includegraphics[width=\linewidth]{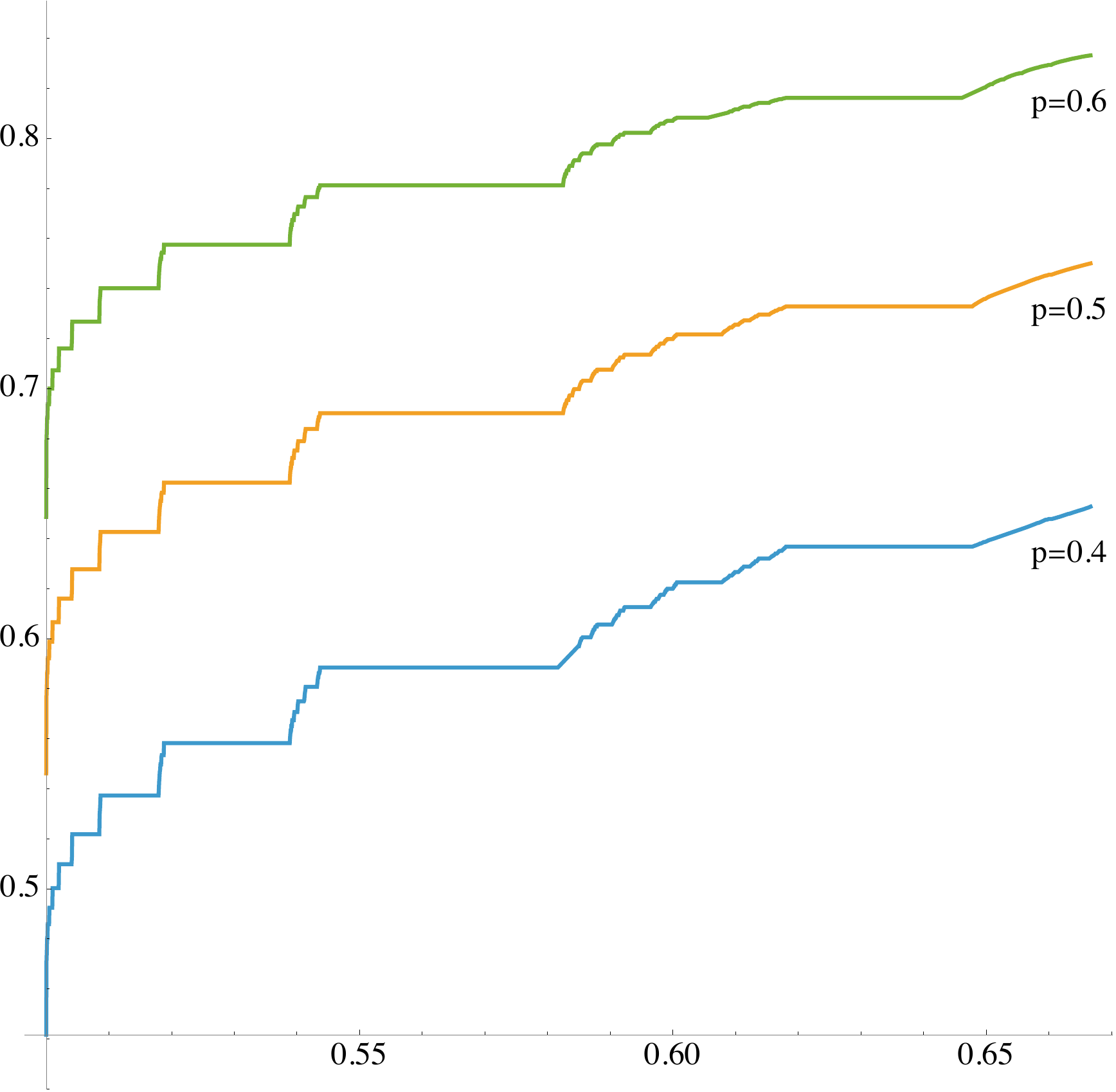} 
       \f \label{fig: lambda} The graphs of $\lambda_a(p)$ for fixed $p$.
     \end{center} }  
    	\hskip 0.2cm     
      	\parbox{0.49\textwidth}{
          \begin{center}
            \includegraphics[width=\linewidth]{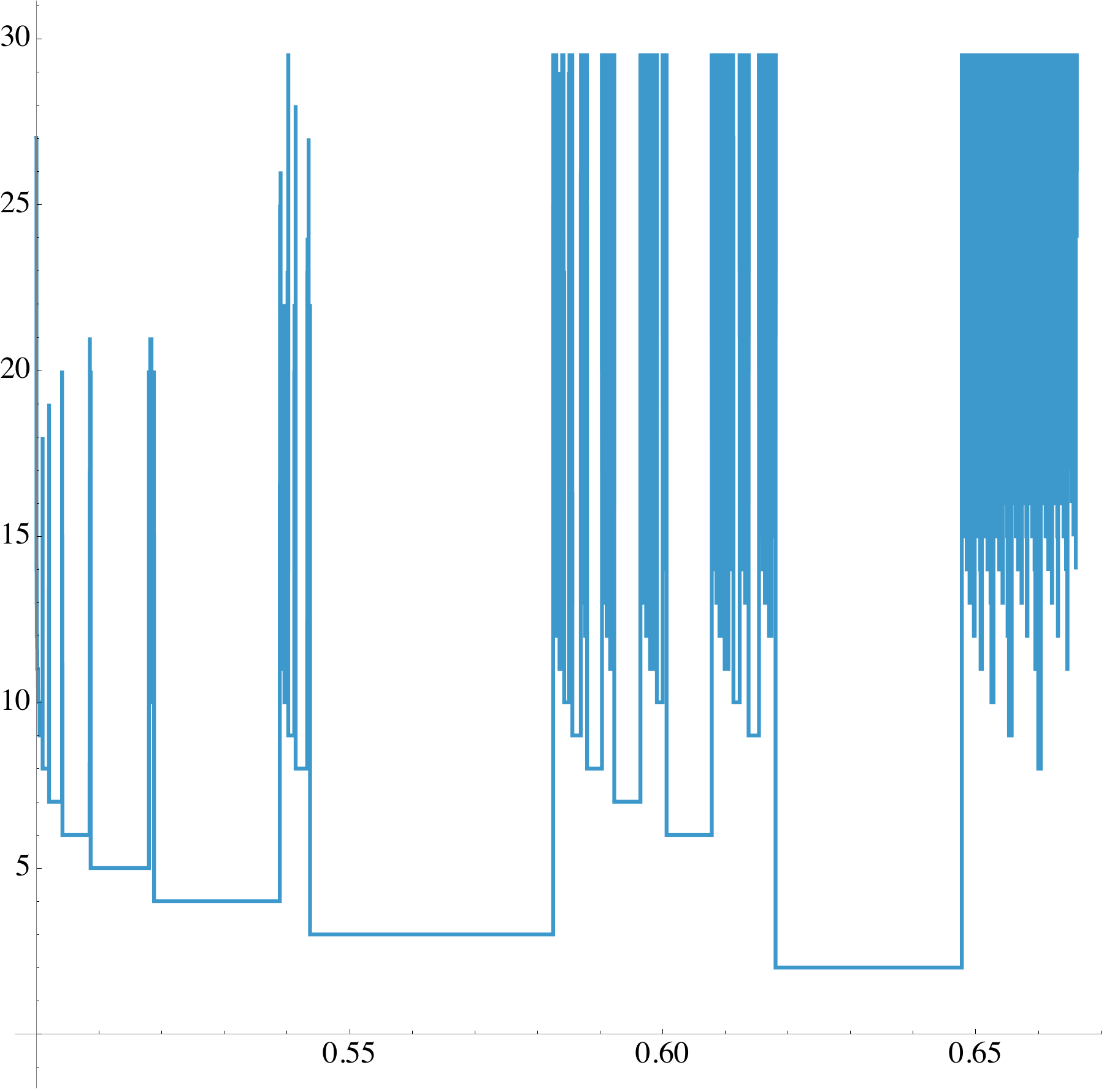} 
             \f \label{fig: kappa} The graph of $\varkappa_a'$.
         \end{center} }  
\end{center}

\begin{center}
	\parbox{0.49\textwidth}{
	\begin{center}
      \includegraphics[width=\linewidth]{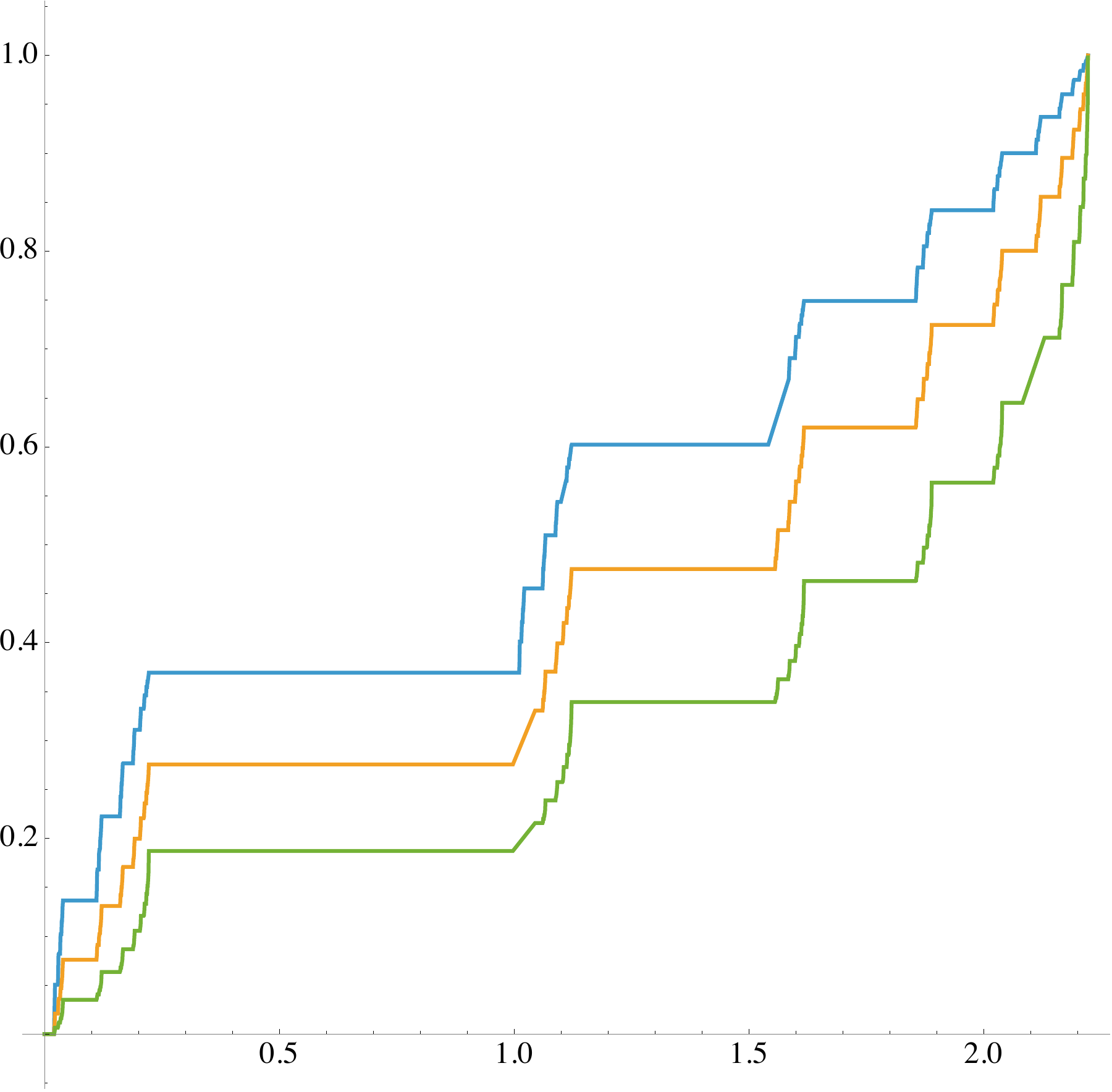} 
      a) $a=0.55$
    \end{center} }  
    \parbox{0.49\textwidth}{
	\begin{center}
      \includegraphics[width=\linewidth]{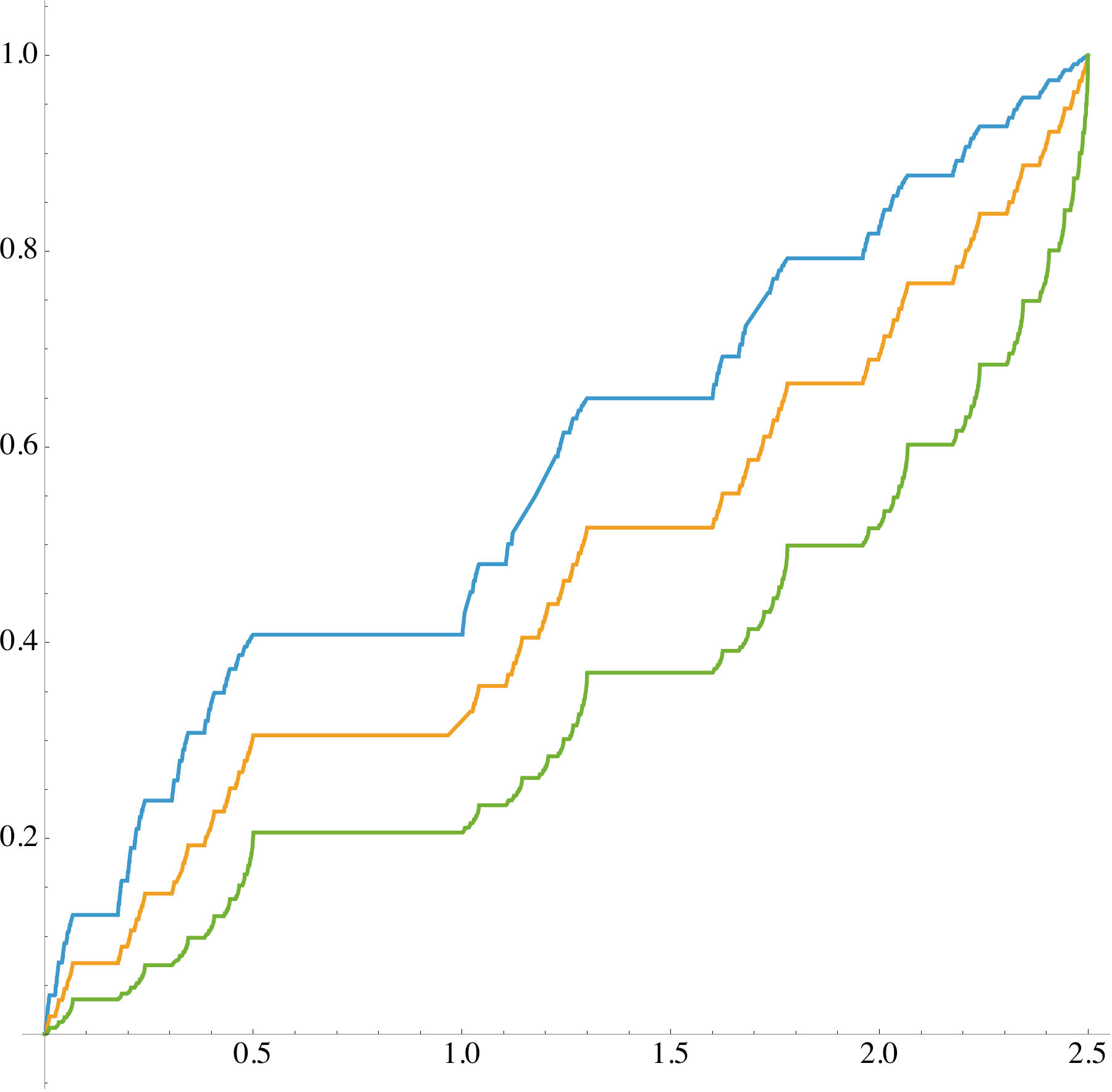} 
        b) $a=0.6$
    \end{center} }  
      
      	\parbox{0.49\textwidth}{
	\begin{center}
      \includegraphics[width=\linewidth]{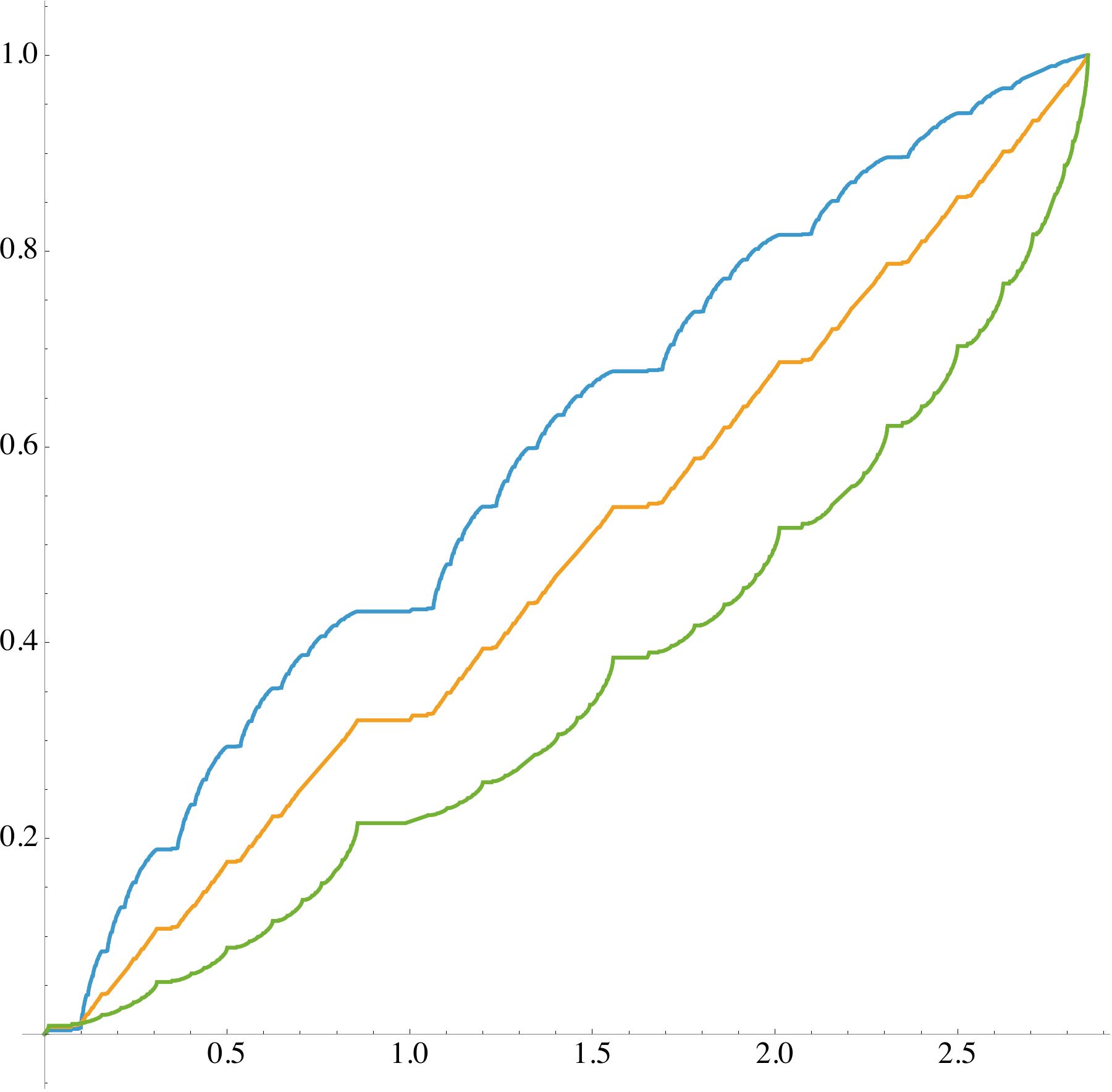} 
      c) $a=0.65$
    \end{center} }  
    \parbox{0.49\textwidth}{
	\begin{center}
      \includegraphics[width=\linewidth]{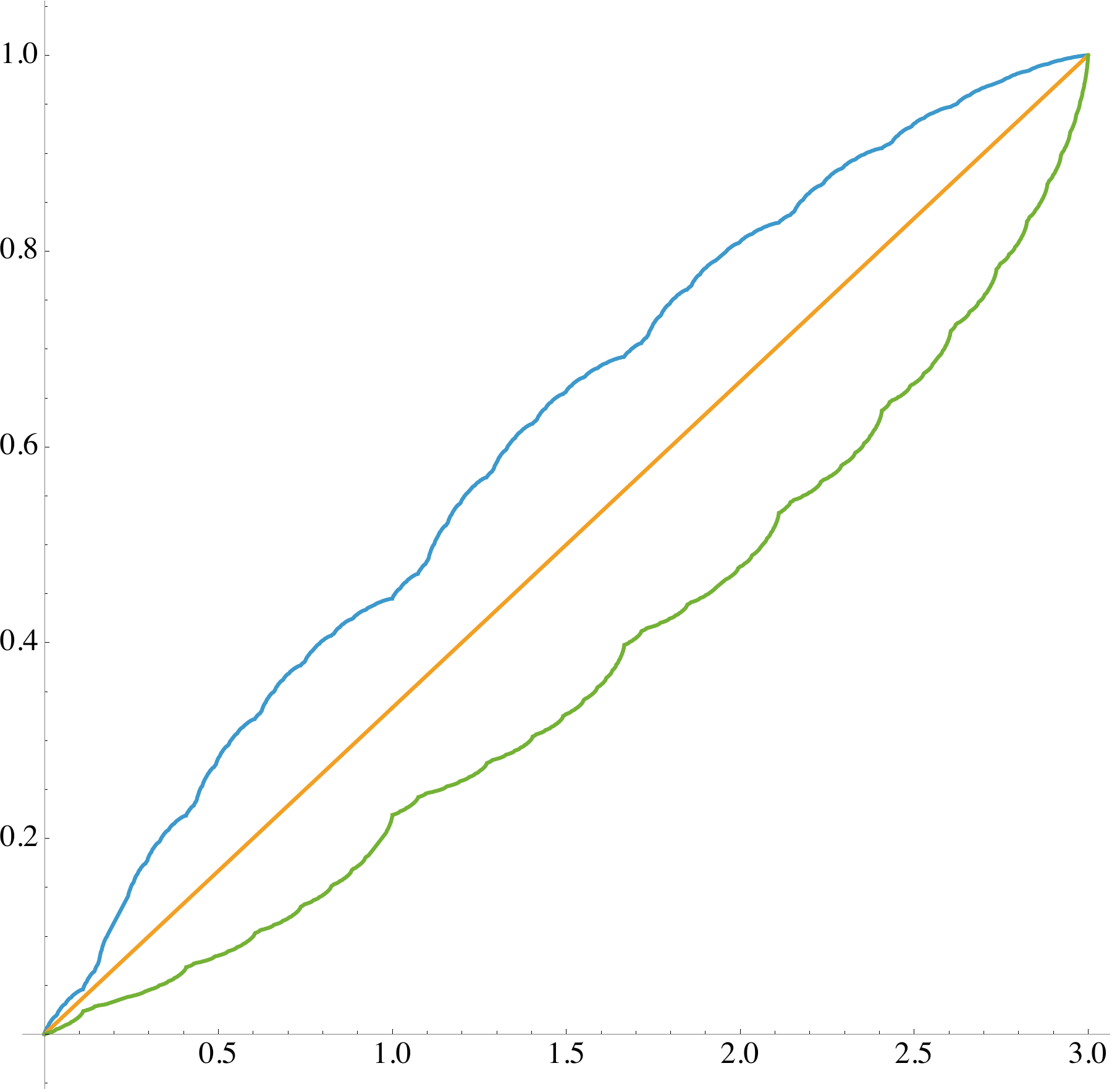} 
        d) $a=2/3$
    \end{center} }  
       
       \f \label{fig: nu} The distribution functions of $\nu_a$ for $p=0.3$, $p=0.5$, $p=0.7$ (top to bottom).
\end{center}

Let us make a few comments. 

\begin{remark} \label{rem}
\begin{enumerate}[(a), leftmargin=*]
\item The value of $\lambda_a$ is rather explicit from~\eqref{eq: lambda} when $\varkappa_a' <\infty$; see Figure~\ref{fig: lambda}.  In this case the left-hand side in~\eqref{eq: lambda} can be written as a finite sum even if $\varkappa_a=\infty$ because then the sequence $\{\delta_k\}$ is eventually periodic. Thus, $\lambda_a$ is a solution to a polynomial equation of order $\varkappa_a'+1$, and we will give some of its values in Section~\ref{Sec: lumped}. See Figure~\ref{fig: kappa} for a graph of~$\varkappa_a'$.
If $\varkappa_a'=\infty$, no simplification of~\eqref{eq: lambda} appears to be possible because of chaotic behaviour of the orbit of zero, unless $a=\frac23$ and $p=\frac12$. In this case 
\[
\lambda_{2/3}(1/2)=3/4;
\]
see Proposition~\ref{prop: p < 1/2}. We found this value computing the left-hand side of~\eqref{eq: lambda} for $\lambda = ap$ using the $(\frac{1}{a}, \frac12)$-expansion of $0$. In Section~\ref{sec: transfer}, we will present an alternative way to compute $\lambda_{2/3}(1/2)$ and to establish \eqref{eq: main equiv} with explicit expressions for $c$ and $V(x)$ in the case $a=\frac23$ and $p=\frac12$, based on a close connection between the transition operator of $\{X_n\}$ killed at leaving $[0,\infty)$ and the {\it transfer operator} associated with $T_{2/3}$. 

\item The mapping $a\mapsto\lambda_a$ has intervals of constancy essentially due to discreteness of the innovations. These intervals are the intervals of constancy of $a\mapsto\varkappa_a'$, cf.~ Figures~\ref{fig: lambda} and~\ref{fig: kappa}.
In contrast, the mapping $a\mapsto\lambda_a$  is strictly increasing if the innovations have a density that is strictly positive on $\R$ and log-concave, see Aurzada et al.~\cite[Theorem~2.7]{AMZ21}. In the particular case of the standard normal innovations, Aurzada and Kettner~\cite{AK19} derived a series expansion for $\lambda_a$. For the uniformly distributed innovations, $\lambda_a$ was found in a rather explicit form by Alsmeyer et al.~\cite{ABRS21}, see Propositions~2.4 and 3.11, and Remark~5.9(b) there. Based on numerical evidence, it appears that $ a\mapsto \lambda_a$  is strictly increasing in this case as well.

\item The fact that $\nu_a$ is singular continuous for every $a\in (\frac12, \frac23)$ and absolutely continuous for $a=\frac23$ and $p=\frac12$
reminds us strongly of the same type of behaviour for Bernoulli convolutions. Recall that the {\it Bernoulli convolution} with  parameter $a$ is the distribution $\pi_a$ of the series $\sum_{k=1}^\infty a^{k-1}\xi_k$. This is the stationary distribution for the chain $X_n$.
It is well-known that if $p=\frac12$, then $\pi_a$ is singular continuous for all $a<\frac12$, $\pi_{1/2}$ is the uniform distribution on $[-2,2]$, and $\pi_a$ is absolutely continuous for almost all $a\in(\frac12,1)$; see~\cite{Solomyak04}. In Section~\ref{sec: reversed time} we shall describe the relation between our model and Bernoulli convolutions in more detail.

\item The rate of convergence in \eqref{eq: main equiv} and ~\eqref{eq: weak lim} is exponential; see~\eqref{eq:PF} and \eqref{eq:quasi-stat}. Moreover, we extend \eqref{eq: weak lim} to convergence of the conditional functionals of the form $\E_x(f(X_n)|\tau >n)$, which also holds true at an exponential rate; see~Proposition~\ref{prop: rate}. 
\end{enumerate}
\end{remark}

We now consider the case $a\in[0,\frac12]$. Here, the analysis of persistence is rather straightforward. By $\frac{1}{1-a} \le \frac1a$, for any starting point $x \in [0, \frac{1}{a})$, we have $\tau=\inf\{n\ge1:\xi_n=-1\}$. Consequently, for such $x$ we have $\pr_x(\tau>n)  = p^n$, and thus 
\begin{equation} \label{eq: lambda <1/2}
\lambda_a(p)=p  \text{ for }  a \in (0, \tfrac12].
\end{equation}
This allows us to obtain the following analogue of Theorem~\ref{thm: main} and Corollary~\ref{cor:large-x} for all $x \ge 0$.

\begin{proposition} \label{prop: a < 1/2}
Let $\{X_n\}$ be a Markov chain defined by  equation
\eqref{AR-def} with some $a \in (0, \frac12]$. Assume that the innovations $\{\xi_n\}$ satisfy \eqref{Rad-assump} with some
$p \in (0,1)$. Denote 
\[
\sigma'=\inf \{n \ge 0: X_n <1/a \}, \qquad \sigma''=\inf \{n \ge 0: X_n <6 \}.
\]
Then for any $x \ge 0$ when $a <\frac12$ and any $x \in [0,2)$ when $a =\frac12$, for all integer $n$ large enough we have
\begin{equation} \label{eq: main triv}
\pr_x( \tau >n) = \E_x p^{-\sigma'} \cdot p^n,
\end{equation}
and for any $x \ge 2$ when $a = \frac12$,  we have
\begin{equation} \label{eq: main equiv n}
\pr_x(\tau >n ) \sim q\E_x p^{-\sigma''} \cdot np^{n-1}
\end{equation}
as $n \to \infty$. Moreover, for any $x \ge 0$, the conditional distributions $\pr_x(X_n \in \cdot \mid \tau>n)$ converge weakly to the $\delta$-measure at point $\frac{1}{1-a}$. This measure is quasi-stationary when $a <\frac12$, in the sense of \eqref{eq: quasi-st}. There is no quasi-stationary probability measure when $a=\frac12$.
\end{proposition}

Notably, the order of asymptotics in \eqref{eq: main equiv n} differs from that in \eqref{eq: main equiv} and \eqref{eq: main triv}.

\subsection{Outline of the approach and generalizations}

Our study of the asymptotics of the persistence probability $\pr_x(\tau>n)$ uses one of the most standard approaches via \mbox{(quasi-)} compactness of the transition operator $P$ of $\{X_n\}$ killed at leaving $[0,\infty)$.
The main novelty consists in the choice of an appropriate Banach space.
Our choice is prompted by the connection between our AR$(1)$ chain with the $\pm 1$ innovations and the dynamical system given by the iterations of $T_a$. This is explained in detail in Section~\ref{sec:Operator}.
We will see that $P$ is quasi-compact on a certain closed subspace of the Banach space $BV$ of functions of bounded variation on $[0, \frac{1}{1-a}]$, where $P$ can be represented by a linear operator acting on summable sequences of length $\varkappa_a'+1$; see Proposition~\ref{prop: A} and Remark~\ref{rem:compact}. We will use this representation to show that the leading eigenvalue of $P$ is $\lambda_a$ and the corresponding eigenfunction is $V$, and then  prove~\eqref{eq: main equiv}.
To prove the convergence of conditional distributions in \eqref{eq: weak lim}, we consider the operator $P$ on the whole of $BV$ and derive an appropriate version of the Perron--Frobenius theorem for $P$, see Subsection~\ref{ssec:cond}. We will also show that $P$ is quasi-compact on $BV$; see Remark~\ref{rem: compact BV}.

In all of our proofs it takes much more effort to consider the case $p < 1/2$, where we need to use uniform upper bounds for the frequencies $L_n(x)/n$ of zeros in the $(\frac1a, \frac12)$-expansion of $ax/2$. We believe that such bounds, presented in Section~\ref{sec: lambda}, are of independent interest. 

\begin{remark}
Our approach can be extended to study persistence of the AR(1) chains with the innovations that take two arbitrary values of different sign. Indeed, thanks to a multiplicative rescaling, it suffices to consider the case where $\pr(\xi_1=1)=p$ and $\pr(\xi_1=-b)=q=1-p$ for some $b>0$. Then for any $a \in (\frac{b}{b+1}, \frac{b+1}{b+2})$, the deterministic relation in reversed time~\eqref{eq: reversed time} remains valid if we substitute $T_a$ by the mapping $T_{a, b}$ that is defined by
\begin{align*} 
T_{a,b}(x)=
\begin{cases}
\frac{1}{a}(x+b), & 0 \le x \le \frac{a(b+1)-b}{1-a},\\
\frac{1}{a}(x-1), &1 \le x \le \frac{1}{1-a}.
\end{cases}
\end{align*}
Define $\varkappa_{a,b}(x)$ as in \eqref{eq:kappa-def} with $I_a$ replaced by $(\frac{a(b+1)-b}{1-a},1)$. 

It is easy to check that, for any $a$ as above, our proofs of \eqref{eq: main equiv}  and~\eqref{eq: weak lim} carry over without change, and these relations remain true when $\lambda_a$, $\delta_k$, etc.\ are replaced by the corresponding quantities $\lambda_{a,b}$, $\delta_k^{(b)}$, etc., defined in terms of $T_{a, b}$ and $\varkappa_{a,b}$ in place of $T_a$ and $\varkappa_a$. In general, it can be that $\varkappa_{a,b}=1$, as opposed to $\varkappa_a \ge 2$. Furthermore, it follows that $1-\delta_k^{(b)}(x)$, for $0 \le k < \varkappa_{a, b}(x)+1 $, are the first digits  in the $(\beta, \alpha)$-expansion of $\frac{ax}{b+1}$ with $\beta=\frac{1}{a}$ and $\alpha=\frac{b}{b+1}$. For the critical value $a = \frac{b+1}{b+2}$, our proofs of \eqref{eq: main equiv}  and \eqref{eq: weak lim} remain valid without change if the orbit of $0$ under $T_{a,b}$ does not hit $1$, yielding 
\[
\lambda_{\frac{b+1}{b+2},  b}\Big(\frac12\Big)=\frac{b+2}{2(b+1)}.
\]


\end{remark}

\medskip

The structure of the rest of the paper is clear from the titles of the following sections. 

\section{Deterministic dynamics under time reversal and three approaches to persistence} \label{sec: discussion}
In this section we prove the deterministic dynamics in the reverse time, given by \eqref{eq: reversed time}, and describe the three possible approaches to persistence of the AR($1$) chains with $\pm 1$ innovations. All these approaches use the dynamical system featuring in~\eqref{eq: reversed time}. The first method is based on a reduction to finite Markov chains. Unfortunately, this reduction does not work for all values of the parameter $a$. For this reason we introduce an alternative, more analytic approach based on compactness properties of the operator $P$. The third approach works only for $a=\frac23$. Although we do not make use of it, we present it to highlight a connection to the well-developed theory of transfer operators.

The following formula consistently extends the definition of $T_a$ to all $a \in (0,1)$:
\begin{align*}
T_a(x)=
\begin{cases}
\frac{1}{a}(x+1), & 0 \le x < 1 \text{ and } x \le \frac{2a-1}{1-a},\\
\frac{1}{a}(x-1), &1 \le x \le \frac{1}{1-a}.
\end{cases}
\end{align*}

\subsection{Deterministic behaviour in reverse time} \label{sec: reversed time}
Let us prove the deterministic dynamics of $\{X_n\}$ in the reversed time as described in \eqref{eq: reversed time}, extending  the range of the parameter to $a \in (0, \frac23)$. Namely, we claim that for all 
$n \ge 1$ and all starting points $x \in [0, \frac{1}{1-a}]$, on the event $\{\tau>n\}$ it holds that $X_{n-k}=T_a^k(X_n)$ for $1 \le k \le n$. 

This is trivial for $a \in (0, \frac12)$ because on the event $\{\tau>n\}$ it must be $X_k=a X_{k-1} + 1$ for all $1 \le k \le n$, because otherwise $X_k<0$ by
\begin{equation} \label{eq: X_k<}
X_k=a X_{k-1} - 1 \le \frac{a}{1-a}-1 =\frac{2a-1}{1-a}.
\end{equation}
Hence $X_k\ge 1$ and $X_{k-1}=\frac{1}{a}(X_k-1) = T_a(X_k)$, as claimed.

Let us prove \eqref{eq: reversed time} for $a \in [\frac12, \frac23)$. If $X_n<1$, then $X_n =a X_{n-1} -1$ because the other option $X_n =a X_{n-1} +1$ is impossible by $a X_{n-1} +1 \ge 1$ on $\{ \tau >n\}$. Hence $X_n\le\frac{2a-1}{1-a}$ by~\eqref{eq: X_k<}. Thus, $X_n$ is in the domain of $T_a$ and we can write $X_{n-1}=\frac{1}{a}(X_n+1)=T_a(X_n)$. If $X_n \ge 1$, then $X_n =a X_{n-1} +1$ because otherwise $X_n =a X_{n-1} -1<1$ by~\eqref{eq: X_k<}. Then $X_{n-1}=\frac{1}{a}(X_n-1)=T_a(X_n)$. 
We thus checked that $X_{n-1}=T_a(X_n)$ holds true on $\{\tau>n\}$ in either case, and \eqref{eq: reversed time} follows by induction.

Note that the above argument does not work when $a \ge \frac23$ because it is impossible to specify whether $X_{n-1}=\frac1a(X_n-1)$ or $X_{n-1}=\frac1a(X_n+1)$ for $ X_n \in [1, \frac{2a-1}{1-a}]$.

The argument above also shows that for $a \in [\frac12, \frac23)$, 
on the event $\{\tau>n\}$  we can recover the innovations as follows:
\[
\xi_{n-k}= (-1)^{\ind \{T_a^k(X_n) < 1\}}, \qquad 0 \le k \le n-1.
\]
In other terms, by \eqref{eq: T mod 1},
\begin{equation}
\xi_{n-k}= (-1)^{1-d_k(a X_n/2)} = 2 d_k(a X_n/2) -1.
\end{equation}
Thus, $(\xi_{n-k}+1)/2$ for $0 \le k \le n-1$ are the first $n$ digits in the $(\frac{1}{a}, \frac12)$-expansion of $a X_n/2$. 

The deterministic dynamics in the reversed time described above is essentially the reason for singular continuity of the quasi-stationary distributions $\nu_a$. The same effect can be also observed in the unconditional setting. More precisely, we shall now show that the unconditioned chain $\{X_n\}$ is deterministic in the reversed time for every $a\in (0, \frac12)$. We first notice that $|X_k|\ge\frac{1-2a}{1-a}$ for all $k\ge1$ and for all starting points $x\in[-\frac{1}{1-a},\frac{1}{1-a}]$; this follows as in \eqref{eq: X_k<}. Assume that $X_n\ge\frac{1-2a}{1-a}$. In general, we have two possibilities for $X_{n-1}$, namely $X_{n-1}=\frac{1}{a}(X_n+1)$ and 
$X_{n-1}=\frac{1}{a}(X_n-1)$. The assumption $X_n\ge\frac{1-2a}{1-a}$ implies that the case where $X_{n-1}=\frac{1}{a}(X_n+1)$ is not possible since 
\[
\frac{1}{a}\left(\frac{1-2a}{1-a}+1\right)
=\frac{2-3a}{a(1-a)}>\frac{1}{1-a}.
\]
Similar arguments show that $X_{n-1}=\frac{1}{a}(X_n+1)$ in the case 
$X_n\le-\frac{1-2a}{1-a}$. As a result we have 
\begin{equation}
\label{eq:reversed.new}
X_{n-k}=G^k_a(X_n),\quad 0 \le k\le n,
\end{equation}
where
\[
G_a(x)=
\begin{cases}
\frac{1}{a}(x+1), & x\in[-\frac{1}{1-a},-\frac{1-2a}{1-a}],\\ 
\frac{1}{a}(x-1), & x\in[\frac{1-2a}{1-a},\frac{1}{1-a}].
\end{cases}
\]

This property implies that the limiting distribution $\pi_a$ is supported on the points $z$ whose orbit under $G_a$ does not hit the interval 
$\left(-\frac{1-2a}{1-a},\frac{1-2a}{1-a}\right)$. The set of such points has Lebesgue measure zero. 
Since the two ``branches'' of $G_a$ are onto the whole of the interval $[-\frac{1}{1-a},\frac{1}{1-a}]$, besides showing that the distribution $\pi_a$  is singular and  continuous, one can compute the Hausdorff dimension of its support. The proof of these properties via deterministic evolution in the reversed time can be found in Lemma 2 by Bovier~\cite{Bovier92}. Relation~(II-17) there is equivalent to \eqref{eq:reversed.new}.
\subsection{Reduction to a finite chain when
$\varkappa'_a<\infty$} \label{Sec: lumped}
In this section we explain the approach based on reduction to a Markov chain obtained by lumping the states between the points of the orbit of $0$ under $T_a$. This Markov chain has a finite number of states when 
$\varkappa'_a<\infty$, which allows one to find the asymptotics in the exit problem using the classical Perron--Frobenius theorem for non-negative matrices. This also allows us to determine the exponent $\lambda_a$ numerically.

The key to the approach is a coupling property for the stopped Markov chain $\{X_{n \wedge \tau}\}_{n \ge 0}$ started from different points. In order to state this property, we introduce additional notation. For any real $x$, denote by $\{X_n^x\}_{n\ge0}$ the autoregressive sequence defined in \eqref{AR-def} starting at $X_0=x$, and put 
\[
\tau_x:=\inf\{n\ge0: X_n^x< 0\}
\] 
as in \eqref{tau-def}. 
Consider the following set of non-decreasing functions on $\R$:
\[
U_+=\left\{f:f(x)=\sum_{k=0}^{\varkappa_a} u_k{\ind}\{x\ge T_a^k(0)\}, \sum_{k=0}^{\varkappa_a} u_k < \infty, u_k>0\right\}.
\]

\begin{proposition} \label{prob: coupling}
Assume that $a\in (0,\frac{2}{3}]$ and $f \in U_+$. Then for any  
$x, y \in [0, \frac{1}{1-a}]$ such that $f(x)=f(y)$, we have $\tau_x = \tau_y$ and
\begin{equation} \label{eq: coupling}
f(X_{n \wedge \tau_x}^x)  = f(X_{n \wedge \tau_y}^y), \qquad n \ge 1. 
\end{equation}
\end{proposition}
\begin{proof}
Assume that $a \in (0, \frac12)$. We have $\varkappa_a'=0$ and thus $f$ is constant on $[0, \frac{1}{1-a}]$. Therefore, $f(X_{n \wedge \tau_x}^x)$ stays constant until $\{X_n\}$ leaves $[0, \frac{1}{1-a}]$ at the moment $\tau_x$. Noting now that $\tau_x=\inf\{n \ge 1: \xi_n =-1\}$ for every $x\in[0, \frac{1}{1-a}]$, we obtain the equalities $\tau_x=\tau_y$ and \eqref{eq: coupling}.

Assume now that $a \in [\frac12, \frac23)$. From the representation
\[
X_{(n+1) \wedge \tau_z}^z = 
\begin{cases}
a X_{n  \wedge \tau_z}^z + \xi_{n+1}, \quad \text{if } \tau_z >n,\\
X_{ n \wedge \tau_z}^z, \quad \text{if } \tau_z \le n,
\end{cases}
\]
which holds true for every real $z$, and the facts that  $X_{\tau_z}^z<0$ and $f=0$ on $(-\infty, 0)$, we see that equality \eqref{eq: coupling}  follows by a simple inductive argument if we prove it for $n=1$. 

Fix an $x$ and consider the set
\[
J=f^{-1}(f(x)) \cap \Big [0, \frac{1}{1-a} \Big].
\] 
Assume that $y \in J$ is distinct from $x$, otherwise the claim is trivial. The assumption $f \in U_+$ implies that $f$ is the distribution function of a finite atomic measure supported at the points of $\{T_a^k(0): 0 \le k < \varkappa_a +1 \}$. Hence $f$ jumps at each of these points and is right-continuous there, and is continuous at all the other points. Therefore, since the set $J$ contains at least two points $x$ and $y$, it is an interval that has no point of  $\{T_a^k(0): 0 \le k < \varkappa_a + 1\}$ in its interior. This interval always includes its left endpoint and it does include the right one whenever this point is not in  $\{T_a^k(0): 0 \le k < \varkappa_a + 1 \}$.

To prove \eqref{eq: coupling}, we shall show that $f$ is constant on both intervals $aJ \pm 1$. 

We first argue that $f$ is constant on the interiors of these intervals. This can be violated only if  there is an integer $0 \le k \le \varkappa_a$ such that $T_a^k(0)$ is in the interior of one of these intervals. Equivalently, there is a $z \in \intr J$ such that $az+1 = T_a^k(0)$ or $az-1 = T_a^k(0)$. In the former case, we have $T_a^k(0) \ge 1$, hence $k < \varkappa_a$ and $z = \frac{1}{a}(T_a^k(0)-1)=T_a^{k+1}(0)$, which is impossible. In the latter case, we have  $T_a^k(0) <  \frac{a}{1-a}-1=\frac{2a-1}{1-a} \le 1$ by $a \le \frac23$ and $z<\frac{a}{1-a}$, hence $k < \varkappa_a$ and $z = \frac{1}{a}(T_a^k(0)+1)=T_a^{k+1}(0)$, which is again impossible. 

Denote by $r$ the right endpoint of $J$. By right-continuity of $f$, the value of $f$ at the left endpoint of $aJ - 1$ equals the value of $f$ on $\intr(aJ-1)$, and the same applies to $aJ +1$. This establishes constancy of $f$ on the whole of $aJ - 1$ and on the whole of $aJ +1$ if $r \not \in J$ because in this case the intervals $aJ \pm 1$ do not include their right endpoints. 

In the opposite case when $r \in J$, it remains to check that $f$ is continuous at $ar \pm 1$. This can be violated only if  there is an integer $0 \le k \le \varkappa_a$ such that $ar+1=T_a^k(0) $ or $ar-1=T_a^k(0) $. In the former case, the argument above applied for $r$ instead of $z$ shows that $r=T_a^{k+1}(0)$ for a $k < \varkappa_a$, hence $f$ is discontinuous at $r$, contradicting the assumption that $r \in J$. In the latter case, we arrive at the same contradiction because the argument above applies verbatim unless $r=\frac{a}{1-a}$ and $a=\frac23$, in which case the equality $T_a^k(0)=1$ is impossible, see Section~\ref{sec:Operator}. Thus, $f$ is constant on $aJ \pm 1$ in either case, which proves~\eqref{eq: coupling}.

Lastly, the equality $\tau_x=\tau_y$ follows from \eqref{eq: coupling} and the facts that $f =0$ on $(-\infty, 0)$ and $f> 0$ on $[0, \infty)$.
\end{proof}
\begin{corollary} \label{cor: lumped}
For any $x \in [0, \frac{1}{1-a}]$, the sequence
$f(X^x_n) \ind \{\tau_x > n \}$ is a time-homogeneous non-negative Markov chain under $\pr_x$ and it is true that
$\tau_x = \inf \{n \ge 1: f(X_n)=0\}$. If $\varkappa'_a<\infty$, this  chain has $\varkappa'_a +2$ states, including the absorbing state at zero.
\end{corollary}
\begin{proof}
We first restrict both function $f$ and the stopped Markov chain $\{X_{n \wedge \tau}\}$ to the set $(-\infty, \frac{1}{1-a}]$. Then Proposition~\ref{prob: coupling} implies that the distribution of $f(X_{1 \wedge \tau})$ under $\pr_z$ does not change as $z$ varies over any level set of $f$. This condition is known to ensure, see e.g.\ \cite[Corollary 1]{BR58} (this particular reference applies only for countable chains), that the lumped sequence $\{f(X_{n \wedge \tau})\}_{n \ge 0}$ is a Markov chain on the range of $f$. This set has cardinality $\varkappa_a'+2$ when $\varkappa_a'<\infty$. The state zero is clearly absorbing. 
Since $f=0$ on $(-\infty, 0)$, we have $f(X_{n \wedge \tau}) = f(X_n) \ind \{\tau > n \}$. Thus, $f(X_n) \ind \{\tau > n \}$ is a Markov chain. The equality for $\tau$ follows because $f > 0$ on $[0, \infty)$ by $f \in U_+$.
\end{proof}
Note that if $\varkappa'_a<\infty$, we can relabel the states, hence the result of Corollary~\ref{cor: lumped} remains valid for the  function $f(x)=\sum_{k=0}^{\varkappa_a'} u_k{\ind}\{x\ge T_a^k(0)\}$ that satisfies $f(T_a^k(0)) = k+1$ for all integer $0 \le k \le \varkappa'_a$. For this $f$, the transition probabilities of the chain $f(X_n) \ind \{\tau > n \}$ satisfy $p_{0,0}=1$, $p_{1,0}=1-p$, and
\[
p_{k+1, k} =(1- p) \delta_{k-1} + p(1- \delta_{k-1}), \qquad 1 \le k \le \varkappa'_a.
\]
The latter equalities follow from
\[
T_a^{k-1}(0)=
\begin{cases}
a T_a^k(0)-1, \quad \text{if } T_a^{k-1}(0)<1,\\
a T_a^k(0)+1, \quad \text{if } T_a^{k-1}(0) \ge 1.
\end{cases}
\]

Moreover, for $1 \le k \le \varkappa'_a$ such that $T_a^k(0)<1/a$, we complement the relations above by $p_{k+1,0}=1-p$ (and for such $k$ it must be $p_{k+1,k}=p$ since $\delta_{k-1}=0$). In addition, if $\varkappa_a<\infty$, we have $\varkappa_a=\varkappa'_a$ and  $p_{1,\varkappa'_a +1} =p$ because $\pr_0(X_1=1)=p$ and $f(1)=f(T_a^{\varkappa'_a}(0))$ since there are no points of the orbit of zero in $(T_a^{\varkappa'_a}(0), 1]$. For the remaining $0 \le k \le \varkappa'_a$, we cannot give a simple expression for the unique $k'$ such that $p_{k+1, k'}=1-p_{k+1, k}$.

Let us give two examples. Suppose that $\varkappa'_a = 2$. Equivalently,
\[
\frac{2a-1}{1-a} < T_a^2(0)= \frac{1-a}{a^2}\le 1,
\] that is 
\[
0.618 \approx \frac{\sqrt5 -1}{2} \le a < \frac16 \left(2 -\sqrt[3]{6 \sqrt{33} - 26} + \sqrt[3]{6 \sqrt{33} + 26} \right) \approx 0.6478.
\]
In this case the transition matrix of the lumped chain is given by
\[
\begin{pmatrix}
1 & 0 & 0&0 \\
1-p & 0 & 0 & p\\
0& 1-p & p & 0 \\
1-p & 0  & p & 0\\
\end{pmatrix}.
\]
Then $\lambda_a$ is the leading eigenvalue of the $3 \times 3 $ matrix obtained by discarding the first row and the first column from the matrix above.  This matches \eqref{eq: lambda}.

Suppose that $\varkappa'_a = 3$. Equivalently, $\frac{2a-1}{1-a} =  \frac{1-a}{a^2}$, i.e.\ $ a \approx 0.6478$, or  
\[
\frac{2a-1}{1-a} < T_a^3(0)= \frac{1-a-a^2}{a^3}\le 1 < T_a^2(0)= \frac{1-a}{a^2},
\] 
i.e.\ approximately, $0.5437 \le a < 0.5825$. In theses cases the transition matrices of the lumped chains are given respectively by
\[
\begin{pmatrix}
1 & 0 & 0&0 &0\\
1-p & 0 & 0 & p&0\\
0& 1-p & p & 0 &0 \\
1-p & 0  & p & 0 &0\\
0& 0& 0& 1-p &p \\
\end{pmatrix}
\text{ and } 
\begin{pmatrix}
1 & 0 & 0&0 &0\\
1-p & 0 & 0 & 0&p\\
0& 1-p & p & 0 &0 \\
1-p & 0  & p & 0 &0\\
1-p& 0& 0& p &0 \\
\end{pmatrix}.
\]

\begin{remark} We do not know how to make use of the lumped Markov chain $f(X_n) \ind \{\tau > n \}$ when $\varkappa'_a=\infty$. In particular, it is unclear how to describe concisely the transition probabilities of this chain because it may have uncountably many states even though every function $f \in U_+$ has countably many jumps. 
\end{remark}
In the next section we present a different approach, which works for all values of $a \in (0, \frac23]$, unlike the reduction described above. This method is based on the compactness properties of the transition operator of the killed chain $\{X_n\}$.


\subsection{Compactness approach} \label{sec:Operator}

Let $P$ be the transition operator of $\{X_n\}$ killed at exiting $[0,\frac{1}{1-a}]$. It acts on a measurable bounded function $f$ on $[0,\frac{1}{1-a}]$ as follows:
$$
Pf(x)=\E_x[f(X_1);\tau>1].
$$
Under assumption \eqref{Rad-assump} that the innovations are $\pm1$, we have
\begin{equation}
\label{eq:P+Rad}
Pf(x)=pf(ax+1)+q	f(ax-1)\ind \{x\ge 1/a\}.
\end{equation}			    							
We can see that if $f$ has finite right and left limits at every point, then the same is true for $Pf$. We will consider only such functions in what follows.

Since
\[
\mathbb{P}_x(\tau>n) = P^n \ind (x),
\]
we seek to find a Banach space of functions on $[0, 1/(1-a)]$ where $P$ would be compact or quasi-compact, expecting that the asymptotics of $P^n \ind$ is defined by the spectral radius of $P$, which should be the largest eigenvalue of $P$. To this end, let us consider the behaviour of the set of discontinuities $D_f$ of a function $f$ under the action of $P$, with continuity at the endpoints $0$ and $1/(1-a)$ understood as one-sided continuity.

1. The case  $a \in (0, \frac23)$.

Assume that $a \in [\frac12, \frac23)$. If $Pf$ is discontinuous at an $x \in [0,\frac{1}{1-a}]$, then the same holds true for at least one of the two terms in \eqref{eq:P+Rad}. Note that $f(ax+1)$ is discontinuous at $x$ if and only if $f$ is discontinuous at $y = ax +1$; here $y \in [1,\frac{1}{1-a}]$ and $x=\frac1a (y-1)$. On the other hand, $f(ax-1)\ind \{x\ge 1/a\}$ is discontinuous at $x$ if and only if $f$ is discontinuous at $y = ax -1$ or it  holds that $x = 1/a, f(0)\neq 0$; here $y \in [0,\frac{2a-1}{1-a}]$ and $x=\frac1a (y+1)$ in both cases. 

If $a \in (0, \frac12)$, then the second term in \eqref{eq:P+Rad} vanishes but the analysis of the first term remains unchanged.

The observations above can be summed up using the mapping 
\begin{align*}
T_a(x)=
\begin{cases}
\frac{1}{a}(x+1), & 0 \le x \le \frac{2a-1}{1-a},\\
\frac{1}{a}(x-1), &1 \le x \le \frac{1}{1-a}.
\end{cases}
\end{align*}
as follows: we showed that if $Pf$ is discontinuous at $x$, then
$x=T_a(y)$ for some $y \in D_f \setminus I_a$ or it holds that $x=1/a$, $f(0) \neq 0$, $a \ge \frac12$. In other words, for all $a \in (0, \frac23)$ we have
\begin{equation}
\label{eq:disc}
D_{Pf} \subset  T_a\big((D_f \cup \{0\}) \setminus I_a \big).
\end{equation}

It follows from \eqref{eq:disc} that for $a \in (0, \frac23)$, the set of measurable bounded functions on $[0,\frac{1}{1-a}]$ that are continuous at every point outside of $\{T_a^k(0): 0 \le k < \varkappa_a +1 \}$ is closed under the action of $P$; recall that
\[
\varkappa_a=\inf\{k\ge0:\, T_a^{k}(0)\in I_a\}.
\]
This suggests to consider the set of functions 
\begin{equation} \label{eq: U}
U=\left\{f:f(x)=\sum_{k=0}^{\varkappa_a} u_k{\ind}\{x\ge T_a^k(0)\}, \sum_{k=0}^{\varkappa_a} |u_k| < \infty, u_k\in \R\right\}
\end{equation}
on $[0, \frac{1}{1-a}]$, because we know how $P$ transforms the jumps of functions and every function in $U$ is defined by its jumps. The idea to consider general right-continuous step functions with countably many jumps (called {\it saltus functions}) in the context of linear mod 1 transforms goes back to Halfin~\cite{Halfin75}. 

We will show that the set $U$ is closed under the action of $P$. Then we will give a simple, explicit description of this action and show that $P$ is a quasi-compact operator on~$U$; see Proposition~\ref{prop: A} and Remark~\ref{rem:compact}.

2. The case $a \in [\frac23, 1)$. 

Here $\frac{2a-1}{1-a} \ge 1$ and thus the mapping $T_a$ does not account for the discontinuities of $Pf$ arising from the discontinuities of $f$ on $[1, \frac{2a-1}{1-a}]$ due to the second term in \eqref{eq:P+Rad}. The argument above gives
\begin{equation}
\label{eq:disc hard}
D_{Pf} \subset  T_a \big(D_f \cup \{0\}\big) \cup \left \{\frac1a(y+1): y \in D_f \cap \Big [1, \frac{2a-1}{1-a} \Big] \right \}.
\end{equation}

The right-hand side simplifies when $f$ has no discontinuities in $[1, \frac{2a-1}{1-a}]$. This motivates us to introduce
\[
\varkappa_a=\inf \left \{k\ge0:\, T_a^{k}(0)\in\Big [1, \frac{2a-1}{1-a} \Big] \right \}.
\]
Let us check that this definition matches the one in \eqref{eq:kappa-def} when $a =\frac23$. To this end, we shall show that the orbit of zero under $T_{2/3}$ is not purely periodic, i.e.\ it does not include $1$ (note in passing that the orbit is actually aperiodic). We use the identity
\[
T_a^k(0) = -\sum_{i=0}^{k-1} (-1)^{\delta_i} a^{i-k}, \qquad 1 \le k < \varkappa_a+1,
\]
which holds true for every $a \in [\frac12, \frac23]$ and follows by simple induction. Hence if $T_{2/3}^k(0)=1$, then $2^k=-\sum_{i=0}^{k-1} (-1)^{\delta_i} 3^{k-i} 2^i$, which is impossible because the right-hand side is odd. 

With this extended definition of $\varkappa_a$, the functional space $U$ introduced in \eqref{eq: U} is again closed under the action of $P$ if $a \in [\frac23 , 1)$ is such that $\varkappa_a = \infty$. The set of such $a$ is contained in $[\frac23, \frac{\sqrt 2}{2})$ because $T_a(0) \not \in  [1, \frac{2a-1}{1-a}]$ only when $\frac{2a-1}{1-a}< \frac{1}{a}$, hence $2a^2<1$. It is possible to show that this set has Lebesgue measure zero. 
Note that it contains points other than $a=2/3$, for example, the unique solution to $T_a^5(0)=1/a$ on $[2/3,1)$, which is $a \approx 0.691$. 
\begin{remark} \label{rem: >2/3}
Our method of proving tail asymptotics~\eqref{eq: main equiv} seem to work unchanged for every $a \in (\frac23 , 1)$ such that
$\varkappa_a = \infty$. 
\end{remark}

There is no reason to restrict our consideration merely to discontinuities. A similar argument yields the following representation for the killed transition operator in \eqref{eq:P+Rad}: for any $a \in (0, \frac23]$ and $x \in [0, \frac{1}{1-a}] \setminus \{3\}$, we have
\begin{equation} \label{eq: weightedPF}
Pf(x) = \sum_{y \in T_a^{-1}(x)} \big [q \ind\{y<1\}+p\ind\{y\ge1\} \big] f(y).
\end{equation} 

\subsection{Transfer operator approach} \label{sec: transfer}
Assume that $a=\frac23$, and use the shorthand $T=T_{2/3}$. The mapping $T$ is defined on the whole of  $[0, \frac{1}{1-a}]$. Equation \eqref{eq: weightedPF} now means that the killed transition operator $P$ is  a {\it weighted transfer} (or the {\it Ruelle}) {\it operator} associated with $T$, where the weight  takes two values $p$ and $q$. The weighted transfer operators are considered e.g.\ in the book by Baladi~\cite{Baladi}. The important particular case is the standard {\it transfer} (or the {\it Perron--Frobenius}) {\it operator} $P_T$, defined by $P_Tf(x)=\sum_{y \in T^{-1}(x)} f(y)/|T'(y)|$ for $x \in [0,3]$.
If $f \ge 0 $, then $P_Tf$ is the density of the measure with density $f$ pushed forward by $T$.

Assume now that $p=\frac12$. Then by \eqref{eq: weightedPF}, we have
\[
Pf(x)=\frac34 P_T f(x), \qquad 0 \le x <3, 
\]
and thus $P=\frac34 P_T$ as operators on $L^1([0,3])$. Since $T$ is piecewise linear and expanding (i.e.\ $\essinf_{0 \le x \le 3} |T'(x)|>1$), the operator $P_T$ is quasi-compact  on the space of functions of bounded variation on $[0,3]$ (with a.e.\ equal functions identified); see e.g.\ Boyarsky and G\'ora~\cite[Theorem~7.2.1]{BG97}. The leading eigenvalue of $P_T$ is simple and equals $1$, and this gives $\lambda_{2/3}(1/2)=3/4$ and also implies \eqref{eq: main equiv}, together with an alternative way of finding the function $V$. To explain this, note that the eigenfunction $h$ corresponding to the leading eigenvalue $1$ of $P_T$ is the invariant density under the transformation $T$. Its scaled version $\widehat h(x)=3h(3x)$ on $[0,1]$  is invariant under $\widehat T_{2/3}$. This density was found explicitly by Parry~\cite{Parry64}, and it is given (up to a multiplicative factor) by formula~\eqref{eq: Parry} below. Simplifying this formula and rescaling back to $[0,3]$, we can  recover our function $V$ given in~\eqref{eq: V}  for $a=\frac23$ and $p=\frac12$. In this case $h=c V/3$ is a probability density. This normalization can be seen by integrating $V$ and combining formula~\eqref{eq: const} for $c$ with the equations $\frac13 T^k_{2/3}(0) = 1- \sum_{n \ge 0}^\infty (\frac23)^{n+1} \delta_{k+n} $ for integer $k \ge 0$, which in turn follow from~\eqref{eq: T mod 1} and \eqref{eq: expansion Parry}. 

We also have an alternative proof of our result that the quasi-invariant distribution $\nu_{2/3}$ is uniform when $p=\frac12$. Indeed, it is easy to show that the density of $\nu_{2/3}$ is an eigenfunction of the {\it composition} (or the {\it Koopman}) {\it operator}  that is dual to $P_T$. The constant density is its eigenfunction and there are no other eigenfunctions by \cite[Theorem~3.5.2]{BG97}. Therefore, the density of $\nu_{2/3}$ is~constant.



\section{Existence of a positive solution to equation~\eqref{eq: lambda}} \label{sec: lambda}
Consider the function
\[
R_a(\lambda)=\sum_{k=0}^{\varkappa_a} \delta_k (p/\lambda)^{k+1} (q/p)^{L_k},
\]
which represents the left-hand side of~\eqref{eq: lambda}. We need to show that the equation $R_a(\lambda)=1$  has exactly one positive solution $\lambda=\lambda_a$. This equation arises naturally in the next section. We will establish existence together with the following bounds. 

\begin{proposition} \label{prop: p < 1/2}
Let $a \in (\frac12, \frac23]$. Then for every $p \in (0, 1)$, equation \eqref{eq: lambda} has a unique positive solution $\lambda_a  =\lambda_a(p)$. Moreover, there exists a constant $C=C_a>0$ such that  for every $p \in (0, 1)$, we have
\begin{equation} \label{eq: > lambda >}
(p/a) \max \big (1,  (q/p)^C\big )\ge \lambda_a > p \max \big (1,  (q/p)^C\big )
\end{equation}
and 
\begin{equation} \label{eq: L_k - L_m <}
L_n(x) - L_k(x) \le C (n-k+1) + 1
\end{equation}
for all $x \in [0, 1/(1-a)]$ and all integers $0 \le k \le n \le \varkappa_a(x)$.
The first inequality in \eqref{eq: > lambda >} is strict unless $a=\frac23$ and $p=\frac12$, in which case $\lambda_{2/3}(1/2) = 3/4$.
\end{proposition}
The choice of $C$ is rather explicit. For example, it is easy to see from the proof that we can take $C=1/3$ when $a=2/3$. This gives an upper bound for the limiting frequency of $1$'s in the $\frac32$-expansions, and this bound is sharp (as there are cycles of length $3$).  The set of all limit frequencies in $\beta$-expansions is described in Example 47 by Boyland et.\ al~\cite{BCH16}.

It takes us much more effort to prove Proposition~\ref{prop: p < 1/2} when $\varkappa_a = \infty$ and $q > p$. The key to our proof is the following assertion, where we essentially compare the (inverse) number of $0$'s among the first digits of the $(\frac{1}{a}, \frac12 )$-expansions of $0$ and of an arbitrary $x \in [0,1)$. This is in the spirit of Theorem~1 by Parry~\cite{Parry60}. To state our result, denote by $\sigma_a(x)$  the total number of returns to $[0,1)$ of the trajectory of $x$ under $T_a$ and by $t_k(x)$ the corresponding return times:
\[
\sigma_a(x)= \sum_{n=1}^{\varkappa_a(x)} \ind \{T^n_a (x)<1\}
\]
and
\[
t_k(x)= \inf \{n> t_{k-1}(x) : T^n_a (x)<1\}, \qquad 1 \le k < \sigma_a(x) +1,
\]
where $t_0(x)=0$. Put $\sigma_a=\sigma_a(0)$ and $t_k=t_k(0)$. 

\begin{lemma} \label{lem: frequency}
Let $a \in (\frac12, \frac23]$, and denote 
\[d_n=
\begin{cases}
t_1, & n=0, \\
\min \big\{\frac{t_1+1}{1}, \ldots, \frac{t_n +1}{n}, \frac{t_{n+1}}{n+1} \big \}, & 1 \le n < \sigma_a, \\
\min \big\{\frac{t_1+1}{1}, \ldots, \frac{t_n +1}{n}\big \}, & n =\sigma_a \text{ and } \sigma_a < \infty.\\
\end{cases}
\] 
Then for any $x \in [0,1)$ and integers $k, n$ such that $0 \le n \le \sigma_a$ and $0 \le k \le \sigma_a(x)$, we have $d_n \ge t_1$ and
\begin{equation} \label{eq: t_k >}
t_k(x) \ge k d_n - \ind\{n \neq 0\}.
\end{equation}
\end{lemma}

\begin{proof}
Assume that $x < \frac{2a-1}{1-a}$, otherwise \eqref{eq: t_k >} is trivial. We prove by induction in $k$ for each fixed $n$. In the basis case $k=0$, inequality~\eqref{eq: t_k >} is trivial. 

We now prove the step of induction. Put 
\[
m= \max\{0 \le i \le \min(k, n): t_j(x)=t_j\text{ for all }j\le i\}
\]
and
\[
m'=\min(m, \sigma_a-1).
\]
If $m=k$, then \eqref{eq: t_k >} immediately follows from the definition of $d_n$. From now on, we assume that $m < k$, and we claim that
\begin{equation} \label{eq: 2nd induction}
T_a^{t_i}(0) \le T_a^{t_i}(x)=T_a^{t_i(x)}(x)<\frac{2a-1}{1-a}, \qquad 0 \le i \le m'.
\end{equation}
We prove this by induction, assuming that $n \ge 1$, otherwise there is noting to prove. In the basis case $i=0$ this is true by the assumption. If we already established this for all $i \le j$ for some $j < m'$, then $ 1 \le T_a^{t_j+1}(0) \le T_a^{t_j +1}(x)$ since $T_a$ increases on the interval $[0, \frac{2a-1}{1-a})$. Then $T_a^{\ell}(0) \le T_a^{\ell}(x)$ for all $t_j+1 \le \ell \le t_{j+1}$ since $T_a$ increases on $[1, \frac{1}{1-a}]$. Moreover, $T_a^{t_{j+1}}(x) = T_a^{t_{j+1}(x)}(x) < \frac{2a-1}{1-a}$ since $j +1  \le m' < \min(\sigma_a (x), \sigma_a)$. This establishes \eqref{eq: 2nd induction}. 

Furthermore, by the same reasoning, it follows from \eqref{eq: 2nd induction} that
$t_{m'+1}(x) \ge t_{m'+1}(0)$ and 
\begin{equation} \label{eq: final}
T_a^{t_{m'+1}}(0)  <1  \ \text{ and } \ T_a^{t_{m'+1}}(0) \le T_a^{t_{m'+1}}(x) .
\end{equation}
Consider three cases. If $m'=n$, then $n< \sigma_a$, and by the definition of $d_n$,
\[
t_{m'+1}(x) \ge t_{m'+1}(0) \ge (m'+1) d_n. 
\]
If $m'<n$ and $m'=m$, then $m < \sigma_a$ and $m'=m<n$,  and $t_{m'+1}(x) = t_{m'+1}(0)$ is impossible by the definition of $m$. Hence by the definition of $d_n$,
\[
t_{m'+1}(x) \ge t_{m'+1}(0) +1 \ge (m'+1) d_n. 
\]
The remaining case $m'<n$ and $m'=\sigma_a-1<m$ is impossible, because otherwise $m=n = \sigma_a$ and it follows from the definition of $m$ and 
inequalities \eqref{eq: final} that $\frac{2a-1}{1-a} < T_a^{t_m}(x) <1$. Thus, $\sigma_a(x)=m$, contradicting our earlier assumption that $m<k$.

Thus, in all possible cases we have $t_{m'+1}(x) \ge (m'+1) d_n$,  where $1 \le m' +1 \le k$. Consequently, 
\[
t_k(x) = t_{m'+1}(x)+ t_{k-m'-1}(T_a^{t_{m'+1}(x)}(x)) \ge (m'+1) d_n + t_{k-m'-1}(T_a^{t_{m'+1}(x)}(x)).
\]
Therefore, if we already proved \eqref{eq: t_k >} for all $k \le j$ some some $0 \le j < \sigma_a(x)$, then 
\[
t_{j+1}(x) \ge  (m'+1) d_n + (j+1-m'-1) d_n - \ind \{ n \neq 0\}= (j+1)d_n  - \ind \{ n \neq 0\}.
\] 
This finishes the proof of \eqref{eq: t_k >}. In particular, \eqref{eq: t_k >} implies that $t_k \ge k d_0 = k t_1$ for all integer $1 \le k \le \sigma_a(x)$. This in turn implies that $d_n \ge t_1$  for  all integer $0 \le n \le \sigma_a $, as needed. 
\end{proof}

We will also need the following result, where we list properties of the iterates of $T_a$ for a fixed $a$ and describe their domains, i.e.\ the sets $\{x: \varkappa_a(x)\ge k\} $.

\begin{lemma} \label{lem: dom T^n}
Let $k \ge 1$ be an integer. Define
\[
G_k= \bigcup_{n=0}^k T_a^{-n}(0) \cap \{x: \varkappa_a(x)\ge k\} \quad \text{and} \quad  D_k= \bigcup_{n=0}^{k-1} T_a^{-n} \left(\frac{2a-1}{1-a} \right) \cup \left \{ \frac{1}{1-a} \right \},
\]
and put $g_k(x)=\max\{y \in G_k: y \le x \}$.

a) For any $a \in (\frac12, \frac23)$, the following is true. The set $\{x: \varkappa_a(x)\ge k\} $ is a union of finitely many disjoint intervals $\{ [g_k(y),y]: y \in D_k\}$. The set of the left endpoints of these intervals is $G_k$. On each of these intervals, the functions $\delta_0(x), \delta_1(x), \ldots, \delta_{k-1}(x)$ are constant  and $T_a^1(x), \ldots, T_a^k(x)$ are continuous and strictly increasing.

b) The assertions of Part a) remain valid for $a=\frac23$ if we substitute the set of intervals by $\{ [g_k(y-),y): y \in D_k, y \neq 3\} \cup \{[g_k(3), 3]\}$.
\end{lemma}

\begin{proof}
a) We argue by induction. In the basis case $k=1$ we have $D_1=\{\frac{2a-1}{1-a},\frac{1}{1-a}\}$, hence the set $\{x: \varkappa_a(x)\ge 1\} $, which is $[0,\frac{1}{1-a}] \setminus I_a$, is of the form stated. The claims concerning $\delta_0(x)$ and $T_a(x)$ are clearly satisfied. 

Assume that the statements are proved for all $1 \le k \le n$. Fix a $y \in D_n$ and consider three cases.

If $T_a^n(g_n(y)) \ge1$, then $\{x  \in [g_n(y),y]: T_a^n(x) \not \in I_a \}=[g_n(y),y]$, and this is a maximal  interval contained in $\text{dom}(T^{n+1}_a)$. We have $g_n(y) \in G_{n+1}$ by $G_n  \cap \{x: \varkappa_a(x) \ge n+1\} \subset G_{n+1}$ and the fact that $\varkappa_a(y)\ge n+1$, and also $y \in D_{n+1}$ by $D_n \subset D_{n+1}$. The interval $(g_n(y),y)$ contains no points of $G_{n+1}$ since it contains no points of $G_n$ and $T_a^{n+1}(y) \neq 0$  for every $x\in (g_n(y),y)$ by  $T_a^n(x)>T_a^n(g_n(y)) \ge 1$. Hence
$g_{n+1}(y)=g_n(y)$.

If $\frac{2a-1}{1-a} < T_a^n(g_n(y)) < 1$, then by $T_a^n(y) = \frac{1}{1-a} >1$, there is a unique $z \in [g_n(y),y]$ such that $T_a^n(z)=1$.
Then $z \in G_{n+1}$ by $T_a^{n+1}(z)=T_a(1)=0$, and $ y \in D_{n+1}$. By the same argument as in the previous case, $[z,y]$ is the maximal subinterval of $[g_n(y),y]$ contained in $\text{dom}(T^{n+1}_a)$, hence $g_{n+1}(y)=z$.

Lastly, if $0 \le T_a^n(g_n(y)) \le \frac{2a-1}{1-a}$, then by $T_a^n(y) >1$, there exist unique $z_1, z_2 \in [g_n(y),y]$ such that $T_a^n(z_1) = \frac{2a-1}{1-a}$ and $T_a^n(z_2)=1$. Similarly to the previous cases, $[g_n(y),z_1]$ and $[z_2, y]$ are  the maximal subintervals of $[g_n(y),y]$ contained in $\text{dom}(T^{n+1}_a)$, and there are no other ones. It holds $z_1, y \in D_{n+1}$ and $g_n(y), z_2 \in D_{n+1}$. We also have $g_{n+1}(z_1)=g_n(y)$, $g_{n+1}(y)=z_2$. 

The above consideration of the three cases combined with the representation
\[
\{x: \varkappa_a(x)\ge n+1\}  = \{x: \varkappa_a(x)\ge n, T_a^n(x) \not \in I_a \} =\bigcup_{z \in D_n} \{x  \in [g_n(z),z]: T_a^n(x) \not \in I_a \},
\]
imply that  the set $\{x: \varkappa_a(x)\ge n+1\}$ is a union of the intervals $\{[g_{n+1}(z),z]: z \in D_{n+1}\}$  and  moreover, we have
\[
\{g_{n+1}(z): z \in D_{n+1}\}=\{g_n(z): z \in D_n, \varkappa_a(z) \ge n+1\} \cup T_a^{-n}(1).
\] 
Hence $G_{n+1} = \{g_{n+1}(z): z \in D_{n+1}\}$ by the assumption of induction and the fact that $T_a^{-n}(1)=T_a^{-(n+1)}(0)$. The set $D_{n+1}$ is finite because so is $D_n$. The intervals $\{[g_{n+1}(z),z]: z \in D_{n+1}\}$ are disjoint as subintervals of disjoint intervals $\{ [g_n(z),z]: z \in D_n\}$.

Furthermore, in either of the three cases,  the function $T_a^{n+1}$ is continuous and strictly increasing on $[g_{n+1}(y),y]$ as a composition of $T_a$ and $T_a^n$, which have these properties on $[1, \frac{1}{1-a}]$ and $[g_n(y),y]$, respectively. Then $\delta_n$ is constant on $[g_{n+1}(y),y]$. In the third case, $T_a^{n+1}$ and $\delta_n$ also have these respective properties on $[g_{n+1}(z_1), z_1]$ because on this interval $T_a^n$ does not exceed $\frac{2a-1}{1-a}$ and $T_a$ is continuous and strictly increasing on $[0, \frac{2a-1}{1-a}]$. 

b) The proof for $a=\frac23$ is analogous. It suffices to replace throughout $[g_n(y), y]$ by $[g_n(y-), y)$ for all $y \in D_n \setminus \{3\}$ and use that $T_{2/3}^n(y-)=3$ instead of $T^n_a(y) = \frac{1}{1-a}$.
\end{proof}

We are now ready to prove the main result of the section.

\begin{proof}[Proof of Proposition~\ref{prop: p < 1/2}]
Let us show that the equation $R_a(\lambda)=1$ has exactly one positive solution $\lambda=\lambda_a$. This is evident when $\varkappa_a$ is finite because in this case $R_a$ is a continuous strictly decreasing function on $(0, \infty)$ which satisfies $R_a(p)>1$ by $\delta_0=1$ and $R_a(+\infty) =0$. Then there is a unique solution $\lambda_a$, which satisfies $\lambda_a > p$. Moreover, this reasoning applies easily when $\varkappa_a = \infty$ and $q \le p$. In this case $R_a$ is finite on $(p, \infty)$ and satisfies $R_a(p)=R_a(p+)$ by the monotone convergence theorem, hence $1<R_a(p')< \infty$ for some $p'>p$. 

In the remaining case where $\varkappa_a=\infty$ and $p<1/2$, put  $r=q/p$. We will use the following representation
\[
R_a(\lambda) = \sum_{k=0}^{\varkappa_a} \delta_k (p/\lambda)^{k+1} (q/p)^{L_k}= \sum_{k=0}^{\sigma_a} (p/\lambda)^{t_k+1} r^k.
\]

We first assume that $\sigma_a = \infty$. Denote $d=\sup_{n \ge 0} d_n$ and $\theta=\liminf_{n \to \infty} t_n/n $. By the Cauchy--Hadamard formula, $R_a(\lambda)$ is finite for $\lambda >p$ when 
\[
r < \frac{1}{\limsup_{k \to \infty} (p/\lambda)^{(t_k+1)/k}} =\frac{1}{(p/\lambda)^\theta},
\]
 i.e.\ for $\lambda > p r^{1/\theta}$. Notice that $\theta \ge d \ge t_1$ by  Lemma~\ref{lem: frequency}.

If $d_n= \frac{t_{n+1}}{n+1}$ for all $n \ge 1$, then $d \ge t_n/n$ for all $n \ge 1$. Therefore, for $C=1/d$, we have $n \ge C t_n$ for all $n \ge 1$. Hence by $r>1$,
\[
R_a(pr^C) =\sum_{k=0}^\infty r^{-(t_k+1)C+k} \ge r^{-C} \sum_{k=0}^\infty r^0  =\infty.
\]
If there exists an $n_0 \ge 1$ such that $d_{n_0}< \frac{t_{n_0+1}}{n_0+1}$. Then $d_{n_0} = \frac{t_{k_0} +1}{k_0}$ for some $1 \le k_0 \le n_0$, hence for $C=1/d_{n_0}$,
\[
R_a(p r^C) >  r^{-(t_{k_0}+1)C+k_0} = r^0=1.
\]

In either case, we have $C \ge 1 /\theta$. Therefore, $1<R_a(p')< \infty$ for some $p'>p r^C$ by the same argument as above for $p\ge 1/2$, and there exists a unique solution $\lambda_a$ to $R_a(\lambda)=1$, which satisfies $\lambda_a > p' > p r^C$. 

If $\sigma_a$ is finite, then we have $R_a(p r^C)>1$ for $C=1/d_{\sigma_a}$ and $\lambda_a > p r^C$,  exactly as above.

We now prove inequality~\eqref{eq: L_k - L_m <}. Fix an $x \in [0, 1/(1-a)]$. 

First assume that $\sigma_a(x) = \infty$. Fix a $k \ge 0$ and put $x'= T_a^k(x)$ if $T_a^k(x)<1$ and $x'=T_a^{t_1(T_a^k(x))}$ otherwise. It follows from Lemma~\ref{lem: frequency} that $t_{\lceil Ck' \rceil}(x') \ge \lceil kC' \rceil /C -1\ge k'-1$ for all integer $k' \ge 1$ such that $1 \le \lceil Ck' \rceil \le \sigma_a(x')$. Substituting $n-k=k'-1$, we get
\[
L_n(x) - L_k(x)  = L_{n-k}(T_a^k(x)) \le L_{n-k}(x') \le L_{t_{\lceil Ck'\rceil}(x')}(x')  = \lceil C(n-k+1)\rceil
\]
for {\it all} integer $n \ge k$ since  $\sigma_a(x')=\sigma_a(x)=\infty$. Then  
\[
L_n(x) - L_k(x)  \le C(n-k+1)
\] 
because the left-hand side is integer. This proves inequality~\eqref{eq: L_k - L_m <} when $\sigma_a(x) = \infty$.

We now assume that $\sigma_a(x) <\infty$, which is possible only when $a <\frac23$. We argue by reduction to the previous case. Assume that $\varkappa_a(x)\ge 1$, otherwise there is nothing to prove. Fix an integer $1 \le n \le \varkappa_a(x)$. By Lemma~\ref{lem: dom T^n}.a, there exists a unique $y \in D_n$ such that $x \in [g_n(y), y]$. 

We first assume that $g_n(y) \neq y$. Since the functions $z \mapsto L_k(z)$ for $0 \le k \le n$ are constant  on $[g_n(y), y]$ by Lemma~\ref{lem: dom T^n}.a, inequality~\eqref{eq: L_k - L_m <} follows instantly from the result in the previous case if we prove that there is an $x' \in [g_n(y), y]$ such that $\sigma_a(x') =\infty$. We will use that there exist periodic orbits of $T_a$ starting arbitrarily close to the left of $y$. Consider two cases. If $T_a^m(y) = \frac{2a-1}{1-a}$ for a unique $0 \le m \le n-1$, then $T_a^m < \frac{2a-1}{1-a}$ on $[g_n(y), y)$ by Lemma~\ref{lem: dom T^n}.a. By piecewise continuity of $T_a$, the equation $z = T_a^{t_1(z)}(z)$ admits infinitely many solutions, which accumulate to $\frac{2a-1}{1-a}$ from the left. Pick any solution $z' \in [ T_a^m(g_n(y)), \frac{2a-1}{1-a})$, and then take the unique $x' \in [g_n(y), y]$ that solves $T_a^m(x')=z'$. In the second case when $y = \frac{1}{1-a}$, we argue similarly and take any $x' \ge g_n(y)$ that satisfies $x' = T_a^{t_1(x')+1}(x')$.

We now assume that $g_n(y) =y$. This is possible only when $G_n$ intersects $D_n$, that is when the orbit of zero hits $\frac{2a-1}{1-a}$. Then $T_a^{k_1}(x)=0$ and $T_a^{k_2}(x) = \frac{2a-1}{1-a}$ for the unique $k_1$, $k_2$ such that $0 \le k_1 < k_2 \le n-1$. Thus, $x \in G_{k_2}$ and by Lemma~\ref{lem: dom T^n}.a, there exists a unique $y' \in D_{k_2}$ such that $g_{k_2}(y')=x$. In other words, $\varkappa_a(z) = k_2$ for every $z \in (g_{k_2}(y'), y')$. Clearly, $g_{k_2}(y') \neq y'$, and we can apply the result of the previous case for all $k$ satisfying $0 \le k \le k_2$. Therefore, since the functions $z \mapsto L_k(z)$ for $0 \le k \le k_2$ are constant  on $[g_{k_2}(y'), y']$,
\[
L_n(x) - L_k(x)= L_{k_2}(y') + 1 - L_k(y') \le C (k_2 - k+1) +1 \le C (n - k+1) + 1.
\]
On the other hand, $L_n(x) - L_k(x)=0$ for all $k$ satisfying $k_2 +1 \le k \le n$. This finishes the proof of~\eqref{eq: L_k - L_m <}. Notice that in the last case $\delta_n (x)=0$ for all $n \ge k_2+1$, therefore we actually proved that 
\begin{equation} \label{eq: L_n < better}
\delta_n (x) L_n(x) \le C (n+1), \qquad 0 \le n \le \varkappa_a(x).
\end{equation}

Finally, we prove the upper bound in \eqref{eq: > lambda >}. For $p\ge 1/2$, we have $r \le 1$, and we use the $(\frac1a, \frac12)$-expansion~\eqref{eq: expansion Parry} to estimate
\begin{align*}
R_a (p/a) &= \sum_{k=0}^{\varkappa_a} \delta_k a^{k+1} r^{L_k} \le \sum_{k=0}^\infty  (1- d_k(0)) a^{k+1} \\
&=\frac{a}{2(1-a)}+ \sum_{k=0}^\infty \Big (\frac12- d_k(0) \Big)a^{k+1} = \frac{a}{2(1-a)} \le 1.
\end{align*}
Hence $\lambda_a \le p/a$, and both inequalities above turn into equalities only when $a = 2/3$ and $p=1/2$. In this case $R_{2/3}(3/4)=1$ and thus  $\lambda_{2/3}(1/2) = 3/4$. 

For $p<1/2$, we use a similar bound based on \eqref{eq: L_n < better}:
\[
R_a ( p r^C/a ) = \sum_{k=0}^{\varkappa_a} \delta_k a^{k+1} r^{L_k - C (k+1)} \le \sum_{k=0}^\infty  (1- d_k(0)) a^{k+1} \le 1.
\]
The second inequality turns into an equality only in the case $a=2/3$, when the first inequality is strict since the orbit of zero under $T_{2/3}$ is not periodic. Hence $\lambda_a < p r^C/a$.


\end{proof}

\begin{remark} \label{rem: C}
For any $a \in (\frac12, \frac23]$ and $p\in (0, \frac12)$, inequalities~\eqref{eq: > lambda >} and \eqref{eq: L_k - L_m <} are satisfied with 
\[
C= C_a:=\min_{0 \le n \le \sigma_a \wedge (q/p)^{1/t_1}} 1/d_n.
\] 

This follows exactly as in the proof of Proposition~\ref{prop: p < 1/2} except for a slight modification of the argument for  the lower bound in~\eqref{eq: > lambda >}. Namely, if there exist integers $1 \le k_0 \le n_0 \le\sigma_a \wedge r^{1/t_1}$ such that $d_{n_0}= \frac{t_{k_0}+1}{k_0}$, then $C \le 1/d_{n_0}$, and we get
\[
R_a(p r^C) \ge R_a(p r^{1/d_{n_0}}) >  r^{-(t_{k_0}+1)C+k_0} =1.
\]
Otherwise, $d_n= \frac{t_{n+1}}{n+1}$ for all $0 \le n \le r^{1/t_1}$ and $\sigma_a >r^{1/t_1}$, hence $n+1 \ge t_{n+1}/d_n \ge C  t_{n+1}$ for such $n$, and we get
\[
R_a(pr^C) =\sum_{n=0}^{\sigma_a} r^{-(t_n+1)C+n} \ge r^{-C} \sum_{n=0}^{[r^{1/t_1}]+1} r^0  \ge r^{-1/d_0} ([r^{1/t_1}]+2)>1.
\]  
\end{remark}

\section{Spectral properties of the killed transition operator on $U$}
\label{sec: P on U}
In this section we establish the tail asympotics of $\tau$, given in \eqref{eq: lambda}. To do this, we first provide a simple representation of the action of the killed transition operator $P$ on the space $U$ of functions defined in \eqref{eq: U}. To this end, let us introduce the following notation.

Denote by $\V[f]$ the total variation of a function $f$ on $[0, 1/(1-a)]$. Then
\begin{equation} \label{eq: BV norm}
\|f\|=|f(0)|+\V[f],
\end{equation}
the norm on the space of functions of bounded variation, is a norm on $U$. 

We can regard functions in $U$ as (right-continuous) distribution functions of finite signed atomic measures supported at $\{T^k_a(0): 0 \le k < \varkappa_a +1\}$, the orbit of zero. 
Then it is easy to see that the mapping $M: U \to \R^{\varkappa_a'+1}$, defined by
\[
(M f)_k = f(T_a^k(0)) - f(T_a^k(0)-), \qquad 0 \le k < \varkappa_a' +1,
\]
where we put $f(0-)=0$, is a bijective isometry between $(U, \| \cdot \|)$ and $(\R^{\varkappa_a'+1}, {\| \cdot \|}_1)$. 

Furthermore, let $A$ be the linear operator on $(\R^{\varkappa_a +1}, {\| \cdot \|}_1)$ defined by
\[
{(Au)}_0=p\sum_{k=0}^{\varkappa_a}\delta_ku_k\ \text{ and }\  
{(Au)}_k=c_{k-1}u_{k-1},\quad 1 \le k < \varkappa_a+1,
\]
where 
\[
\delta_k={\ind}\{T_a^k(0)<1\}\ \text{ and }\ 
c_k=q\delta_k+p(1-\delta_k),
\]
and let $v \in\R^{\varkappa_a +1}$ and $ v^* \in\R^{\varkappa_a' +1}$ be the vectors with coordinates
\[
v_k= (p/\lambda_a)^k (q/p)^{L_k},\quad 0 \le k < \varkappa_a + 1,
\]
and
\begin{equation} \label{eq: v^*}
v_k^*=\sum_{n=k}^{\varkappa_a} \delta_n (p/\lambda_a)^{n-k+1} (q/p)^{L_n-L_k},\quad 0 \le k < \varkappa_a' + 1.
\end{equation}
Notice that $\|v\|_1$ and $\|v^*\|_\infty$ are finite by $\lambda_a>p$ if $p \ge 1/2$ and by \eqref{eq: L_k - L_m <} if $p<1/2$. Notice also that $V(x)=\sum_{k=0}^{\varkappa_a} v_k \ind \{x \ge T^k_a(0)\}$
is exactly the function introduced in \eqref{eq: V}.

The main difficulty is when the orbit of zero under $T_a$ is infinite and aperiodic. In this case the operator $A$ acts on the space of infinite summable sequences $\ell_1$, and the adjoint operator $A^*$ on $\ell_\infty$ is defined by the standard duality $(\cdot, \cdot)$ between $\ell_1$ and $\ell_\infty$. If $\varkappa_a'<\infty$, we use $(\cdot, \cdot)$ to denote the scalar product on $\R^{\varkappa_a'+1}$.

\begin{proposition} \label{prop: A}
For any $a \in (\frac12, \frac23]$, the following is true depending on the type of the orbit of zero under $T_a$.

If the orbit is aperiodic (i.e.\ $\varkappa_a=\varkappa_a'$), then the restriction of the killed transition operator $P$ to $U$ is equivalent to A (i.e.\ $P= M^{-1} A M$). The quantity $\lambda_a$, defined in \eqref{eq: lambda}, is the leading eigenvalue of $A$ (i.e.\ every other eigenvalue $\lambda$ satisfies $|\lambda|< \lambda_a$), 
and $v$ and $v^*$ are the eigenvectors of $A$ and $A^*$  corresponding to $\lambda_a$, respectively. 
Moreover, there exist constants $C_1>0$ and $\gamma \in (0,1)$ such that for every $u \in (\R^{\varkappa_a'+1}, {\| \cdot \|}_1)$ and integer $n \ge 0$,
\begin{equation} \label{eq: A^n <}
\left\|\lambda_a^{-n} A^nu - \frac{(u, v^*)}{(v,v^*)} v \right\|_1
\le C_1 \gamma^n{\|u\|}_1.
\end{equation}

If the orbit  is eventually periodic, then the restriction of $P$ to $U$ is equivalent to the linear  operator $\widehat A$ on $\R^{\varkappa_a' +1}$ defined by
\[
{(\widehat A u)}_k={(Au)}_k + c_{\varkappa_a'} u_{\varkappa_a'} \ind\{k = k_0\}, \qquad 0 \le k \le \varkappa_a',
\]
where $k_0$ is the smallest integer such that $T_a^{k_0}(0)=T_a^{\varkappa_a'+1}(0)$. 
Then $\lambda_a$ is leading eigenvalue of $\widehat A$, and $Mv$ and $v^*$ are the eigenvectors of $\widehat A$ and ${\widehat A}^*$  corresponding to $\lambda_a$, respectively. 
Relation~\eqref{eq: A^n <} remains true when $A$ is replaced by $\widehat{A}$ and $v$ is replaced by $MV$.
\end{proposition}

\begin{remark}\label{rem:compact}
Recall that a bounded linear operator $Q$ on a Banach space is called {\it quasi-compact} if there exists a compact operator $Q_c$ such that $\rho(Q) > \rho(Q-Q_c)$, where $\rho$ stands for the spectral radius. Assertion~\eqref{eq: A^n <} implies that $A$ and $\widehat A$ are quasi-compact.  
This implies that the operator $P$ is quasi-compact on $U$ for all $a \in (\frac12, \frac23]$. 
\end{remark}



We postpone the proof of Proposition~\ref{prop: A} and use this result to establish the tail asymptotics~\eqref{eq: main equiv} with the constant 
\begin{equation} \label{eq: const}
c=\left (1 + \sum_{k=1}^{\varkappa_a} k \delta_k (p/\lambda_a)^{k+1} (q/p)^{L_k} \right)^{-1}.
\end{equation}

We first assume that the orbit of zero is aperiodic, i.e.\ $\varkappa_a=\varkappa_a'$. 
We  compute the constant factor $(v,v^*)$ in \eqref{eq: A^n <} using the equality $\varkappa_a=\varkappa_a'$ and the definition~\eqref{eq: lambda} of $\lambda_a$:
\begin{align*}
(v, v^*)&= \sum_{k=0}^{\varkappa_a'}  v_k v_k^* = \sum_{k=0}^{\varkappa_a} \sum_{n=k}^{\varkappa_a} \delta_n (p/\lambda_a)^{n+1} (q/p)^{L_n}  = \sum_{n=0}^{\varkappa_a} (n+1) \delta_n (p/\lambda_a)^{n+1} (q/p)^{L_n}  = \frac{1}{c}.
\end{align*}

Now assume that the orbit of zero is eventually periodic. Relation \eqref{eq: A^n <} is valid with $A$ replaced by $\widehat{A}$ and $v$ replaced by $MV$. Let us compute the constant factor $(Mv,v^*)$ there. Put $n_0=\varkappa'-k_0+1$ and let $\bar v_k^*$ be defined as in~\eqref{eq: v^*} for {\it all} $k \ge 0$. The sequence $\{\delta_k\}$ is eventually periodic, hence so is $\{L_k\}$, and hence so is $\{\bar v_k^*\}$. Therefore, using equality~\eqref{eq: M explicit},
\[
(MV, v^*) = \sum_{k=0}^{k_0-1} v_k v_k^* + \sum_{k=k_0}^{\varkappa_a'} \sum_{m=0}^{\infty} v_{k+mn_0} v_k^* =
 \sum_{k=0}^{k_0-1} v_k \bar v_k^* + \sum_{k=k_0}^{\varkappa_a'} \sum_{m=0}^{\infty} v_{k+mn_0} \bar v_{k+mn_0}^*,
\]
hence $(MV, v^*)=(v, \bar v^*) = 1/c$ as in the aperiodic case.

To state \eqref{eq: A^n <}  in terms of the operator $P$, use that $P^nf = M^{-1} A^n Mf $. We have $M^{-1} v =V$ if the orbit of zero is aperiodic, and $M^{-1}(M V) = V$ if it is eventually periodic. Therefore,~\eqref{eq: A^n <} directly implies that for every $a \in (\frac12, \frac23]$, we have
\begin{equation}
\label{eq:PF-bound}
\|\lambda_a^{-n}P^nf - c (Mf,v^*) V \|\le C_1 \gamma^n\|f\|,\quad f\in U.
\end{equation}
Clearly, $\lambda_a$ is the leading eigenvalue of $P$ and the corresponding eigenfunction is $V$. Since 
\begin{equation} \label{eq: norms}
|f(x)| \le |f(0)| + |f(x)-f(0)| \le \|f\| 
\end{equation} 
for every $x \in [0, \tfrac{1}{1-a}]$, it follows from \eqref{eq:PF-bound} that
\begin{align}
\label{eq:PF}
\sup_{0 \le x \le \frac{1}{1-a}} \Big |\lambda_a^{-n}P^nf(x)- c (Mf,v^*) V(x) \Big |\le C_1 \gamma^n\|f\|,\quad f\in U.
\end{align}

Finally, we determine the tail asymptotics for $\tau$. We have $\pr_x(\tau>n) = P^n \ind(x)$ and the constant function $\ind(x)\equiv 1$ belongs to $U$. Since $M \ind =(1, 0, 0, \ldots)$, we get $(M \ind, v^*)=1$ by~\eqref{eq: v^*}, and the asymptotic relation~\eqref{eq: main equiv} follows from~\eqref{eq:PF}.


\begin{proof}[Proof of Proposition~\ref{prop: A}]
In the proof we shall omit the subscript $a$ in $T_a$, $\varkappa_a$, and $\varkappa_a'$.

For every function $f \in U$, which is of the form $f(x)=\sum_{k=0}^\varkappa u_k{\ind}\{x\ge T^k(0)\}$, one has
\begin{align*}
Pf(x)
&=pf(ax+1)+q{\ind}\{x\ge1/a\}f(ax-1)\\
&=p\sum_{k=0}^\varkappa u_k{\ind}\{ax+1\ge T^k(0)\}
+q{\ind}\{x\ge1/a\}\sum_{k=0}^\varkappa u_k{\ind}\{ax-1\ge T^k(0)\}\\
&=p\sum_{k=0}^\varkappa u_k{\ind}\{x\ge (T^k(0)-1)/a\}
+q{\ind}\{x\ge1/a\}\sum_{k=0}^\varkappa u_k{\ind}\{x\ge(T^k(0)+1)/a\}.
\end{align*}

Consider the first sum in the last line. We notice that for every $k<\varkappa$, it is true that
\begin{equation*}
\text{if }\ T^k(0)<1,\ \text{ then }\  {\ind}\{x\ge (T^k(0)-1)/a\}=1
\end{equation*}
and 
\begin{equation*}
\text{if }\ T^k(0)\ge1, \ \text{ then }\  
{\ind}\{x\ge (T^k(0)-1)/a\}={\ind}\{x\ge T^{k+1}(0)\}.
\end{equation*}
If $\varkappa$ is finite, then $T^\varkappa(0)\in (\frac{2a-1}{1-a},1)$ and therefore, $\delta_\varkappa =1$ and $(T^\varkappa(0)-1)/a<0$. Hence
\begin{equation}
\label{eq:sum1}
\sum_{k=0}^{\varkappa} u_k{\ind}\{x\ge (T^k(0)-1)/a\}
=\sum_{k=0}^{\varkappa} u_k\delta_k
+\sum_{k=0}^{\varkappa-1} u_k(1-\delta_k){\ind}\{x\ge T^{k+1}(0)\}.
\end{equation}

To rewrite the second sum, we notice that for every $k<\varkappa$,  it is true that
\begin{equation*}
\text{if }\ T^k(0)<1, \ \text{ then }\  
{\ind}\{x\ge (T^k(0)+1)/a\}={\ind}\{x\ge T^{k+1}(0)\}
\end{equation*}
and
\begin{equation*}
\text{if }\ T^k(0)\ge1, \ \text{ then }\ (T^k(0)+1)/a>1/(1-a)
\ \text{ and }\ {\ind}\{x\ge (T^k(0)+1)/a\}=0,
\end{equation*}
because in the case $a =2/3$, where $2/a=1/(1-a)$, we have $T^k(0) \neq 1$; see Section~\ref{sec:Operator}.
If $\varkappa$ is finite, then $T^\varkappa(0)\in (\frac{2a-1}{1-a},1)$ and therefore, $(T^\varkappa(0)+1)/a>1/(1-a)$. Hence
\begin{equation}
\label{eq:sum2}
{\ind}\{x\ge1/a\} \sum_{k=0}^{\varkappa} u_k{\ind}\{x\ge(T^k(0)+1)/a\}
=\sum_{k=0}^{\varkappa-1} u_k\delta_k {\ind}\{x\ge T^{k+1}(0)\}.
\end{equation}

Putting \eqref{eq:sum1} and \eqref{eq:sum2} together and using that $c_k=q\delta_k+p(1-\delta_k)$, we conclude that 
\begin{equation} \label{eq: Pf =}
Pf(x)
=\left(p\sum_{k=0}^\varkappa \delta_ku_k\right){\ind}\{x\ge T^{0}(0)\}
+\sum_{k=1}^\varkappa c_{k-1}u_{k-1}{\ind}\{x\ge T^{k}(0)\}.
\end{equation}

We now consider the two cases separately.

\underline{The orbit of zero is aperiodic}. Comparing the coefficients at ${\ind}\{x\ge T^{k}(0)\}$ in the equalities $f(x)=\sum_{k=0}^\varkappa u_k{\ind}\{x\ge T^k(0)\}$ and \eqref{eq: Pf =}, we see that the restriction of $P$ to the space $U$ is equivalent to the linear operator $A$ on $(\R^{\varkappa_a +1}, {\| \cdot \|}_1)$.

Consider  the eigenvalue problem $Au=\lambda u$. By the definition of $A$,
$$
p\sum_{k=0}^\varkappa \delta_ku_k=\lambda u_0
\quad\text{and}\quad 
c_{k-1}u_{k-1}=\lambda u_k, \quad  1 \le  k <  \varkappa +1.
$$
These equations imply that 
\begin{equation} \label{eq: eigenvector A}
u_k= u_0 \frac{c_0 c_1 \cdot \ldots \cdot c_{k-1}}{\lambda^k},\quad 1 \le k <  \varkappa +1,
\end{equation}
Since $\delta_0=1$, we see that there is a nontrivial solution only if $\lambda$ satisfies the equation
\begin{equation} \label{eq: eigenvalue A}
\lambda= p + p\sum_{k=1}^\varkappa \delta_k  \frac{c_0 c_1 \cdot \ldots \cdot c_{k-1}}{\lambda^k}.
\end{equation}
Recalling that $L_k = \sum_{i=0}^{k-1} \delta_i$ and $c_i=q\delta_i+p(1-\delta_i)$, we get
\begin{equation} \label{eq: c product}
 c_0 c_1 \cdot \ldots \cdot c_{k-1}= p^k (q/p)^{L_k}.
\end{equation}
Hence \eqref{eq: eigenvalue A} is equivalent to equation \eqref{eq: lambda}, which has a unique positive solution $\lambda_a >p$, as we showed in Section~\ref{sec: lambda}. 

Notice that in the case when $\varkappa_a = \infty$, the vectors $u$ defined by \eqref{eq: eigenvector A} with $\lambda=\lambda_a$ satisfy $u \in \ell_1$ by $\lambda_a>p$ if $p \ge 1/2$ and by \eqref{eq: L_k - L_m <} if $p<1/2$. Then it follows that the vector $v$, which is proportional to such $u$, is the eigenvector of $A$ corresponding to $\lambda_a$.

We now consider the eigenvalue problem $A^* u^* =\lambda_a u^*$. Invoking the definition of $A$, we get the following coordinate-wise equations:
\begin{equation} \label{eq: left ev}
\begin{array}{rlc}
&\lambda_a u_k^*=p\delta_k u_0^*+c_k u^*_{k+1}, &  0 \le k < \varkappa, \\
&\lambda_a u^*_{\varkappa} = p \delta_{\varkappa} u_0^* & \text{if } \varkappa< \infty.\\
\end{array}
\end{equation}
Setting
$$
V_k=u_k^*\frac{c_0c_1\ldots c_{k-1}}{\lambda_a^k},\quad 0 \le k < \varkappa+1,
$$
where $c_0c_1\ldots c_{-1} = 1$ by convention, we get the equations 
$$
\lambda_a V_k=\frac{c_0c_1\ldots c_{k-1}}{\lambda_a^k}p\delta_kV_0
+\lambda_aV_{k+1},\quad 0 \le k < \varkappa.
$$
Therefore,
$$
V_{k+1}-V_k=-\frac{c_0c_1\ldots c_{k-1}}{\lambda_a^{k+1}}
p\delta_kV_0,\quad 0 \le k < \varkappa,
$$
and using the definition of $\lambda_a$ and \eqref{eq: c product}, we obtain
\begin{align*}
V_{n+1}&=V_0+\sum_{k=0}^n(V_{k+1}-V_k)
=V_0\left(1-\sum_{k=0}^n\delta_k (p / \lambda_a)^{k+1} (q/p)^{L_k}\right)\\
&=V_0 \sum_{k=n+1}^\varkappa \delta_k (p / \lambda_a)^{k+1} (q/p)^{L_k} ,\quad 0 \le n < \varkappa.
\end{align*}
Consequently,
\[
u_{n+1}^*=\frac{\lambda_a^{n+1}}{c_0c_1\ldots c_n}V_{n+1}
= u_0^* \sum_{k=n+1}^\varkappa \delta_k (p / \lambda_a)^{k-n} (q/p)^{L_k-L_{n+1}} ,\quad 0 \le n < \varkappa.
\]
Notice that in the case of infinite $\varkappa$, the vectors $u^*$ satisfy $u^* \in \ell_\infty$ by $\lambda_a>p$ if $p \ge 1/2$ and by \eqref{eq: L_k - L_m <} if $p<1/2$. In the case of finite $\varkappa$, the equation in the second line of \eqref{eq: left ev} is satisfied since $\delta_\varkappa=1$. It follows that the vector $v^*$, which is proportional to such $u^*$, is the eigenvector of $A^*$ corresponding to $\lambda_a$.

Furthermore, we already know that there are no positive eigenvalues of $A$ other than $\lambda = \lambda_a$. Assume that $\lambda \in \C \setminus (0, \infty)$ is a different eigenvalue of $A$.
By \eqref{eq: eigenvalue A}, we have
\begin{equation} \label{eq: other lambdas}
|\lambda|-p<|\lambda-p|
=\left|p\sum_{k=1}^\varkappa \delta_k (p/\lambda)^k (q/p)^{L_k}\right|
\le p\sum_{k=1}^\varkappa \delta_k (p/|\lambda|)^k (q/p)^{L_k}.
\end{equation}
This implies that $R(|\lambda|)>1$, hence $|\lambda|<\lambda_a$. Thus, $\lambda_a$ is the leading eigenvalue of $A$. Denote $\varrho'=0$ if $A$ has no other eigenvalues, otherwise put
\[
\varrho'=\sup \{|\lambda|: \lambda \text{ is an eigenvalue of }A, \lambda \neq \lambda_a\}.
\]

Assume that $\varkappa$ is finite. It is easy to see that the matrix $\mathcal A$ that represents $A$ in the standard basis, has non-negative entries and is irreducible, that is for some $m \in \N$, all entries of  $\mathcal A+{\mathcal A}^2 + \ldots + {\mathcal A}^m$ are strictly positive. Then $\lambda_a$ is a simple root of the characteristic polynomial of $\mathcal A$ and ${\mathcal A}^n/\lambda_a^n \to v v^*/ v^* v$ (with $v^*$ regarded as a row vector) in the operator norm as $n \to \infty$ by the classical Perron--Frobenius theorem; see Meyer~\cite[p.~673 and Eq.~(8.3.10)]{Meyer00}. For the rate of this convergence, the spectral resolution theorem for $A^n$ implies the bound~\eqref{eq: A^n <} for any $\gamma \in (\varrho'/\lambda_a, 1)$; see Eq.~(7.9.9) and the last formula on p.~629 in~\cite{Meyer00}.



From now on we assume that the orbit of zero is infinite (and still aperiodic). Notice that $\varrho'< \lambda_a$, otherwise there is a convergent sequence of eigenvalues $\{\mu_n\}$ such that $|\mu_n| \to |\lambda_a|$. If $p\ge 1/2$, then $R$ is  analytic  on the set $D=\{\lambda \in \C:|\lambda|>p \}$, which contains $\mu_n$ for all $n$ large enough by $\lambda_a >p$. If $p<  1/2$, then it follows from \eqref{eq: L_k - L_m <} that $R$ is  analytic  on the set $D=\{\lambda \in \C:|\lambda|>p (q/p)^C \}$, which contains $\mu_n$ for all $n$ large enough since $\lambda_a >p (q/p)^C$  by \eqref{eq: > lambda >}. In either case it must be that $R \equiv 1 $ on $D$ by $R(\mu_n)\equiv 1$,  which is a contradiction.

We now show that the spectrum of $A$ without the point $\lambda_a$ is contained in the closed centred ball of radius 
$\varrho=\max(\varrho', p, p (q/p)^C)$, where $\varrho< \lambda_a$. Equivalently, the resolvent operator
$(A-\lambda I)^{-1}$ is bounded for every $\lambda\neq \lambda_a$ with $|\lambda|>\varrho$. 

We first prove that the equation $(A-\lambda I)u=w$ has a unique solution for every $w\in\ell_1$. Writing this equation coordinatewise, we have 
\begin{equation} \label{eq: resolvent}
\begin{array}{cl}
&(p-\lambda)u_0+p\sum_{k=1}^\infty\delta_ku_k=w_0,\\
&c_{k-1}u_{k-1}-\lambda u_k=w_k,\quad k\ge1.
\end{array} 
\end{equation}
It is easy to see that
\begin{equation}
\label{eq:uk}
u_k=\frac{c_0c_1\cdots c_{k-1}}{\lambda^k}u_0-
\sum_{j=1}^k w_j\frac{c_jc_{j+1}\cdots c_{k-1}}{\lambda^{k-j+1}}, \quad k \ge 1,
\end{equation}
is the unique solution to the equations in the second line of \eqref{eq: resolvent}. Therefore,
\begin{equation}
\label{eq:u0}
\left(p-\lambda+p\sum_{k=1}^\infty\delta_k\frac{c_0c_1\cdots c_{k-1}}{\lambda^k}\right)u_0
=w_0 + p\sum_{k=1}^\infty\delta_k
\sum_{j=1}^k w_j\frac{c_jc_{j+1}\cdots c_{k-1}}{\lambda^{k-j+1}},
\end{equation}
and thus $u_0$ is defined uniquely since $\lambda$ is not an eigenvalue of $A$, and therefore the factor of $u_0$ on the left-hand side is non-zero by~\eqref{eq: eigenvalue A}. 

We now estimate the norm of $u$. Assume first that $p \ge 1/2$. Then
\begin{align}
\label{eq:A}
\nonumber
\sum_{k=1}^\infty|u_k|
&\le |u_0|\sum_{k=1}^\infty\frac{c_0c_1\cdots c_{k-1}}{|\lambda|^k}
+\sum_{k=1}^\infty\sum_{j=1}^k|w_j|
\frac{c_jc_{j+1}\cdots c_{k-1}}{|\lambda|^{k-j+1}}\\
\nonumber
&\le |u_0|\sum_{k=1}^\infty\left(\frac{p}{|\lambda|}\right)^k
+\sum_{j=1}^\infty|w_j|\sum_{k=j}^\infty\frac{p^{k-j}}{|\lambda|^{k-j+1}}\\
&\le \frac{p|u_0|}{|\lambda|-p}+\frac{1}{|\lambda|-p}\sum_{j=1}^\infty|w_j|
\end{align}
and 
\begin{align}
\label{eq:B}
 \left|\lambda (R(\lambda)-1) \right||u_0|
 &\le |w_0|+p\sum_{k=1}^\infty\delta_k
 \sum_{j=1}^k|w_j|\frac{p^{k-j}}{|\lambda|^{k-j+1}} \le |w_0|+\frac{p}{|\lambda|-p}\sum_{j=1}^\infty|w_j|.
\end{align}
Consequently, for every $\lambda\neq\lambda_a$ with $|\lambda|>\varrho$ there exists a constant $C(\lambda)$ such that 
\begin{equation}
\label{eq:resolvent}
{ \|(A-\lambda I)^{-1}  \|}_1\le C(\lambda) {\|w\|}_1.
\end{equation}


Let us estimate the norm of $u$ in the case when $p < 1/2$. Using the first line in \eqref{eq:A} and recalling that 
$c_0c_1\ldots c_{j-1}=p^jr^{L_j}$ with $r=q/p$, we obtain
\begin{align*}
\sum_{k=1}^\infty|u_k|
&\le |u_0|\sum_{k=1}^\infty\frac{p^kr^{L_k}}{|\lambda|^k}
+\sum_{k=1}^\infty\sum_{j=1}^k\frac{p^{k-j}r^{L_k-L_j}}{|\lambda|^{k-j+1}}|w_j|\\
&=|u_0|\sum_{k=1}^\infty\frac{p^kr^{L_k}}{|\lambda|^k}
+\sum_{j=1}^\infty|w_j|\sum_{k=j}^\infty\frac{p^{k-j}r^{L_k-L_j}}{|\lambda|^{k-j+1}}
\end{align*}
Applying now \eqref{eq: L_k - L_m <}, we get 
\begin{align*}
 \sum_{k=1}^\infty|u_k|
&\le |u_0|r^{C+1}\sum_{k=1}^\infty\frac{p^k r^{Ck}}{|\lambda|^k}
+\frac{r^{C+1}}{|\lambda|}\sum_{j=1}^\infty|w_j|
\sum_{k=j}^\infty\left(\frac{pr^C}{|\lambda|}\right)^{k-j}\\
&\le |u_0|r^{C+1}\frac{pr^C}{|\lambda|-pr^C}
+r^{C+1}{\|w\|}_1\frac{1}{|\lambda|-pr^C}.
\end{align*}
Similar arguments lead to the following estimate:
\begin{align*}
|\lambda(R(\lambda)-1)||u_0|
&\le |w_0|+p\sum_{k=1}^\infty\delta_k\sum_{j=1}^k|w_j|\frac{p^{k-j}r^{L_k-L_j}}{|\lambda|^{k-j+1}}\\
&\le|w_0|+\frac{pr^{C+1}}{|\lambda|}\sum_{j=1}^\infty|w_j|
\sum_{k=j}^\infty\left(\frac{pr^C}{|\lambda|}\right)^{k-j}\\
&\le |w_0|+pr^{C+1}{\|w\|}_1\frac{1}{|\lambda|-pr^C}.
\end{align*}
We know from \eqref{eq: > lambda >} that $\lambda_a>pr^C$, hence \eqref{eq:resolvent} is also valid in the case $p<1/2$. Consequently, for every $p \in (0,1)$, the spectrum of $A$ without the point $\lambda_a$ is contained in the closed centred ball of radius $\varrho<\lambda_a$.

The property of the spectrum shown just above implies that the operator $A$ is quasi-compact. If one shows additionally that $A$ possesses certain additional positivity properties, then one can apply a Krein--Rutman type result, for example Theorems 6 and 7 in Sasser~\cite{Sasser64}. In our particular case we will arrive at the same conclusion by the direct calculations presented below, which are quite standard for the area of quasi-compact operators.

Define the linear operators $Qu=\frac{(u,v^*)}{(v,v^*)}v$ and $Bu=A(u-Qu)$ on $\ell_1$, where $(v,v^*) \ge v_0 v_0^*>0$ by \eqref{eq: v^*}. Clearly, $Q^2=Q$. Since $v$ and $v^*$ are the eigenvectors corresponding to $\lambda_a$, we have $AQu=\lambda_a Qu$ and
$$
QAu=\frac{(Au,v^*)}{(v,v^*)}v=\frac{(u,A^*v^*)}{(v,v^*)}v
=\lambda_aQu.
$$
Thus, $A$ and $Q$ are commuting and $QB=BQ=0$. These properties of the operators $A$, $B$, and $Q$ imply that 
\begin{equation} \label{eq: projection}
A^nu=\lambda_a^n Qu+B^nu,\quad u\in\ell_1,\ n\ge1.
\end{equation}
We claim that the spectral radius $\rho(B)$ of $B$ satisfies
$\rho(B)\le \varrho$. Then the estimate~\eqref{eq: A^n <} for any $\gamma \in (\varrho/\lambda_a, 1)$ follows follow from the representation \eqref{eq: projection} combined with Gelfand's formula $\rho(B)=\lim_{n \to \infty} \|B^n\|^{1/n}$.

To prove the claim, we consider the closed linear subspace
$$
Y=\{u\in\ell_1: (u,v^*)=0\}.
$$
Since $Bu\in Y$ for any $u\in\ell_1$, it suffices to consider the restriction of $B$ to this subspace, which we denote by $B_Y$. In fact, we have $\rho(B)=\rho(B_Y)$, where the inequality $\rho(B)\ge \rho(B_Y)$ is immediate from Gelfand's formula, and the reverse one follows from the said formula by
\[
\rho(B)= \lim_{n\to \infty} \sup_{u\in \ell_1: \|u\| \le 1} \| B^n u\|^{1/n} \le \lim_{n\to \infty} \sup_{v\in Y: \|v\| \le \| B\|} \| B^{n-1} v\|^{1/n} = \rho(B_Y).
\]

We first show that the spectrum of $B_Y$ is a subset of the spectrum of $A$. Let $\lambda$ be any regular value of $A$, i.e.\ $(A-\lambda I)u=w$ is solvable for every 
$w\in\ell_1$. Assume now that $w\in Y$ and let $u_w$ denote the solution to $(A-\lambda I)u=w$. Then
\begin{align*}
0&=(w,v^*)=((A-\lambda I)u_w,v^*)
=(Au_w,v^*)-\lambda(u_w,v^*)\\
&=(u_w,A^*v^*)-\lambda(u_w,v^*)
=(\lambda_a-\lambda)(u_w,v^*).
\end{align*}
Since $\lambda\neq\lambda_a$, we have $(u_w,v^*)=0$. In other words, $u_w\in Y$ and $Qu_w=0$. Consequently,
\[
(B_Y-\lambda I_Y)u_w=(A-\lambda I)u_w=w
\]
and $\lambda$ is regular for $B_Y$. Since the spectrum of $A$ without point $\lambda_a$ is contained in the closed centred ball of radius $\varrho$, it remains to show that $\lambda_a$ is regular for $B_Y$. We need to show that the equation $(B_Y-\lambda_a I_Y)u=w$ for $w \in Y$ has a solution $u \in Y$. Since $B_Y=A$ on $Y$, this is equivalent to showing that $(A- \lambda_a I)u=w$ has a solution $u \in Y$. Therefore, we can take $\lambda=\lambda_a$ in \eqref{eq:uk} and notice that \eqref{eq:u0} is satisfied when $\lambda=\lambda_a$ for any $u_0$ because in this case the right-hand side of~\eqref{eq:u0} is $(w, v^*)$. It remains to determine $u_0$ from the equation $(u,v^*)=0$. 

Alternatively, we can solve $(B_Y-\lambda_a I_Y)u=w$ as follows. Let $u^{(\lambda)}$ denote the solution to
$(B_Y-\lambda I_Y)u=w$ for $\lambda$ such that $0<|\lambda-\lambda_a|<\lambda_a - \varrho$. Then, by \eqref{eq:u0},
\[
\lambda(R(\lambda)-1)u_0^{(\lambda)}
=w_0+p\sum_{k=1}^\infty\delta_k
\sum_{j=1}^k w_j\frac{c_jc_{j+1}\cdots c_{k-1}}{\lambda^{k-j+1}}.
\]
Using the condition $(w,v^*)=0$ and the explicit form of the vector $v^*$, one gets easily
\[
\lambda(R(\lambda)-1)u_0^{(\lambda)}
=p\sum_{k=1}^\infty\delta_k
\sum_{j=1}^k w_j\left(\frac{c_jc_{j+1}\cdots c_{k-1}}{\lambda^{k-j+1}}-\frac{c_jc_{j+1}\cdots c_{k-1}}{\lambda_a^{k-j+1}}\right).
\]
Dividing both sides by $\lambda-\lambda_a$ and letting $\lambda\to\lambda_a$, we obtain
\[
u_0^{(\lambda_a)}:=\lim_{\lambda\to\lambda_a}u_0^{(\lambda_a)}
=-\frac{1}{\lambda_aR'(\lambda_a)}p\sum_{k=1}^\infty\delta_k
\sum_{j=1}^k w_j(k-j+1)\frac{c_jc_{j+1}\cdots c_{k-1}}{\lambda_a^{k-j+2}}
\]
Using now \eqref{eq:uk} with $\lambda=\lambda_a$ and
$u_0=u_0^{(\lambda_a)}$ defined just above, we find the unique solution $u^{(\lambda_a )} $ to 
$(B_Y-\lambda_a I_Y)u=w$. It follows from Proposition~\ref{prop: p < 1/2} that $u^{(\lambda_a )} \in \ell_1$. Then $u^{(\lambda_a )} \in Y$ since $Y$ is closed in $\ell_1$. This completes the proof in the aperiodic case.

\underline{The orbit of zero is eventually periodic.} By definition, $k_0 \ge 0$ is the minimal integer such that $T^{k_0}(0)=T^{\varkappa'+1}(0)$. The orbit of zero has $\varkappa'+1$ points and the space $(U, \| \cdot \|)$ is isometric to $(\R^{\varkappa'+1}, {\| \cdot \|}_1)$. The isometry $M$ between these spaces can be written as follows: for any function $f \in U$, which is of the form $f=\sum_{k=0}^\infty u_k{\ind}\{x\ge T^k(0)\}$ for some $u \in \ell_1$, 
\begin{equation} \label{eq: M explicit}
(Mf)_k=u_k,\quad 0 \le k < k_0, \quad \text{ and } \quad (Mf)_k= \sum_{m=0}^\infty u_{k+m(\varkappa'-k_0+1)},\quad k_0 \le k \le \varkappa'.
\end{equation}

Every $f \in U$ admits a unique representation $f(x)=\sum_{k=0}^{\varkappa'} s_k {\ind}\{x\ge T^k(0)\}$ with $s_k = (Mf)_k$. Hence by \eqref{eq: Pf =}, 
\[
Pf(x)
= \left(p\sum_{k=0}^{\varkappa'} \delta_k s_k\right){\ind}\{x\ge T^{0}(0)\}
+\sum_{k=1}^{\varkappa' +1} c_{k-1 }s_{k-1}{\ind}\{x\ge T^{k}(0)\}.
\]
Since in the second sum the term with $k=\varkappa'+1$ is $c_{\varkappa'}s_{\varkappa'} {\ind}\{x\ge T^{k_0}(0)\}$, we see that $P$ is equivalent to the finite dimensional linear operator $\widehat{A}$ on $\R^{\varkappa'+1}$ given by
\begin{equation}
\label{eq:2.3}
\begin{array}{cl}
&(\widehat{A}s)_0=p\sum_{k=0}^{\varkappa'}\delta_k s_k + c_{\varkappa'}s_{\varkappa'} \ind \{ k_0=0\},\\
&(\widehat{A}s)_k=c_{k-1}s_{k-1} + c_{\varkappa'}s_{\varkappa'} \ind \{ k_0=k\},\quad 1 \le k \le \varkappa'.\\
\end{array}
\end{equation}

We first note that the operator $\widehat A$ is non-negative and irreducible. Moreover, the matrix that represents $\widehat A$ in the standard basis is primitive, i.e.\ it does not have a cyclic structure, see~\cite[p.~680]{Meyer00}. Therefore, by the Perron--Frobenius theorem, $\widehat{A}$ has a positive simple eigenvalue that strictly exceeds the maximum $\varrho'$ of absolute values of all other eigenvalues.

Let us solve the equation $\widehat{A}s=\lambda s$ to determine this leading eigenvalue.  It is immediate from the second line in \eqref{eq:2.3} that 
\begin{equation}
\label{eq:2.4}
s_k=\frac{c_0c_1\cdots c_{k-1}}{\lambda^k}s_0,\quad k=1, \ldots,k_0-1,
\end{equation}
and
\begin{equation}
\label{eq:2.5}
s_k=\frac{c_{k_0}c_{k_0+1}\cdots c_{k-1}}{\lambda^{k-k_0}}s_{k_0},
\quad k=k_0+1,\ldots,\varkappa'.
\end{equation}

Assume first that $k_0 \neq 0$.  From the second line in \eqref{eq:2.3} for $k=k_0$ we get 
\begin{align*}
\lambda s_{k_0}&=c_{k_0-1}s_{k_0-1}+c_{\varkappa'}s_{\varkappa'}=\frac{c_0c_1\cdots c_{k_0-1}}{\lambda^{k_0-1}}s_0
+\frac{c_{k_0}c_{k_0+1}\cdots c_{\varkappa'}}{\lambda^{\varkappa' - k_0}}s_{k_0}.
\end{align*}
The positive solution to $\lambda^{\varkappa'-k_0+1}=c_{k_0}c_{k_0+1}\cdots c_{\varkappa'}$ is not an eigenvalue of $\widehat A$. Indeed, in this case $s_0 = 0$, and plugging \eqref{eq:2.4} and \eqref{eq:2.5} into the first line of \eqref{eq:2.3} gives an impossible identity
\[
\sum_{k=k_0}^{\varkappa'}\delta_k\frac{c_{k_0}c_1\cdots c_{k-1}}{\lambda^{k-k_0}} =0.
\]
Consequently, 
\begin{equation}
\label{eq:2.6}
s_{k_0}=\frac{c_0c_1\cdots c_{k_0-1}}{\lambda^{k_0}}
\left(1-\frac{c_{k_0}c_{k_0+1}\cdots c_{\varkappa'}}{\lambda^{\varkappa'-k_0+1}}\right)^{-1} s_0.
\end{equation}
Plugging \eqref{eq:2.4}, \eqref{eq:2.5} and \eqref{eq:2.6} into the first line of \eqref{eq:2.3}, we conclude that the leading eigenvalue satisfies the equation 
\begin{equation}
\label{eq:2.7}
\lambda=p\sum_{k=0}^{k_0-1}\delta_k\frac{c_0c_1\cdots c_{k-1}}{\lambda^k}
+p\left(1-\frac{c_{k_0}c_{k_0+1}\cdots c_{\varkappa'}}{\lambda^{\varkappa'-k_0+1}}\right)^{-1}
\sum_{k=k_0}^{\varkappa'}\delta_k\frac{c_0c_1\cdots c_{k-1}}{\lambda^k}.
\end{equation}

Let us show that this equation is equivalent to \eqref{eq: lambda} for all possible $k_0$, including $k_0=0$. To this end we notice that periodicity of the orbit implies that $\delta_k=\delta_{(\varkappa'-k_0+1)m+k}$ and 
$c_k=c_{(\varkappa'-k_0+1)m+k}$ for all $k=k_0,k_0+1,\ldots,\varkappa'$ and all $m\ge 0$. Then
\begin{align}
\label{eq:lambda_per}
\nonumber
p\sum_{k=k_0}^\infty\delta_k\left(\frac{p}{\lambda}\right)^k 
\Big(\frac{q}{p}\Big)^{L_k}
&=p\sum_{k=k_0}^\infty\delta_k\frac{c_0c_1\ldots c_{k-1}}{\lambda^k}\\
\nonumber
&=p\sum_{m=0}^\infty\sum_{k=k_0}^{\varkappa'}
\delta_{(\varkappa'-k_0+1)m+k}\frac{c_0c_1\ldots c_{(\varkappa'-k_0+1)m+k-1}}{\lambda^{(\varkappa'-k_0+1)m+k}}\\
\nonumber
&=p\sum_{k=k_0}^{\varkappa'}\delta_k\frac{c_0c_1\ldots c_{k-1}}{\lambda^k}\sum_{m=0}^\infty\left(\frac{c_{k_0}c_{k_0+1}\ldots c_{\varkappa'}}{\lambda^{\varkappa'-k_0+1}}\right)^m\\
&=\left(1-\frac{c_{k_0}c_{k_0+1}\ldots c_{\varkappa'}}{\lambda^{\varkappa'-k_0+1}}\right)^{-1}
p\sum_{k=k_0}^{\varkappa'}\delta_k\frac{c_0c_1\ldots c_{k-1}}{\lambda^k},
\end{align}
and thus equations~\eqref{eq: lambda} and \eqref{eq:2.7} are indeed equivalent. 

We now consider the purely periodic case $k_0 = 0$. Plugging \eqref{eq:2.5} into the first line of~\eqref{eq:2.3}, we get 
\[
\lambda s_0=p\sum_{k=0}^{\varkappa'}\delta_k
\frac{c_0c_1\ldots c_{k-1}}{\lambda^k}s_0+
\frac{c_0c_1\ldots c_{\varkappa'}}{\lambda^{\varkappa'}}s_0.
\]
Consequently, a non-trivial solution exists if and only if $\lambda$ satisfies the equation
\begin{equation}
\label{eq:lead.eig}
\lambda=p\sum_{k=0}^{\varkappa'}\delta_k
\frac{c_0c_1\ldots c_{k-1}}{\lambda^k}+
\frac{c_0c_1\ldots c_{\varkappa'}}{\lambda^{\varkappa'}}.
\end{equation}
On the other hand, by \eqref{eq:lambda_per}, equation~\eqref{eq: lambda} is equivalent to
\[
\lambda \left(1-\frac{c_0c_1\ldots c_{\varkappa'}}{\lambda^{\varkappa'+1}}\right) = p\sum_{k=0}^{\varkappa'}\delta_k\frac{c_0c_1\ldots c_{k-1}}{\lambda^k}.
\]
We now easily see that \eqref{eq:lead.eig} is equivalent to \eqref{eq: lambda} in the purely periodic case. Therefore, $\lambda_a$ is the leading eigenvalue of $\widehat A$ for all possible $k_0$.

The eigenspace corresponding  to $\lambda_a$ has dimension one by the Perron--Frobenius theorem. It is spanned by the vector $MV$. In the case $k_0\neq 0$, this can be seen from \eqref{eq:2.4}, \eqref{eq:2.5}, \eqref{eq:2.6}, and \eqref{eq: M explicit}. 
In the purely periodic case $k_0=0$, this is readily seen from  \eqref{eq:2.5} and~\eqref{eq: M explicit}.


Lastly, we find the eigenvectors $ {\widehat A}^* u^* = \lambda_a u^*$. Using the definition of $\widehat A$, we get the  coordinate-wise equations
\[
\begin{array}{rlc}
&\lambda_a  u_k^*=p\delta_k u_0^*+c_k u^*_{k+1}, &  0 \le k < \varkappa', \\
&\lambda_a  u^*_{\varkappa'} = p \delta_{\varkappa'} u_0^* + c_{\varkappa'}  u^*_{k_0}.\\
\end{array}
\]
We solve them exactly as we did in the aperiodic case for \eqref{eq: left ev}. This gives us the same answer $u_k^* = u_0^* v_k^* $ for $1 \le k \le \varkappa'$ because the first $\varkappa'$ equations are the same and we did not use the equation in the second line.  It is not hard  to check that this last equation  is  indeed satisfied for such $u^*$. Then $v^*$, which is proportional to $u^*$, is the eigenvector of ${\widehat A}^*$ corresponding to $\lambda_a$. Therefore,~\eqref{eq: A^n <} follows for any $\gamma \in (\varrho'/\lambda_a, 1)$ as in the aperiodic case with finite~$\varkappa$.

\end{proof}

\section{Properties of the mapping $a \mapsto \lambda_a$} \label{sec: properties}
In this section we prove the properties of $\lambda_a$ stated in Theorem~\ref{thm: main}. We first study the properties of the trajectory of $0$ under $T_a$ as a function of parameter~$a$. 

\begin{lemma} \label{lem: maximal}
For any integer $k \ge 2$, the following is true.

a) The set $\{a \in [1/2,2/3]: \varkappa_a \ge k\}$ is a finite union of closed non-degenerate intervals. On each of these intervals, the functions $T_a^1(0), \ldots, T_a^k(0)$ are strictly decreasing and continuous, and $\delta_1, \ldots, \delta_{k-1}$ are constant. 

b) The set $\{a \in [1/2,2/3]: \varkappa_a = k\}$ is a union of all disjoint open intervals $(a'',a')$ such that $T_{a''}^k(0)=1$ and $T_{a'}^k(0)=\frac{2a'-1}{1-a'}$.
\end{lemma}

\begin{proof}
We argue by induction. In the basis case $k=2$, we have $\{a \in [\frac12, \frac23]: \varkappa_a \ge k\}= [\frac12, \frac23]$, where both functions $T_a^1(0)=\frac1a$ and $T_a^2(0)=\frac{1-a}{a^2}$ are strictly decreasing and continuous. Therefore, since the range of the second one is $[\frac34,2]$ and the function $\frac{2a-1}{1-a}$ increases on $[\frac12, \frac23]$ from $0$ to $1$, we have
\[
\{a \in [1/2,2/3]: \varkappa_a = 2\} = \{a\in [1/2,2/3]: T_a^2(0) \in I_a\}= (a'',a')
\]
for the unique $a'$ and $a''$ such that $T_{a''}^2(0)=1$ and $T_{a'}^2(0)=\frac{2a'-1}{1-a'}$.

Assume now that the claims are proved for all $2\le k\le n$. Let $J$ be a maximal closed interval contained in $\{a \in [1/2,2/3]: \varkappa_a \ge n\}$. It is non-degenerate (not a point) by the assumption of induction. Since 
\[
\{a : \varkappa_a \ge n\} =    \Big[\frac12, \frac23 \Big]\setminus \bigcup_{i=2}^{n-1} \{a : \varkappa_a = i\}
\]
and the sets under the union are disjoint, it follows from the assumption of induction for Part b) that $J = [a',a'']$ for some distinct $a' < a''$ that satisfy the following restrictions. If $\frac12 \in J$, then $a'=\frac12$ and  $T_{a''}^i(0)=1$ for some integer $2 \le i \le n-1$. If $\frac23 \in J$, then $a''=\frac23$ and $T_{a'}^i(0)=\frac{2a'-1}{1-a'}$ for some integer $2 \le i \le n-1$.  Otherwise,  $T_{a''}^i(0)=1$ and $T_{a'}^j(0)=\frac{2a'-1}{1-a'}$ for some integers $2 \le i, j \le n-1$. 
Notice that in all these cases, we have 
\begin{equation} \label{eq: dicho}
T^n_{a'}(0)= \frac{1}{1-a'} \quad \text{ and } \quad \text{either } T^n_{a''} (0)\ge 1 \text{ or }T^n_{a''}(0) < \frac{2a''-1}{1-a''}.
\end{equation}
Indeed, if $a'' =\frac23$, then  the dichotomy above is trivial because in this case  $\frac{2a''-1}{1-a''}=1$, otherwise the orbit of zero under $T_{a''} $ is purely periodic and  thus $0$ is its only point in $[0,1)$. 

If $T^n_{a''} (0)\ge 1$, then $\{a \in J: \varkappa_a \ge n+1\} = J$. Otherwise, since the mapping $a \mapsto T_a^n(0)$ is continuous and strictly decreasing on $J$, it follows from~\eqref{eq: dicho} that there exist unique  $b', b'' \in (a',a'')$ such that $T_{b''}^n(0)=1$ and $T_{b'}^n(0)=\frac{2b'-1}{1-b'}$; consequently,
\[
\{a \in J: \varkappa_a \ge n+1\} = [a', b''] \cup [b',a''] \quad \text{ and } \quad
\{a \in J: \varkappa_a = n+1\} =  (b'',b').
\]
Thus, since $\{a:\varkappa_a \ge n\} $ is a finite union of closed disjoint non-degenerate intervals, the set $\{a:\varkappa_a \ge n+1\} $ is so.  The assertion of Part b) follows for $k=n+1$ because every interval contained in $\{a: \varkappa_a = n+1\} $ is a subinterval of a maximal interval in $\{a:\varkappa_a \ge n\} $. 

Furthermore, let $J' = J$  if $T^n_{a''} (0)\ge 1$, otherwise let $J'$ be either of the intervals $[a', b'']$ and $[b',a'']$. Consider the mapping $(a, x) \mapsto T_a(x)$ defined on its domain $D$ given by
\[
D= \bigg \{(a,x): a \in \Big[\frac12, \frac23 \Big], \, x \in \Big [0, \frac{2a-1}{1-a} \Big] \cup \Big [1, \frac{1}{1-a}\Big ] \bigg \}.
\]
This mapping is continuous on $D \setminus \{(\frac23,1)\}$. We have $(a, T_a^n(0)) \in D \setminus \{(\frac23,1)\}$ for every $a \in J'$ by the choice of $J'$ and the fact that the orbit of $0$ under $T_{2/3}$ does not hit $1$, as shown in Section~\ref{sec:Operator}. Therefore, $T_a^{n+1}(0)$ is continuous on $J'$, as claimed, by $T_a^{n+1}(0) = T_a(T_a^{n}(0))$.  Clearly, $\delta_n$ is constant on $J'$, as claimed.

Lastly, pick any $a,b \in J'$ that satisfy $a > b$. To finish the proof, we claim that
\[
T^{n+1}_{a}(0)=T_{a}(T^n_{a}(0)) < T_{a}(T^n_b(0)) <  T_b(T^n_b(0)) = T^{n+1}_b(0).
\]
Indeed, in the first inequality we used that $T^n_{a}(0) < T^n_b(0)$ by the assumption of induction and the facts that (i) $T_a(x)$ increases in $x$ on each of the intervals 
$[0, \frac{2a-1}{1-a})$ and $[1, \frac{1}{1-a}]$, and (ii) one of these intervals contains both quantities $T^n_{a}(0)$ and $T^n_{b}(0)$ by the choice of $J'$. In the second inequality we used that $T_a(x) < T_b(x)$ for every fixed $x$ in the domain of $T_b$, which contains $T^n_b(0)$ (by the choice of $b$), and is included in the domain of $T_a$. 
\end{proof}

We now prove the properties of $\lambda_a$ stated 
in Theorem~\ref{thm: main}.

1. It is evident from \eqref{AR-def} and a simple coupling argument that the persistence probabilities $\pr_x(\tau>n)$ are monotone in $a$ for every fixed $x$ and $n$. Together with~\eqref{lambda-def} this yields monotonicity of $\lambda_a$ on $(0,1)$. 

Note in passing that it is easy to give a self-contained proof of the monotonicity using the facts that $\lambda_a$ is constant on every interval in $[\frac12, \frac23] \setminus S$ and is constant on no open interval that meets $S$. We will prove these facts later on using equation~\eqref{eq: lambda}.

2. Let us show that the function $\lambda$ is continuous at every $a' \in [\frac12, \frac23]$. 

Assume that $\varkappa_{a'}=k_0$ is finite. It follows from Lemma~\ref{lem: maximal} that the  functions $a \mapsto \delta_1, \ldots, \delta_{k_0}, \varkappa_a$ are  constant on the maximal open interval that includes $a'$ and is contained in the set $\{a: \varkappa_a = k_0\}$. Then $\lambda_a$ is constant on this interval, and therefore continuous at $a'$.

We now assume that $\varkappa_{a'}= \infty$. Denote by $ \delta_k', L_k'$ the respective values of $\delta_k, L_k$ at $a=a'$.

a) Assume first that $\{ T_{a'}^k(0): k \ge 0 \}$ hits neither of the points $1$ and $\frac{2a'-1}{1-a'}$; then $a' \neq \frac12$. Consider the case where  $a' \neq \frac23$. Then for every $N \ge 1$, the point $a'$ is in the interior of some maximal interval contained in the set $\{a: \varkappa_a \ge N\}$, because otherwise $a'$ is an endpoint of some maximal interval in one of these sets, which is a contradiction by Lemma~\ref{lem: maximal}.b. Therefore, by Lemma~\ref{lem: maximal}.a,  
for every $N \ge 1$ there exists an $\varepsilon>0$ such that $(a'-\varepsilon, a'+ \varepsilon) \subset (1/2, 2/3)$ and $|a-a'|< \varepsilon$ implies that $\varkappa_a \ge N$ and $\delta_0= \delta_0', \ldots, \delta_N=\delta_N'$. For any $a \in (a'-\varepsilon, a')$, we have $\lambda_a \le \lambda_{a'}$ and by \eqref{eq: lambda},
\begin{align*}
0 \le \lambda_{a'} - \lambda_a &= p\sum_{k=0}^{\varkappa_{a'}} \Big( \frac{p}{\lambda_{a'}} \Big)^k \Big( \frac{q}{p} \Big)^{L_k'} \delta_k' - p\sum_{k=0}^{\varkappa_a} \Big( \frac{p}{\lambda_a} \Big)^k \Big( \frac{q}{p} \Big)^{L_k} \delta_k \\
&\le p\sum_{k=0}^\infty \Big( \frac{p}{\lambda_{a'}} \Big)^k \Big( \frac{q}{p} \Big)^{L_k'} \delta_k' - p\sum_{k=0}^N \Big( \frac{p}{\lambda_{a'}} \Big)^k \Big( \frac{q}{p} \Big)^{L_k} \delta_k \\
&= p\sum_{k=N+1}^\infty \Big( \frac{p}{\lambda_{a'}} \Big)^k \Big( \frac{q}{p} \Big)^{L_k'} \delta_k'.
\end{align*}
The last expression tends to zero as $N \to \infty$ as a remainder of a converging series. This proves left continuity of $\lambda_a$ at $a'$. Similarly, for any $a \in (a', a'+\varepsilon)$, we get 
\[
0 \le \lambda_a - \lambda_{a'} \le p\sum_{k=N+1}^{\varkappa_a} \Big( \frac{p}{\lambda_a} \Big)^k \Big( \frac{q}{p} \Big)^{L_k} \delta_k \le \sum_{k=N+1}^\infty \Big( \frac{p}{\lambda_{a'}} \Big)^k \Big( \frac{q}{p} \Big)^{C(k+1)}
\]
with $C=C_a>0$ given in Remark~\ref{rem: C} if $p < 1/2$ and $C= 0$ if $p \ge 1/2$. The last expression above tends to zero as $N \to \infty$ uniformly in $a \in (a', a'+\varepsilon)$. This is obvious if $p \ge 1/2$, otherwise this follows from the inequality $\lambda_{a'} > p (q/p)^{C_{a'}}$ and the fact that $C_a=C_{a'}$ for all $a \in (a', a'+\varepsilon)$ when $N \ge t_{[(q/p)^{1/t_1}]}$. This proves right continuity of $\lambda_a$ at $a'$. 

For $a'=\frac23$, we need to prove only the left continuity of $\lambda_a$, which follows as above.

b)  Assume now that $T_{a'}^{k_0}(0)=\frac{2a'-1}{1-a'}$ for some $k_0 \ge 0$. If $a' \neq 1/2$, it follows from Lemma~\ref{lem: maximal}.b that the  functions $a \mapsto \delta_1, \ldots, \delta_{k_0}, \varkappa_a$ are  constant on the maximal open interval that has the right endpoint $a'$ and is contained in the set $\{a: \varkappa_a = k_0\}$. Moreover, $\delta_k=\delta'_k$ for all $k\le k_0$. Then
$a \mapsto \lambda_a$ is constant on this interval, and therefore 
we will prove left-continuity of this mapping at $a'$ if we show that
$\lambda_a=\lambda_{a'}$ for all $a$ in this interval. To this end we notice that $T_{a'}^k(0)=\frac{1}{1-a'}$ for all $k>k_0$. Consequently, $\varkappa_{a'}=\infty$ and $\delta_k=0$ for all $k>k_0$. Using these properties we can represent equation \eqref{eq: lambda} in the following way:
\begin{align*}
1=\sum_{k=0}^{\varkappa_{a'}}\delta_k'\left(\frac{p}{\lambda}\right)^{k+1}\Big(\frac{q}{p}\Big)^{L_k'}
=\sum_{k=0}^{k_0}\delta_k'\left(\frac{p}{\lambda}\right)^{k+1}\Big(\frac{q}{p}\Big)^{L_k'}
=\sum_{k=0}^{\varkappa_a}\delta_k\left(\frac{p}{\lambda}\right)^{k+1}\Big(\frac{q}{p}\Big)^{L_k}.
\end{align*}
Therefore, $\lambda_a=\lambda_{a'}$ as required, and the proof of left-continuity of $\lambda_a$ at $a'$ is completed. 

Furthermore, for every $N \ge k_0$, $a'$ is the left endpoint of some maximal interval contained in the set $\{a: \varkappa_a \ge N\}$. By Lemma~\ref{lem: maximal}.a, for every $N \ge k_0$ there exists an $\varepsilon \in (0, 2/3-a')$ such that $a\in (a', a'+ \varepsilon)$ implies that $\varkappa_a \ge N$ and $\delta_0= \delta_0', \ldots, \delta_N=\delta_N'$. Hence for the constant $C_a$ given in Remark~\ref{rem: C} it is true that $C_{a'} \ge C_a$ when $a\in (a', a'+ \varepsilon)$. Then right continuity of $\lambda_a$  at $a'$ follows as above in Case a).

c) Assume that $ T_{a'}^k(0) = 1$ for some $k \ge 0$. The left-continuity of $\lambda_a$ at $a'$ follows as above in Case a). The right-continuity of $\lambda_a$ at $a'$ follows as above in case of finite $\varkappa_{a'}$ since the orbit of zero is purely periodic and therefore $\lambda_{a'}$ satisfies equation~\eqref{eq:lead.eig}.

3. Let us prove that the (topological) support of the Lebesgue--Stieltjes measure $d \lambda_a$ on $[1/2,2/3]$ is the set $S$, which is closed and  has measure zero. 

Recall that $S= \{ a \in [1/2,2/3]: \varkappa_a =\infty \}$ and $\widehat T_a(x)=\frac1a x +\frac12 \pmod 1$ for $ 0 \le x \le 1$. By Parry~\cite[Theorem~6]{Parry64} and Halfin~\cite[Theorem~4.4]{Halfin75}, the mapping $\widehat T_a$ has an absolutely continuous invariant probability measure $\widehat  \mu_a$ on $[0,1]$ with the density proportional to 
\begin{equation} \label{eq: Parry}
\widehat  h_a(x)=\sum_{k=0}^\infty a^{k+1} \ind \{\widehat T_a^k(0) \le x \}-
\sum_{k=0}^\infty a^{k+1} \ind \{\widehat T_a^k(1) \le x\}.
\end{equation}
Note in passing that there are no other invariant probability densities for $\widehat T_a$.

It follows from \eqref{eq: T mod 1} that $\varkappa_a=\inf \{k\ge 0:\widehat T_a^k(0)\in \widehat I_a \}$, where $\widehat I_a = (\frac{a(2a-1)}{2(1-a)},\frac{a}{2})$.
According to Corollary 1 in Faller and Pfister \cite{FP09}, the orbit of zero $\{\widehat T_a^k(0): k \ge 0\}$ is $\widehat  \mu_a$-{\it normal} for almost all values of $a$. This means that for every continuous function $f$ on $[0,1]$, we have
\[
\lim_{n\to\infty}\frac{1}{n}\sum_{i=0}^{n-1}f(\widehat T_a^k(0)) =\int_0^1 f(x) \widehat \mu_a(dx).
\]
On the other hand, for $a\in (\frac{1}{2},\frac{2}{3})$ we have $\widehat \mu_a(\widehat I_a)>0$ by Theorem 3 in Hofbauer~\cite{Hofbauer81}, which asserts that the support of $\widehat  \mu_a$ is the whole of the interval $[0,1]$. Hence $\varkappa_a$ is finite for almost all  $a\in (\frac{1}{2},\frac{2}{3})$, and thus $S$ has Lebesgue measure zero.

The set $S$  is closed because its complement in $[\frac{1}{2},\frac{2}{3}]$ is open by  Lemma~\ref{lem: maximal}.b.
The topological support of $d \lambda_a$  is a subset of $S$ since $\lambda_a$ is constant on every interval contained in $[\frac{1}{2},\frac{2}{3}] \setminus S$, as we proved above in Item 2. Then $d \lambda_a$ is singular with respect to the Lebesgue measure. It remains to show that $\lambda_a$ is constant on no open interval that meets $S$.

To this end, we first observe that if $T^{k_0}_{a'}(0)=\frac{2a'-1}{1-a'}$ for some $k_0 \ge 1$, then $\lambda_{a'} < \lambda_a$ whenever $a'< a \le 2/3$. Indeed, pick the $\varepsilon>0$ defined in Item 2.b above for $N=k_0$. If there is an $\varepsilon_1 \in (0, \varepsilon)$ such that $\lambda_a = \lambda_{a'}$ for any $a \in (a', a'+\varepsilon_1)$, then for such $a$ it follows from~\eqref{eq: lambda} that
\[
\lambda_a - \lambda_{a'} 
=  p\sum_{k=k_0+1}^{\varkappa_a} \Big( \frac{p}{\lambda_a} \Big)^k \Big( \frac{q}{p} \Big)^{L_k} \delta_k
\]
because $\delta_k'=0$ for all $k > k_0$. At least one term in the sum is strictly positive because $1 \le T_a^{k_0+1}(0) < \frac{1}{1-a}$, hence the trajectory of $0$ eventually returns to $[0,1)$. This is a contradiction.

Now assume that $a' \in S$ and  $\{ T_{a'}^k(0): k \ge 1 \}$ does not hit the point $\frac{2a'-1}{1-a'}$. Then $a'>\frac12$, and since the set $S$ has measure $0$, there exists an increasing sequence $\{a_n\}_n \subset (1/2, a') \setminus S$ that converges to $a'$ as $n \to \infty$. Denote by $a_n'$ the right endpoint of the maximal open interval in $[\frac12, \frac23]\setminus S$ that contains $a_n$. The sequence $a_n'$ increases and converges to $a'$. We have $a_n'< a'$ for every $n $ because otherwise $ T_{a'}^k(0)=\frac{2a'-1}{1-a'}$ for some $k \ge 1$ by Lemma~\ref{lem: maximal}.b. Then $\lambda_{a_n'} < \lambda_{a'}$, as shown above, hence $\lambda_a < \lambda_{a'}$ for all $a<a'$.  Thus, we showed that  $\lambda_a$ is constant on no open neighbourhood of $a'$. 

4. The equality $\lambda_{1/2}(p)=p$ for all $p \in (0,1)$ is already established in \eqref{eq: lambda <1/2}.

\section{Convergence to the quasi-stationary distribution $\nu_a$ and its properties} \label{sec: condit}
In this section we prove convergence of the conditional distributions stated in \eqref{eq: weak lim}. Then we prove the properties of the limiting quasi-stationary distribution $\nu_a$ stated in Theorem~\ref{thm: main}.

\subsection {Convergence of the conditional distributions}
\label{ssec:cond}

The functional space $U$, which we used in our analysis of the probabilities $\pr_x(\tau>n)$, is quite narrow. This space does not contain indicator functions of all subintervals of $[0,1/(1-a)]$ and therefore, we cannot use $U$ to study the distribution of $X_n$ conditioned on $\{\tau>n\}$. For this reason, we shall now regard $P$ as an operator acting on the larger space $BV$ of functions of bounded variation on $[0, 1/(1-a)]$ equipped with the standard norm~\eqref{eq: BV norm}. 

We will use the following decomposition of the operator $P$. Define
\[
P_1f(x):=pf(1){\ind}\{x\ge T^0(0)\}+qf(0){\ind}\{x\ge T^1(0)\}
\]
and 
\[
P_2f(x)=p(f(ax+1)-f(1))+q{\ind}\{x\ge 1/a\}(f(ax-1)-f(0)).
\]
Then $P=P_1+P_2$ and $P_1f\in U$ for every function $f$ of bounded variation on $[0, 1/(1-a)]$. 

Using induction one can easily show that
\begin{equation}
\label{eq:Pnf}
P^nf=P_2^nf+\sum_{j=0}^{n-1}P^{n-j-1}(P_1P_2^jf), \quad n\ge1.
\end{equation}
Since each function $P_1P_2^jf$ belongs to $U$, it follows from \eqref{eq:PF} that
\[
P^{n-j-1}(P_1P_2^jf)(x)\sim c (M(P_1P_2^jf),v^*)\lambda_a^{n-j-1}V(x)
\]
as $n \to \infty$ for every fixed $j \ge 0$ and $x \in  [0, 1/(1-a)]$. Moreover, by \eqref{eq: norms}, for every $n \ge 1$, 
\begin{equation}
\label{eq:P^n f diff}
\bigg  |\lambda_a^{-n}P^n f(x)- \sum_{j=0}^{n-1} \frac{c V(x)}{\lambda_a^{j+1} }(M(P_1P_2^jf),v^*) \bigg | \le \lambda_a^{-n} \| P_2^n f \| + \sum_{j=0}^{n-1} \frac{C_1\gamma^n}{(\gamma \lambda_a)^{j+1}}\|P_1P_2^jf\|.
\end{equation}

For now  it suffices to consider the step functions $f_z(x)= \ind\{x > z\}$ for $z \in  [0, 1/(1-a)]$. Let us compute $P_1 P_2^j f_z$ for a fixed  $z$. Denote
\[
c_k(z)=p(1-\delta_k(z))+q\delta_k(z) 
\quad\text{for}\quad
0 \le k < \varkappa_a(z)+1,
\]
where, recall, $\delta_k(z) = \ind\{T_a^k(z)<1\}$. By the definition of $P_2$, we have
\[
P_2 f_z(x) =p \big (\ind\{ax+1 > z \}-\ind\{1 > z \} \big)+q{\ind}\{x \ge 1/a\} \big(\ind\{ax-1 > z \}-\ind\{0 > z \}\big).
\]
Considering the three possible positions of $z$ relative to the set $I_a$, it is easy to check that 
$
P_2f_z(x)
=c_0(z)\ind\{x>T_a(z)\}
$
for $z \not \in  I_a$ and $P_2f_z\equiv 0$ for $z\in I_a$. Iterating this, we get
\begin{equation} \label{eq: P_2^j step}
P_2^jf_z(x)=c_0(z)c_1(z)\ldots c_{j-1}(z)\ind\{x>T_a^j(z)\}
\end{equation} 
for all integer $j \le \varkappa_a(z)$
and $P_2^jf_z\equiv 0$ for $j>\varkappa_a(z)$.
Therefore, since 
\begin{equation} \label{eq: P_1 P_2^j step}
P_1 P_2^j f_z(x) = p P_2^j f_z(1)
\quad\text{for every }x\in [0,1/(1-a)],
\end{equation} 
we obtain 
\begin{equation} \label{eq: bar F}
\sum_{j=0}^\infty \lambda_a^{-j-1}P_1P_2^jf_z(x)
=\sum_{j=0}^{\varkappa_a(z)} \delta_j(z)  (p/\lambda_a)^{j+1} (q/p)^{L_j(z)}=:\overline{F}_a(z).
\end{equation}


Combining equalities \eqref{eq: P_2^j step} and  \eqref{eq: P_1 P_2^j step} with estimate \eqref{eq:P^n f diff}, where $(M \ind, v^*)=1 $ for the constant function $\ind (x) \equiv 1$ by~\eqref{eq: v^*}, and estimating the reminder of the sum in \eqref{eq: bar F}, we obtain
\begin{align*}
\big |\lambda_a^{-n}  P^nf_z(x) - c \overline{F}_a(z) V(x) \big|  & \le  (C_1+ \gamma \lambda_a/p) \gamma^n \sum_{j=0}^{\varkappa_a(z) \wedge n} (p/(\gamma \lambda_a))^{j+1} (q/p)^{L_j(z)} \\
&\qquad  +  cV(x)  \sum_{j=n}^{\varkappa_a(z)} (p/\lambda_a)^{j+1} (q/p)^{L_j(z)}.
\end{align*}
Denote $r=\max(1, q/p)$ and $C_2'=C_1+ \gamma \lambda_a/p+ cV(\tfrac{1}{1-a})$. Since $V(x) \le V(\tfrac{1}{1-a})$, by Proposition~\ref{prop: p < 1/2} we get
\begin{align*}
\big |\lambda_a^{-n}  P^nf_z(x) - c \overline{F}_a(z) V(x) \big| 
& \le C_2' \gamma^n \sum_{j=0}^n (p/(\gamma \lambda_a))^{j+1} r^{C(j+1)+1} +  C_2' \sum_{j=n}^\infty (p/\lambda_a)^{j+1} r^{C(j+1)+1}.
\end{align*}
This implies that for any $\gamma_1 \in \big (\max(\gamma, p r^C / \lambda_a),1 \big)$, there exists a constant $C_2>0$ such that
\begin{equation}
\label{eq:Pnf conv rate}
|\lambda_a^{-n}  P^nf_z(x) - c \overline{F}_a(z) V(x)|  \le C_2  \gamma_1^n
\end{equation}
for every $x, z \in [0,1/(1-a)]$ and $n \ge 1$. In particular, \eqref{eq:Pnf conv rate} implies that $\overline{F}_a(z)$ is finite.

Finally, using \eqref{eq:Pnf conv rate} and taking \eqref{eq:PF} into account, we conclude that  
\begin{equation}
\label{eq:quasi-stat}
\lim_{n \to \infty} \pr_x(X_n>z|\tau>n)
= \lim_{n \to \infty} \frac{P^n f_z(x)}{P^n \ind (x)}
= \overline{F}_a(z)  
\end{equation}
uniformly in  $x,z \in[0,1/(1-a)]$ since $V(x) \ge V(0) \ge 1$. This limit does not depend on the starting point $x$. It also follows that the function $\overline{F}_a$ is non-increasing. Therefore, the conditional distributions converge weakly to the measure $\nu_a$ on $[0, \frac{1}{1-a}]$ such that 
\[
\nu_a\left(\left(z,\frac{1}{1-a}\right]\right)
=\overline{F}_a(z)
\]
whenever $\overline{F}_a$ is continuous at $z$. This measure is a probability because the conditional distributions are tight, since they are supported on $[0,1/(1-a)]$.

\subsection{Support and non-atomicity of $\nu_a$}

Fix an $a \in (\frac12, \frac23]$. It is readily seen from \eqref{eq: lambda} that $\overline{F}_a(0) =1$, therefore $\nu_a(\{0\})=0$ for every $a$. Let us check continuity of the function $\overline{F}_a$ at an arbitrary  point $z$.

We first assume that $\varkappa_a(z)$ is finite. Then the orbit of $z$ does not hit the point $\frac{2a-1}{1-a}$. If the orbit does not hit $1$, then by piecewise continuity of the iterations of $T_a$, there exists an $\varepsilon>0$ such that $|z-z'|<\varepsilon$ implies that 
$\varkappa_a(z')=\varkappa_a(z)$,
$L_j(z')=L_j(z)$, and $\delta_j(z)=\delta_j(z')$ for all $j\le\varkappa_a(z)$; cf.\ Lemma~\ref{lem: dom T^n}. Consequently,
\[
\overline{F}_a(z)=\overline{F}_a(z'),\quad |z-z'|<\varepsilon.
\]
Therefore, $z$ is not in the support of $\nu_a$.


Assume that $T^k_a(z)=1$ for some $k\ge0$. Let $k_0$ be the minimal integer with this property. By a continuity argument as above combined with piecewise monotonicity of $T_a$ (cf.\ Lemma~\ref{lem: dom T^n} and use that $z \in G_{\varkappa_a(z)}$), there exists an $\varepsilon >0$ such that $\varkappa_a(z)=\varkappa_a(z')$ and $\overline{F}_a(z)=\overline{F}_a(z')$ for all $z' \in [z, z+\varepsilon)$. In order to consider the values of $\overline{F}_a(z')$  for $z'<z$ (if $z>0$), we note that $T^{k_0+j+1}_a(z)=T^j_a(0)$ and thus $L_{k_0+j+1}(z) = L_{k_0+1}(z) + L_j$ for every  integer $0 \le j \le \varkappa_a$.
Therefore, 
\begin{align*}
\overline{F}_a(z)
&=\sum_{j=0}^{k_0-1}\delta_j(z)\left(\frac{p}{\lambda_a}\right)^{j+1}
\left(\frac{q}{p}\right)^{L_j(z)}\\
&\hspace{1cm}
+\left(\frac{p}{\lambda_a}\right)^{k_0+1}\left(\frac{q}{p}\right)^{L_{k_0+1}(z)}
\sum_{j=0}^{\varkappa_a(0)}\left(\frac{p}{\lambda_a}\right)^{j+1}
\left(\frac{q}{p}\right)^{L_j} \delta_j(0).
\end{align*}
Taking into account the equalities $L_{k_0+1}(z) -  L_{k_0}(z)=\delta_{k_0}(z)=0$ and  \eqref{eq: lambda}, we arrive at
\begin{equation} \label{eq: F hits 1}
\overline{F}_a(z)
=\sum_{j=0}^{k_0-1}\delta_j(z)\left(\frac{p}{\lambda_a}\right)^{j+1}
\left(\frac{q}{p}\right)^{L_j(z)}
+\left(\frac{p}{\lambda_a}\right)^{k_0+1}\left(\frac{q}{p}\right)^{L_{k_0}(z)},
\end{equation}
which is valid even if $\varkappa_a(z)=\infty$. Using this representation and repeating the argument which we gave above for $z'\ge z$ (cf. Lemma~\ref{lem: dom T^n} and use that $z \in G_{k_0+1}$), we can assume w.l.o.g.\ that $\varkappa_a(z')=k_0$ for $z' \in (z-\varepsilon, z)$ and $\overline{F}_a(z)=\overline{F}_a(z')$ for $z' \in (z-\varepsilon, z)$. Thus, $\overline{F}_a$ is constant on the whole of $(z-\varepsilon,  z+\varepsilon)$.

Thus, recalling that $Q_a=\{z:\varkappa_a(z)=\infty\}$, in either case we showed that if $z \not \in Q_a$, then  $\nu_a((z-\varepsilon,  z+\varepsilon)) = 0$ and the interval $(z-\varepsilon,  z+\varepsilon) $ does not intersect $Q_a$. Hence the set $Q_a$ is closed and $\nu_a$ is supported on $Q_a$. As we have already mentioned in Section~\ref{sec: properties}, for every $a \in (\frac12, \frac23)$, the invariant measure $\widehat \mu_a$ of the transformation $\widehat{T}_a$ is ergodic and is equivalent to the Lebesgue measure on $[0,1]$ by Theorem 3 in~\cite{Hofbauer81}. Hence for every $a \in (\frac12, \frac23)$, almost all orbits $\{\widehat T_a^k(z):k\ge0\}$ are normal. This implies that the set $Q_a$ has Lebesgue measure zero, and thus the measure $\nu_a$ is singular.

We now assume that $\varkappa_a(z)$ is infinite. Denote
\[
\overline{F}_a(z, N) = \sum_{j=0}^{\varkappa_a(z) \wedge N } \delta_j(z) \Big(\frac{p}{\lambda_a}\Big)^{j+1}
\Big(\frac{q}{p}\Big)^{L_j(z)}.
\]
If the orbit of $z$ does not hit the points $\frac{2a-1}{1-a}$ and $1$, then for any $N >0 $ there exists an $\varepsilon(N) >0$ such that 
\[
\overline{F}_a(z, N)=\overline{F}_a(z', N), \qquad |z-z'|<\varepsilon(N),
\]
as in the case of finite $\varkappa_a(z)$. We arrive at the same conclusion if the orbit of $z$ hits $1$ but does not hit $\frac{2a-1}{1-a}$, once we separately consider the points $z' < z $ and $z' \ge z$, as in the finite case; here  $\varkappa_a(z')$ and $\overline{F}_a(z')$ are constant on $(z-\varepsilon, z)$ for some $\varepsilon >0$ and $\varkappa_a(z') \to \infty$ as 
$z' \to z+$. Likewise, we arrive at the same conclusion if the orbit of $z$ hits $\frac{2a-1}{1-a}$ but does not hit $1$; here  $\varkappa_a(z')$ and  $\overline{F}_a(z')$ are constant on $(z, z+\varepsilon)$ for some $\varepsilon >0$  and  $\varkappa_a(z') \to \infty$ as 
$z' \to z-$ if $z>0$ (the case $z=0$ is possible). In each of the three cases, we have
\begin{align}  \label{eq: F tail}
|\overline{F}_a(z')- \overline{F}_a(z)| &= \Big| [\overline{F}_a(z')-\overline{F}_a(z',N)] - [\overline{F}_a(z)-\overline{F}_a(z,N)] \Big| \notag  \\
& \le 2 r^C \sum_{j=N+1}^\infty  \bigg(\frac{p r^C}{\lambda_a}\bigg)^{j+1}
\end{align}
with $r = \max(1, q/p)$ and the constant $C>0$ as in Proposition~\ref{prop: p < 1/2}.
Taking $N \to \infty $ establishes continuity of  $\overline{F}_a$ at point $z$.  

It remains to consider the case where the orbit of $z$ hits both points $\frac{2a-1}{1-a}$ and $1$. This can only happen if $T_a^{k_1}(z)=1$ and $T_a^{k_2}(z) = \frac{2a-1}{1-a}$ for some $0 \le k_1 < k_2$ (and hence $T_a^{k_2-k_1}(0)=\frac{1}{1-a})$. It is easy to see that in this case there is an $\varepsilon >0$ such that $\varkappa_a(z')=k_1$ on $(z-\varepsilon, z)$; $\varkappa_a(z')=k_2$ on $(z,z+\varepsilon)$; and  $\overline{F}_a$ is constant on the whole of $(z-\varepsilon,  z+\varepsilon)$ by \eqref{eq: F hits 1}. In this case $z$, which is an isolated point of $Q_a$, is not in the support of $\nu_a$.  This completes the proof of  continuity of $\overline{F}_a$ on the whole of the interval $[0,\frac{1}{1-a}]$.

It remains to show that the topological support of $\nu_a$ is the set
$Q_a \setminus H_a$, where $H_a = \varnothing$ if $T^k_a(0) \neq \frac{1}{1-a}$ for all integer $1 \le k \le \varkappa_a$, and $H_a = \cup_{k=0}^\infty T_a^{-k}(0)$ otherwise. Our proof above of the continuity of $\overline{F}_a$ at points $z$ with $\varkappa_a(z)=\infty$ actually showed that $\nu_a$ is supported on $Q_a \setminus H_a$, and that each $z \in H_a$ is an isolated point of $Q_a$. 
 

We first assume that $a< \frac23$.  Let $z$ be a point in $Q_a \setminus H_a$. If $T_a^{k_0}(z) =\frac{2a-1}{1-a}$ for some $k_0 \ge 0$, then $z \neq 0$ and  $T_a^k(z) \neq 1$ for all $k < k_0$. Therefore, in this case we can choose an $\varepsilon >0$ such that $z' \in (z-\varepsilon, z)$ implies that $\varkappa_a(z') \ge k_0$ and $\delta_j(z)=\delta_j(z')$ for all $j\le k_0$. Hence $\overline{F}_a(z')> \overline{F}_a(z) $ for such $z'$, because $\delta_j(z)=0$ for all $j >k_0$ and there exists an integer $j_0(z')$ such that $k_0 < j_0(z') \le \varkappa_a(z')$ and $\delta_{j_0(z')}(z')=1$. Then $\overline{F}_a(z')> \overline{F}_a(z) $ for {\it all} $z'<z$.  

If $T_a^k(z) \neq \frac{2a-1}{1-a}$ for all $k \ge 0$ and $z \neq \frac{1}{1-a}$, then since $Q_a$ has measure zero, we can choose a strictly decreasing sequence $\{z_n\} \in [0, \frac{1}{1-a}] \setminus Q_a$ that converges to $z$. Denote $k_n=\varkappa_a(z_n)$ and $z_n' = \max (D_{k_n +1} \cap [0, z_n]) $. Then $z_n' \to z+$ as $n \to \infty$, and $z_n'>z$ for every $n$ because $z_n' \in D_{k_n +1}$ by Lemma~\ref{lem: dom T^n}.a  and $z \not \in D_{k_n +1}$ by the assumption. Therefore, $\overline{F}_a(z)> \overline{F}_a(z_n') $ for every $n$, as shown above. Hence   $\overline{F}_a(z)> \overline{F}_a(z') $ for all $z'>z$.  

Lastly, it is clear that $0=\overline{F}_a(\frac{1}{1-a}) < \overline{F}_a(z')$ for all $z'< \frac{1}{1-a}$. Thus, we showed that if $a \in (\frac12, \frac23)$, then $\overline{F}_a$ is constant on no open neighbourhood of any point in the set $Q_a \setminus H_a$, which therefore is the topological support of $\nu_a$. It has no isolated points since $\overline{F}_a$ is continuous.

It remains to consider the case $a=2/3$, where we shall prove that the support of $\nu_{2/3}$ is $[0,3]$. Let $J \subset[0,3]$ be an open interval. We need to show that $\nu_{2/3}(J)>0$. 

The invariant measure $\widehat \mu_{2/3}$ of the transformation $\widehat{T}_{2/3}$ is ergodic and is equivalent to the Lebesgue measure on $[0,1]$ by Corollary to Theorem 2 in~\cite{Hofbauer81}. Then the measure $\mu_{2/3}$, defined by $\mu_{2/3}(A)=\widehat \mu_{2/3}(A/3)$ for every Lebesgue measurable set $A \subset [0,3]$,  is invariant and ergodic for $T_{2/3}$. Therefore, almost all orbits $\{T_{2/3}^k(z):k\ge0\}$ are normal, and there exist a $z \in J$ and an $\varepsilon \in (0,1)$ such that the orbit of $z$ hits the interval $(1-\varepsilon, 1)$ and $(z, z+\varepsilon) \subset J$. Denote by $k_0 \ge 0$ the first hitting time. Since the mapping $T_a$ is piecewise continuous and satisfies $T_a'=1/a>1$ on the interior of its domain, it follows that there exists a $z' \in (z, z+\varepsilon)$ such that $T_{2/3}^{k_0}(z') =1$. Then $\delta_j(z)=\delta_j(z')$ for all $0 \le j \le k_0 -1$. Since $T_{2/3}(1-)=3$, it is easy to show, using representation \eqref{eq: F hits 1} and arguing as above, that  $\overline{F}_{2/3}(z)> \overline{F}_{2/3}(z') $. Therefore, $\overline{F}_{2/3}$ is not constant on $J$, and thus $\nu_{2/3}(J)>0$, as needed.

\subsection{Singularity properties of $\nu_{2/3}$}

In this part of the proof we assume throughout that $a=2/3$. To start, note that for the $\frac32$-transformation $\overline T(x)=\frac32 x \pmod 1$, we have 
\[
T_{2/3}(x)=3(1-\overline T(1-x/3)), \qquad x \in [0, 1) \cup (1, 3].
\]
In fact, this equality holds true for $x\in \{0,3\}$ and the functions on both sides of the equality are piecewise linear with the only discontinuity at $x=1$, where they have the same one-sided limits. Then it follows by induction that 
\[
T_{2/3}^k(x) = 3 (1 - \overline T^k(1-x/3)), \qquad x \in [0,3] \setminus H, k \ge 1,
\]
where $H=\{z:T_{2/3}^n(z)=1\text{ for some }n \ge 0\}$. Hence, for all $x \in [0,3] \setminus H$ and $k \ge 0$,
\begin{equation} \label{eq: digits 2/3}
\delta_k(x)=\ind\{T_{2/3}^{k}(x)<1\}=\left[\frac{3}{2}\bar{T}^k(1-x/3)\right].
\end{equation}
In particular, $\delta_0, \delta_1, \ldots $ are the digits in the $\frac32$-expansion of $1$, because the orbit of zero under $T_{2/3}$ does not include $1$, as shown in Section~\ref{sec:Operator}.

We know that $\lambda_{2/3}=3/4$ when $p=q=1/2$. Therefore, by \eqref{eq: expansion Parry},
\begin{align*}
\overline{F}_{2/3}(z)
=\sum_{k=0}^{\varkappa_{2/3}(z)} \delta_k(z) \Big ( \frac23 \Big)^{k+1}=1-\frac{z}{3},\quad z\in[0,3]\setminus H.
\end{align*}
Since the set $H$ of exceptional points is countable, the above means that
in the symmetric case $p=q$, the distribution $\nu_{2/3}$ is uniform  on $[0,3]$. 

Let us prove that $\nu_{2/3}$ is singular when $p \neq 1/2$. For any $x \in [0,3]$ and any real $\varepsilon \neq 0$ such that $x+ \varepsilon \in [0,3]$, denote 
\[
k_x(\varepsilon)= \min \{n \ge 0: \delta_n(x) \neq \delta_n(x+\varepsilon)\}.
\]  
We need the following result, which we will prove shortly afterwards. 

\begin{lemma} \label{lem: derivative}
For almost every $x \in (0,3)$, it is true that $k_x(\varepsilon) \sim \log_{2/3} |\varepsilon|$ as $\varepsilon \to 0$.
\end{lemma}

Combined with \eqref{eq: F tail}, this result implies that for every $\delta \in (0,1)$ and almost every $x \in (0,3)$, there exist  constants $\varepsilon_\delta(x)>0$ and $C_\delta(x)>0$ such that if $0<|\varepsilon| < \varepsilon_\delta(x)$, then 
\[
|\overline{F}_{2/3}(x)- \overline{F}_{2/3}(x+\varepsilon)| \le C_\delta(x) \left(\frac{p r^C}{\lambda_{2/3}} \right)^{(1-\delta) \log_{2/3} |\varepsilon|}, 
\]
where $r =\max(1, q/r)$ and the constant $C>0$ is as in Proposition~\ref{prop: p < 1/2}. If $p \neq 1/2$, then $p r^C/\lambda_{2/3} \in (2/3, 1)$ by \eqref{eq: > lambda >}. Therefore, by choosing $\delta$ to be small enough, we see that $|\overline{F}_{2/3}(x)- \overline{F}_{2/3}(x+\varepsilon)|= o(\varepsilon)$ for almost every $x$. Thus, $\overline{F}_{2/3}'(x)=0$ for such $x$, which implies that $1-\overline{F}_{2/3}$ is a singular distribution function. It remains to prove the lemma.

\begin{proof}
It is easy to show by induction that for every $x \in [0,3]$ and non-zero $\varepsilon \in [-x, 3-x]$, 
\begin{equation} \label{eq: D T^n(x)}
T_{2/3}^n(x) - T_{2/3}^n(x+\varepsilon) = \varepsilon(3/2)^n, \qquad 0 \le n \le k_x(\varepsilon).
\end{equation}
Therefore, since $|y - y'|>2$ implies that $\delta_0(y) \neq \delta_0(y')$ for any $y, y' \in [0,3]$, it follows that $|\varepsilon| (3/2)^{k_x(\varepsilon) - 1} \le 2$. Hence 
\begin{equation} \label{eq: k(eps) <}
k_x(\varepsilon) \le \log_{2/3} |\varepsilon/3|.
\end{equation}

Let us obtain a matching lower bound. Notice that the density of the invariant measure $\mu_{2/3}$ of $T_{2/3}$, which is $\widehat h_{2/3}(x/3)/3$, is bounded by \eqref{eq: Parry}. Combined with the Borel--Cantelli lemma, this implies that for $\mu_{2/3}$-almost every $x$, there exists an $n_0(x) \ge 1$ such that
\begin{equation} \label{eq: far from 1}
|T_{2/3}^n(x) - 1| \ge 1/n^2, \qquad n \ge n_0(x). 
\end{equation}
On the other hand, for every $x \in (0,3) \setminus H$, it follows from piece-wise continuity of $T_{2/3}$ (cf.~Lemma~\ref{lem: dom T^n}.b) that $k_x(\varepsilon) \to \infty $ as $\varepsilon \to 0$. Together with \eqref{eq: far from 1}, this implies that for $\mu_{2/3}$-almost every $x$, there exists an $\varepsilon_0(x) \in (0, 2/3)$ such that for every non-zero $\varepsilon \in (-\varepsilon_0(x) , \varepsilon_0(x) )$, the following implication is true:
\[
\text{if }|T_{2/3}^n(x) - T_{2/3}^n(x+\varepsilon)| < 1/n^2 \text{ for all } 0 \le n \le k, \ \text{ then } \ k_x(\varepsilon) > k.
\]
Combined with equality~\eqref{eq: D T^n(x)}, this implies by induction that the following is true:
\[
\text{if } |\varepsilon| (3/2)^k <1/k^2, \ \text{ then } \ k_x(\varepsilon) > k.
\]
Let us take $k=\log_{2/3} |\varepsilon| + 3 \log_{2/3} \log_{2/3} |\varepsilon|$. For $0<|\varepsilon|<2/3$, we have $k >\log_{2/3} |\varepsilon| > 1$ and 
\[
|\varepsilon| (3/2)^k = \frac{1}{\log_{2/3}^3 |\varepsilon| } < \frac{1}{\log_{2/3}^2 |\varepsilon| } < \frac{1}{k^2}.
\]
Therefore, we obtain that for $\mu_{2/3}$-almost every $x$, 
\[
k_x(\varepsilon) >\log_{2/3} |\varepsilon| + 3 \log_{2/3} \log_{2/3} |\varepsilon|, \qquad 0<|\varepsilon| < \varepsilon_0(x).
\]
Combined with \eqref{eq: k(eps) <}, this finishes the proof once we recall that $\mu_{2/3}$ is equivalent to the Lebesgue measure on $[0,3]$.
\end{proof}

\subsection{Rate of convergence on test functions in BV}

Recall that $\V[f]$ denotes the total variation of a function $f$. We clam the following.

\begin{proposition} \label{prop: rate}
Let $a \in (\frac12, \frac23]$ and $p \in (0,1)$. Then there exist constants $C_3>0$ and $\gamma_1 \in(0,1)$ such that for every function $f$ of bounded variation on $[0, 1/(1-a)]$, $x \in [0, 1/(1-a)]$, and $n \ge 1$, we have 
\begin{equation} \label{eq:Pnf rate gen}
\bigg |\E_x( f (X_n)| \tau >n) - \int_{[0,\frac{1}{1-a}]} f d \nu_a \bigg |  \le C_3 \gamma_1^n \V[f].
\end{equation}
\end{proposition}
\begin{proof}
We extend the argument we gave above in Section~\ref{ssec:cond} for $f_z$ to an arbitrary function $f$ of bounded variation. To this end, we estimate the spectral radius of the operator $P_2$. First,
\begin{align*}
\|P_2f\|
&\le p\|f(ax+1)-f(1)\|
+q \big \| {\ind}\{x\ge 1/a\}(f(ax-1)-f(0)) \big \|\\
&=p\V[ {f|}_{[1, 1/(1-a)]}]
+q\V [{f|}_{[0,(2a-1)/(1-a)]}]\\
&\le \max\{p,q\}\|f\|,
\end{align*}
where the last inequality is trivial when $a \in (\frac12, \frac23)$, while in the case $a=\frac23$ we used that
\begin{equation*} 
\V [{f|}_{[0,1]}]+ \V [ {f|}_{[1, 3]}] = \V [{f|}_{[0,1)}]+ |f(1)-f(1-)| + \V [ {f|}_{[1, 3]}] = \V[f].
\end{equation*}
Therefore, $\|P_2^n\|\le p^n$ if $p\ge 1/2$, otherwise we need a more delicate estimate below. 
\begin{proposition} \label{prop: P_2^n norm}
Let $a \in (\frac12, \frac23]$, $p \in (0,1)$, and $n \in \N$. Then
\begin{equation} \label{eq: P_2^n norm}
\|P_2^n \| \le p^n \sup_{\substack{x \in [0, 1/(1-a)]: \\ \varkappa_a(x) \ge n}}(q/p)^{L_n(x)}.
\end{equation}
\end{proposition}
We postpone the proof of this estimate and first finish the proof of Proposition~\ref{prop: rate}; note that the right-hand side of \eqref{eq: P_2^n norm} equals $p^n$ when $p \ge 1/2$. 

\begin{remark} \label{rem: compact BV}
Proposition~\ref{prop: P_2^n norm} combined with the estimates of Proposition~\ref{prop: p < 1/2} imply that the spectral radius of $P_2$ on $BV$ satisfies $\rho(P_2) < \lambda_a$. This implies that the operator $P$ is quasi-compact on $BV$ for all $a \in (\frac12, \frac23]$, since $\rho(P) \ge \lambda_a$ and the operator $P_1=P-P_2$ is compact because its range is two-dimensional; cf.~Remark~\ref{rem:compact}. 
\end{remark}

Since \eqref{eq:Pnf rate gen} clearly holds true for constant functions, we can assume w.l.o.g.\ that $f(0)=0$, and thus $\|f\|=\V[f]$. 
Repeating the argument from Section~\ref{ssec:cond} and using Proposition~\ref{prop: P_2^n norm} combined with the bound  $|(M(P_1 P_2^j f), v^*)| \le \| P_2^j f \| {\|v^*\|}_\infty$ instead of \eqref{eq: P_2^j step} and \eqref{eq: P_1 P_2^j step}, we obtain the following counterpart to \eqref{eq:Pnf conv rate}:
\begin{equation}
\label{eq:Pnf conv rate gen}
\left |\lambda_a^{-n}  P^n f(x) - c \Big (\sum_{j=0}^\infty\lambda_a^{-j-1} M(P_1P_2^j f ),v^* \Big) V(x) \right| \le C_4 \gamma_1^n \V [f],
\end{equation}
where $C_4=C_2 {\|v^*\|}_\infty/p$ and the constants $C_2>0$ and $\gamma_1 \in (0,1)$ are as in  \eqref{eq:Pnf conv rate}.

The factor $(\cdot, \cdot)$ in \eqref{eq:Pnf conv rate gen} equals $\int_{[0,\frac{1}{1-a}]} f d \nu_a$, which we denote by $\nu_a(f)$. This follows from the weak convergence in \eqref{eq: weak lim} combined with the continuous mapping theorem, which applies because the limiting distribution $\nu_a$ has no atoms and every function of bounded variation has at most countable number of discontinuities. We have $|\nu_a(f)| \le {\|f\|}_\infty \le \V[f]$ by \eqref{eq: norms}, and it is easy to obtain from \eqref{eq:PF} and \eqref{eq:Pnf conv rate gen} that
\begin{align*}
\Big |\E_x( f (X_n)| \tau >n) - \nu_a(f) \Big | &= \Big | \frac{P^n f(x)}{P^n \ind (x)} - \nu_a(f) \Big |  \le (C_1+C_4)C_5^{-1} \gamma_1^n \V [f]
\end{align*}
for every  $x,z \in[0,1/(1-a)]$ and $n \ge 1$, where $C_5=\min \{ \lambda_a^{-n} P^n \ind (0): n \ge 1\}$ is strictly positive by \eqref{eq:PF} and the fact that $V(0)=1$. This proves~\eqref{eq:Pnf rate gen} with $C_3=(C_1+C_4)C_5^{-1}$.
\end{proof}
It remains to prove Proposition~\ref{prop: P_2^n norm}. To this end we need the following auxiliary result. 
\begin{lemma} \label{lem: P_2^n =}
Assume that $a \in (\frac12, \frac23]$ and $n \in \N$. Then for every $x \in [0, \tfrac{1}{1-a}] \setminus \{3\}$, we have
\begin{equation} \label{eq: P_2^n}
P_2^n f(x)= \sum_{y \in T_a^{-n}(x)} [f(y)-f(g_n(y))] p^n (q/p)^{L_n(y)}.
\end{equation}
\end{lemma}

\begin{proof}
We first show by induction that for every $x \in [0, \tfrac{1}{1-a}] \setminus \{3\}$,
\begin{equation} \label{eq: P^n}
P^n f(x)= \sum_{y \in T_a^{-n}(x)}  f(y) p^n (q/p)^{L_n(y)}.
\end{equation}
In the basis case $n=1$, this holds true by \eqref{eq: weightedPF}. The step of induction is justified by
\begin{align*} 
P^{n+1} f(x) &= \sum_{y \in T_a^{-1}(x)}  p (q/p)^{\delta_0(y)} P^n f(y)  \\ 
&= \sum_{y \in T_a^{-1}(x)} p (q/p)^{\delta_0(y)} \sum_{z \in T_a^{-n}(y)} f(z) p^n (q/p)^{L_n(z)}  \\
&= \sum_{y \in T_a^{-1}(x)} \sum_{z \in T_a^{-n}(y)} f(z) p^{n+1} (q/p)^{L_n(z) + \delta_0(T_a^n(z))}  \\
&= \sum_{z \in T_a^{-(n+1)}(x)} f(z) p^{n+1} (q/p)^{L_{n+1}(z)},
\end{align*}
where in the second equality we applied the assumption of induction using that $3 \not \in T_a^{-1}(x)$. 

Second, we claim that for every $x \in [0, \tfrac{1}{1-a}] \setminus \{3\}$, 
\begin{equation} \label{eq: P_2^n 2}
\sum_{y \in T_a^{-n}(x)} f(g_n(y)) p^n (q/p)^{L_n(y)} = \sum_{v \in G_n} \ind \{x \ge T_a^n(v)\} f(v) p^n (q/p)^{L_n(v)}.
\end{equation}
Indeed, in the case when $a < \frac23$ we have 
\begin{equation} \label{eq: P_2^n 3}
\sum_{y \in T_a^{-n}(x)} f(g_n(y)) p^n (q/p)^{L_n(y)} = \sum_{u \in D_n} \ind \{x \in T_a^n([g_n(u), u])\} f(g_n(u)) p^n (q/p)^{L_n(g_n(u))}
\end{equation}
since by Lemma~\ref{lem: dom T^n}.a, $L_n(y)$ is constant on each of the intervals $[g_n(u), u]$, whose disjoint union constitutes the domain of $T^n_a$. Then equality \eqref{eq: P_2^n 2} follows for $a< \frac23$ since $T_a^n(x)$ is increasing on $[g_n(u), u]$ and $T_a^n(u)= \frac{1}{1-a} \ge x$ for $u \in D_n$. In the case when $a=\frac23$ and $x \neq 3$, equality~\eqref{eq: P_2^n 3} remains valid if on its right-hand side for every $u \in D_n$ we replace $g_n(u)$ by $g_n(u-)$ and $[g_n(u), u]$ by $[g_n(u-), u)$. Then \eqref{eq: P_2^n 2} follows for $a=\frac23$ from this version of \eqref{eq: P_2^n 3} and the fact that  $T_{2/3}^n(u-)= 3>x$.

Furthermore, it follows from \eqref{eq: P^n} and \eqref{eq: P_2^n 2} that \eqref{eq: P_2^n} is equivalent to 
\begin{equation} \label{eq: P_2^n v2}
P^n f(x) - P_2^n f(x)=  \sum_{y \in G_n} \ind \{x \ge T_a^n(y)\}  f(y) p^n (q/p)^{L_n(y)}.
\end{equation}
We prove this equality by induction. In the basis case $n=1$ it holds true by the definition of $P_2$. Assuming that equality~\eqref{eq: P_2^n v2} is satisfied for a concrete $n$, we get
\begin{align} \label{eq: P_2^{n+1}}
P^{n+1} f(x) - P_2^{n+1} f(x)&= P_2 (P^n f - P_2^n f)(x) +  P_1 P^n f(x)\notag\\
&= P_2 \sum_{y \in G_n}  \ind \{x \ge T_a^n(y)\}  f(y) p^n (q/p)^{L_n(y)} \\
& \quad + p P^n f(1) + q P^nf(0) \ind \{x \ge T_a(0)\}.  \notag
\end{align}

For every $z \in [1, 1/(1-a)]$, denote $f_z(x)= \ind \{x \ge z\}$. Then by the definition of $P_2$,
\[
P_2 f_z(x) =p \big (\ind\{ax+1 \ge z \})-\ind\{1 \ge z \} \big)+q{\ind}\{x\ge 1/a\} \big(\ind\{ax-1 \ge z \})-\ind\{0 \ge z \}\big).
\]
It is easy to check (considering five cases) that $P_2 f_z(x)\equiv 0$ if $z \in I_a \cup \{0,1\}$, otherwise $P_2  f_z(x)=c_0(z)\ind\{x \ge T_a(z)\}$ since $x \neq 3$. Hence by \eqref{eq: P^n} and \eqref{eq: P_2^{n+1}}, 
\begin{align*} 
P^{n+1} f(x) - P_2^{n+1} f(x) &=   \sum_{\substack{y \in G_n: \\  T_a^n(y) \not \in I_a \cup \{0,1\}}}  \ind \{x \ge T_a^{n+1}(y)\}  f(y) p^n (q/p)^{L_n(y)} c_0(T_a^n(y)) \\
&\quad + \sum_{y \in T_a^{-n}(1)}  f(y) p^{n+1} (q/p)^{L_n(y)} \\
& \quad + \sum_{y \in T_a^{-n}(0)}  \ind \{x \ge T_a(0)\} f(y) q p^n (q/p)^{L_n(y)}, 
\end{align*}
and since $(q/p)^{L_n(y)} c_0(T_a^n(y))= p (q/p)^{L_{n+1}(y)}$ for every $y$ in the domain of $T_a^n$, we get
\begin{align*} 
P^{n+1} f(x) - P_2^{n+1} f(x) &=   \sum_{\substack{y \in G_n: \\  T_a^n(y) \not \in I_a \cup \{0,1\}}} \ind \{x \ge T_a^{n+1}(y)\}  f(y) p^{n+1} (q/p)^{L_{n+1}(y)} \\
&\quad  + \sum_{y \in T_a^{-n}(1)} \ind \{x \ge T_a^{n+1}(y)\} f(y) p^{n+1} (q/p)^{L_{n+1}(y)} \\
&\quad  + \sum_{y \in T_a^{-n}(0)} \ind \{x \ge T_a^{n+1}(y)\}  f(y) p^{n+1} (q/p)^{L_{n+1}(y)} \\
&= \sum_{y \in G_{n+1}} \ind \{x \ge T_a^{n+1}(y)\}  f(y) p^{n+1} (q/p)^{L_{n+1}(y)} 
\end{align*}
using that $G_{n+1}$ is a union of three disjoint sets 
\[
(G_n \cap \{z: \varkappa_a(z) >n\}) \setminus (T^{-n}(0) \cup T^{-n}(1)), \ T^{-n}(0), \text{ and }T^{-n}(1).
\]
This finishes the proof of equality~\eqref{eq: P_2^n v2}, which is equivalent to \eqref{eq: P_2^n}.
\end{proof}

\begin{proof}[Proof of Proposition~\ref{prop: P_2^n norm}]
Assume that $a \in (\frac12, \frac23)$. By Lemma~\ref{lem: dom T^n}.a and equality~\eqref{eq: P_2^n}, we have 
\begin{align*} 
P_2^n f(x)&= \sum_{y \in T_a^{-n}(x)} [f(y)-f(g_n(y))] p^n (q/p)^{L_n(y)} \\
&= \sum_{u \in D_n} \ind \{x \in T_a^n([g_n(u), u])\} \Big[f \big(({T_a^n |}_{[g_n(u), u]})^{-1}(x) \big) - f(g_n(u)) \Big] p^n (q/p)^{L_n(g_n(u))}
\end{align*}
because $L_n(y)$ is constant and $T_a^n(y)$ is bijective on each of the intervals $[g_n(u), u]$. Since $P_2^n f(0)=0$ and each of the functions under the second sum vanishes at $x=0$, 
\begin{align*} 
\| P_2^n f \| & \le \sum_{u \in D_n} \V \! \left[ \ind \{x \in T_a^n([g_n(u), u])\} \Big( f \big(({T_a^n |}_{[g_n(u), u]})^{-1}(x) \big) - f(g_n(u)) \Big) \right] \\
&\qquad \times  p^n (q/p)^{L_n(g_n(u))} \\
&= \sum_{u \in D_n} p^n (q/p)^{L_n(g_n(u))} \V \! \big[ {f|}_{[g_n(u), u]} \big] \\
&\le p^n \sup_{v \in G_n} (q/p)^{L_n(v)} \|f\|, 
\end{align*}
where the first equality holds true because $T_a^n$ is continuous and strictly increasing on each interval $[g_n(u), u]$. By Lemma~\ref{lem: dom T^n}.a, this yields estimate \eqref{eq: P_2^n norm} for $a<\frac23$.

We now assume that $a=\frac23$. Arguing as above gives 
\begin{equation} \label{eq: incomplete Var}
\V \! \big [ {(P_2^n f)|}_{[0,3)} \big ] \le \sum_{u \in D_n} p^n (q/p)^{L_n(g_n(u-))} \V \! \big[ {f|}_{[g_n(u-), u)} \big], 
\end{equation} 
as Lemma~\ref{lem: P_2^n =} does not cover the case $x=3$. 

Furthermore, it follows from \eqref{eq: P_2^n} that
\[
P_2^n f(1) -P_2^n f(1-) =\sum_{y \in T_{2/3}^{-n}(1)} [f(y)-f(y-)] p^n (q/p)^{L_n(y)}
\]
since by Lemma~\ref{lem: dom T^n}.b, the functions $g_n(x)$ and $L_n(x)$ under the sum in \eqref{eq: P_2^n} are constant in a small neighbourhood of every point in the set $T_{2/3}^{-n}(1)$, which does not meet $D_n$. Then by the definition of~$P_2$,
\begin{align*}
&\phantom{=} P_2^{n+1} f(3) -P_2^{n+1} f(3-) \\
&= p \big(P_2^n f(3) - P_2^n f(3-) \big) + q \big(P_2^n f(1) - P_2^n f(1-) \big) \\
&= p \big(P_2^n f(3) - P_2^n f(3-) \big) + \sum_{y \in T_{2/3}^{-n}(1)} [f(y)-f(y-)] q p^n (q/p)^{L_n(y)}, 
\end{align*}
and since $L_{n+1}(y-) = L_n(y) + 1$ for every $y \in T_{2/3}^{-k}(1)$, it follows that
\[
P_2^n f(3) -P_2^n f(3-) = p^n \big( f(3) - f(3-) \big)  + \sum_{k=0}^{n-1} \sum_{y \in T_{2/3}^{-k}(1)} [f(y)-f(y-)] p^n (q/p)^{L_{k+1}(y-)}.
\]
Hence, using in the case when $q >p$ that $L_k(x)$ is non-decreasing in $k$ for all fixed $x \in [0,3]$ and in the case when $q \le p$ that $L_n(3-)=0$, we obtain
\[
|P_2^n f(3) -P_2^n f(3-)| \le p^n \sup_{y \in D_n} (q/p)^{L_n(y-)} \sum_{y \in D_n} |f(y)-f(y-)|.
\]
Combined with \eqref{eq: incomplete Var}, this implies estimate \eqref{eq: P_2^n norm} for $a=\frac23$.
\end{proof}

\section{Large starting points}
In this section we prove the last remaining statements, Corollary~\ref{cor:large-x} and Proposition~\ref{prop: a < 1/2}.  The main step is to consider the case where the starting point $X_0=x$ of the chain $\{X_n\}$ is outside of the absorbing interval $[0, \frac{1}{1-a}]$. For such $x$, the chain strictly decreases until the stopping time
\[
\sigma=\inf \{n \ge 0: X_n \le 1/(1-a) \}.
\]

\begin{proof}[Proof of Corollary~\ref{cor:large-x}]
In view of Theorem~\ref{thm: main}, we only need to consider $x > \frac{1}{1-a}$. Define the stopping time
\[
\sigma'''=\inf \{n \ge 0: X_n \le (2-a)/(a(1-a)) \}.
\]
It is upper-bounded a.s.\ by a deterministic constant because $X_n < x a^n + \frac{1}{1-a}$ for all $n$. Since $\sigma=\inf\{n > \sigma''': \xi_n=-1\}$, this implies that for some constant $C(x)=C_{a,p}(x)>0$,
\begin{equation} \label{eq: sigma pmf}
\pr_x(\sigma=n)\le C(x)p^n, \qquad n \ge 1.
\end{equation}

For any $y \in [0, \frac{1}{1-a}]$, by conditioning on $\sigma$ and $X_\sigma$ and using the Markov property of the chain $\{ X_n\}$, we get
\begin{equation} \label{eq: strong Markov}
\pr_x(X_n \le y, \tau >n) = \sum_{k=1}^n \int_{[0, \frac{1}{1-a}]} \pr_z(X_{n-k} \le y, \tau >n-k) \pr_x(X_\sigma \in dz, \sigma =k)\\
\end{equation}
Notice that for the integrand, by \eqref{eq: main equiv} and \eqref{eq: weak lim} we have 
\begin{equation} \label{eq: integrand ->}
\lim_{n \to \infty} \lambda_a^{-n} \pr_z(X_{n-k} \le y, \tau >n-k) = c \lambda_a^{-k} \nu_a([0,y]) V(z) 
\end{equation}
for every fixed $k \ge 1$ and $z \in [0, \frac{1}{1-a}]$, and we also have the bound
\begin{equation*} 
\lambda_a^{-n} \pr_z(X_{{(n-k)}_+} \le y, \tau >n-k) \ind (k \le n ) \le \lambda_a^{-n} \pr_{\frac{1}{1-a}}(\tau >n-k) \le C' \lambda_a^{-k}
\end{equation*}
for some constant $C'>0$ and every integer $k, n \ge 1$ and $z \in [0, \frac{1}{1-a}]$. Since $\E_x \lambda_a^{-\sigma}<\infty$ by~\eqref{eq: sigma pmf}, by the dominated convergence theorem it follows from \eqref{eq: strong Markov} and \eqref{eq: integrand ->}  that 
\begin{align*}
\lim_{n \to \infty} \lambda_a^{-n} \pr_x(X_n \le y, \tau >n) &= \sum_{k=1}^\infty \int_{[0, \frac{1}{1-a}]} c \lambda_a^{-k} \nu_a([0,y]) V(z) \pr_x(X_\sigma \in dz, \sigma =k) \\
&= c \nu_a([0,y]) \E_x [\lambda_a^{-\sigma} V(X_\sigma)]
\end{align*}
for every fixed $x$ and $ y$. Taking $y=\frac{1}{1-a}$ gives the first claim of Corollary~\ref{cor:large-x}, and this in turn implies that  \eqref{eq: weak lim} is valid for every $x \ge 0$.
\end{proof}

\begin{proof}[Proof of Proposition~\ref{prop: a < 1/2}] 
Assume throughout that $X_0=x \ge 0$, and recall that
\[
\sigma'=\inf \{n \ge 0: X_n <1/a \}, \qquad \sigma''=\inf \{n \ge 0: X_n < 6 \}.
\]
We already used in the introduction that for any starting point $x \in [0, \frac{1}{a})$ it is true that $\tau=\inf\{n\ge1:\xi_n=-1\}$. For any $x \ge \frac1a$, we note that $\sigma'$ is upper-bounded by a deterministic constant when $a <\frac12$ because $X_n < x a^n + \frac{1}{1-a}$ and in this case $\frac{1}{1-a} < \frac1a$. Therefore, for every $x \ge 0$ when $a<\frac12$ and for every $0 \le x <2$ when $a=\frac12$ (call these two options Case 1), for all $n$ large enough  we get
\[
\pr_x( \tau >n) = \E_x p^{-\sigma'} \cdot p^n
\] 
by conditioning on $\sigma'$ and $X_{\sigma'}$. This proves \eqref{eq: main triv}.

For $a =\frac12$ and $x \ge 2$ (call this Case 2), the random variable $\sigma'$ is not bounded and it is easy to see that $\E_x p^{-\sigma'} = \infty$. However, we can write $\sigma' = \inf\{ n > \sigma'': \xi_n=-1\}$, where $\sigma'' \le C(x)$ for some deterministic integer constant $C(x)=C_a(x)$ by the same reasoning as above in Case 1. We also have $\tau = \inf\{ n > \sigma': \xi_n=-1\}$. Now use that $\tau = (\tau - \sigma') + (\sigma' - \sigma'')+ \sigma''$, where the three terms on the right-hand side are independent random variables and the first two of them are geometric with parameter $q$. Conditioning on $\sigma''$ and $X_{\sigma''}$, we get
\begin{equation*} 
\pr_x(\tau >n ) = \sum_{k=0}^{C(x)} \pr_x(\sigma''=k) \pr_x(\tau -  \sigma''>n-k) =  \sum_{k=0}^{C(x)} \pr_x(\sigma''=k) (q (n-k)+p)p^{n-k-1}.
\end{equation*}
Hence
\[
\pr_x(\tau >n ) \sim q \E_x p^{-\sigma''} \cdot n p^{n-1}
\]
as $n \to \infty$, establishing \eqref{eq: main equiv n}.

To prove the conditional weak convergence of $X_n$, 
notice that we always have $X_{\sigma'} \ge 0$. On the event $\{\sigma' \le n, \tau > n\}$, we have $X_n = a^{n-\sigma'} X_{\sigma'} + (1-a^{n-\sigma'})/(1-a)$. Therefore, for a fixed $y <1/(1-a)$ there exists an $M>0$ large enough such that 
\[
\pr_x(X_n \le y,  \tau > n) = \pr_x (X_n \le y, n-M \le \sigma' \le n, \tau > n ) \le \pr_x (\sigma' \ge n-M ).
\]
In Case 1, $\sigma'$ is bounded, hence $\pr_x (\sigma' \ge n-M )=0$ for all $n$ large enough. In Case 2, we have
\[
\pr_x (\sigma' \ge n-M ) \le \pr_x (\sigma'- \sigma'' \ge n-M - C(x)) = o(\pr_x(\tau>n))
\]
as $n \to \infty$ since $\sigma'- \sigma''$ is geometric. Thus, in either case $\pr_x(X_n \ge y|  \tau > n) \to 1$. This implies that $\pr _x(X_n \in \cdot \, | \tau >n)$ converges weakly to $\delta_{1/(1-a)}$ since we always have $X_n < x a^n + \frac{1}{1-a}$.

The $\delta$-measure at $\frac{1}{1-a}$ is quasi-stationary in the sense of \eqref{eq: quasi-st} for $a<\frac12$ since 
\[
\pr_{\frac{1}{1-a}} (X_1 \in \cdot \, | \tau>1) = p^{-1} \pr \Big ( \frac{a}{1-a} + \xi_1 \in \cdot \,, \xi_1=1 \Big) = \delta_{\frac{1}{1-a}}.
\]

It remains to argue that there is no quasi-stationary probability measure when $a=\frac12$. Suppose that $\nu$ is such a probability. Put $b=\nu(\{\frac{a}{1-a}\})$. If $b=0$, then for any $y \in [1, \frac{1}{1-a}]$, 
\[
\nu([0,y))= \pr_\nu(X_1 <y | \tau >1) = p^{-1} \pr_\nu(a X_0 + \xi_1 <y, \xi_1 =1) = \nu ([0,(y-1)/a)).
\]
Hence it follows by induction that $\nu([0,\sum_{k=0}^n a^k))=0$ for every integer $n \ge 0$. Therefore, $\nu=0$, which is a contradiction. If $b>0$, we arrive at a contradiction by
\[
b=\pr_\nu \Big ( X_1= \frac{a}{1-a} \Big | \Big . \tau >1 \Big) = \frac{bp}{\pr_\nu(\tau>1)} = \frac{bp}{p(1-b)+b} <b.
\]
\end{proof}

\section*{Acknowledgements}
We thank Denis Denisov for discussions on the problem. 



\end{document}